\documentclass{amsart}
\usepackage{amsmath}
\usepackage{amssymb}
\usepackage{amsthm}  
\usepackage{amscd}
\usepackage{graphicx,xcolor} 
\usepackage[all]{xy}
\usepackage{wrapfig}  
\usepackage{color}
\usepackage{latexsym}
\usepackage{enumerate}
\usepackage{cases}
\usepackage[top=25mm,bottom=25mm,left=30mm,right=30mm]{geometry}
\usepackage{graphics}
\usepackage{tikz} 
\usetikzlibrary{intersections, calc, patterns} 
\usepackage{extarrows} 
\usepackage{pxpgfmark} 
\usepackage{multirow} 
\usepackage{comment}

\usepackage[pagewise]{lineno}
\usetikzlibrary{intersections, calc,decorations.markings}
\numberwithin{equation}{section}
\theoremstyle{definition}
\newtheorem{example}{Example}[section]
\newtheorem{definition}[example]{Definition}

\newtheorem{remark}[example]{Remark}

\theoremstyle{plain}
\newtheorem{lemma}[example]{Lemma}
\newtheorem{theorem}[example]{Theorem} 
\newtheorem{proposition}[example]{Proposition}
\newtheorem{corollary}[example]{Corollary}

\newtheorem{defprop}[example]{Definition-Proposition}

\DeclareMathOperator{\add}{{\rm add}}
\DeclareMathOperator{\proj}{{\rm proj}}

\DeclareMathOperator{\thick}{{\rm thick}}

\DeclareMathOperator{\Hom}{{\rm Hom}}
\DeclareMathOperator{\End}{{\rm End}}

\DeclareMathOperator{\twosilt}{{\rm 2\mathchar`-silt}}
\DeclareMathOperator{\twoips}{{\rm 2\mathchar`-ips}}
\DeclareMathOperator{\twopresilt}{{\rm 2\mathchar`-presilt}}
\DeclareMathOperator{\twoscx}{{\rm 2\mathchar`-scx}}

\DeclareMathOperator{\Kb}{\mathsf{K}^{\rm b}}

\title[Complete gentle and special biserial algebras are $g$-tame]{Complete gentle and special biserial algebras are $g$-tame} 
\author{Toshitaka Aoki}
\address{T. Aoki: Graduate School of Mathematics, Nagoya University, Chikusa-ku, Nagoya, 464-8602 Japan}
\email{m15001d@math.nagoya-u.ac.jp}

\author{Toshiya Yurikusa$^{\ast}$}\thanks{$^{\ast}$Corresponding author email: toshiya.yurikusa.d8@tohoku.ac.jp}
\address{T. Yurikusa: Mathematical Institute, Tohoku University, Aoba-ku, Sendai, 980-8578, Japan}
\email{toshiya.yurikusa.d8@tohoku.ac.jp}

\begin{document}
\maketitle

\begin{abstract}
The $g$-vectors of two-term presilting complexes are important invariants.
We study a fan consisting of all $g$-vector cones for a complete gentle algebra.
We show that any complete gentle algebra is $g$-tame, by definition, the closure of a geometric realization of its fan is the entire ambient vector space. 
Our main ingredients are their surface model and their asymptotic behavior under Dehn twists.
On the other hand, it is known that any complete special biserial algebra is a factor algebra of a complete gentle algebra and the $g$-tameness is preserved under taking factor algebras.
As a consequence, we get the $g$-tameness of complete special biserial algebras.
\end{abstract}
 
\section{Introduction}
Gentle algebras, introduced in 1980's, form an important class of special biserial algebras and their representation theory has been studied by many authors (e.g. \cite{AH,AS,BR87}).
Moreover, the derived categories of gentle algebras are related to various subjects, such as discrete derived categories \cite{Vossieck01}, numerical derived invariants \cite{AG,APS}, and Fukaya categories of surfaces \cite{HKK,LP}. 

An aim of this paper is to study two-term silting theory for gentle algebras. 
In this paper, we do not assume that gentle algebras are finite dimensional. 
For our purpose, we consider the {\it complete gentle algebras} (see Definition \ref{def QD}). 
They are module-finite over $k[[t]]$ (i.e., finitely generated as an $k[[t]]$-module), 
where $k[[t]]$ is the formal power series ring in one valuable over an algebraically closed field $k$. 
In particular, finite dimensional gentle algebras are complete gentle algebras.

We discuss two-term silting theory over module-finite $k[[t]]$-algebras, see Section \ref{Secrep} (cf. \cite{Kimura}). 
For a module-finite $k[[t]]$-algebra $A$, the homotopy category $\Kb(\proj A)$ of bounded complexes of finitely generated projective right $A$-modules is Krull-Schmidt. We denote by $\twopresilt A$ (resp., $\twosilt A$) the set of isomorphism classes of basic two-term presilting (resp., silting) complexes for $A$. 
Each $T\in \twopresilt A$ has a numerical invariant $g(T)\in\mathbb{Z}^n$, called the $g$-vector of $T$, where $n$ is the number of non-isomorphic indecomposable projective $A$-modules.
Then one can define a cone in $\mathbb{R}^n$, called the $g$-vector cone of $T$, by  
\[
 C(T) := \biggl\{\sum_{X} a_X g(X) \mid a_X \in \mathbb{R}_{\ge0}\biggr\},
\]
where $X$ runs over all indecomposable direct summands of $T$. 
We denote by $\mathcal{F}(A)$ a collection of $g$-vector cones of all basic two-term presilting complexes for $A$, by $\vert \mathcal{F}(A)\vert $ its geometric realization. 
It follows from \cite{Kimura} (see also Proposition \ref{prop fan}) that $\mathcal{F}(A)$ is a non-singular fan (i.e., each maximal cone is generated by a $\mathbb{Z}$-basis for $\mathbb{Z}^n$) and its maximal faces correspond to basic two-term silting complexes for $A$. 
Namely, 
$$ 
\vert\mathbf{\mathcal{F}}(A)\vert  = \bigcup_{C \in \mathcal{F}(A)} C = \bigcup_{T\in \twosilt A} C(T). 
$$
Such a fan plays an important role in the study of stability scattering diagrams and their wall-chamber structures (see e.g. \cite{Asai21,Bridgeland17,BST,Yurikusa18}). 
The following result is well-known.

\begin{theorem} \cite{Asai21,DIJ,ZZ} \label{def:g-finite}   
   Let $A$ be a finite dimensional $k$-algebra. Then the following conditions are equivalent: 
\begin{enumerate}
   \item $\twosilt A$ is finite; 
   \item $\vert \mathcal{F}(A)\vert  =\mathbb{R}^n$. 
\end{enumerate}
\end{theorem}

This result naturally leads the following definition in a general setting. 

\begin{definition} \label{def:g-tame}
Let $A$ be a module-finite $k[[t]]$-algebra. We say that $A$ is {\it $g$-tame} if it satisfies 
\begin{equation*}
   \overline{\vert \mathcal{F}(A)\vert } =\mathbb{R}^n,
\end{equation*} 
where $\overline{(-)}$ is the closure with respect to the natural topology on $\mathbb{R}^n$.
\end{definition}

This means that $g$-vector cones are dense in the stability scattering diagram \cite{Bridgeland17} for a finite dimensional $g$-tame algebra. Note that a similar notion, called $\tau$-tilting tame, was given in \cite{BST}.

The $g$-tameness is known for path algebras of extended Dynkin quivers \cite{Hille06}, for complete preprojective algebras of extended Dynkin graphs \cite{KM}, and for Jacobian algebras associated with triangulated surfaces \cite{Yurikusa20}. We prove the $g$-tameness of a new class.

\begin{theorem}\label{comp SB gtame}
Any complete special biserial algebra is $g$-tame. 
\end{theorem}

To prove Theorem 1.3, it suffices to prove that complete gentle algebras are $g$-tame. In fact, any complete special biserial algebra is a factor algebra of a complete gentle algebra (Proposition \ref{SB-gentle}), and $g$-tameness is preserved under taking factor algebras (Corollary \ref{factor-g-tame}). 

To prove the $g$-tameness of complete gentle algebras, their surface model plays a central role (see \cite{APS,OPS,PPP}). A similar construction has been developed in several areas, such as \cite{AAC,KS,OPS}. For each dissection $D$ of a $\circ\bullet$-marked surface $(S,M)$, one can define a complete gentle algebra $A(D)$. 
Conversely, any complete gentle algebra arises in this way (see Sections \ref{secms} and \ref{secgentle} for the details). 
Note that the cardinality $n$ of $D$ is completely determined by $(S,M)$ (Remark \ref{invariant}). 

For a given dissection $D$ of $(S,M)$, we observe a certain class of non-self-intersecting curves of $S$, called $D$-laminates, and finite multi-sets of pairwise non-intersecting $D$-laminates, called $D$-laminations. 
Notice that we take account of closed curves here. 
To each $D$-laminate $\gamma$, we associate an integer vector $g(\gamma) \in \mathbb{Z}^{n}$, called the $g$-vector, whose entries are intersection numbers of $\gamma$ and $d\in D$. 
Our $g$-vector is an analog of shear coordinates on triangulated surfaces
\cite{FT}. From this point of view, we give the next result as an analog
of \cite[Theorems 12.3, 13.6]{FT}. It is also a generalization of
\cite[Proposition 6.14]{PPP} to an arbitrary dissection.

\begin{theorem}(Theorem \ref{lattice points}) \label{introbij} 
The map $\mathcal{X}\mapsto \sum_{\gamma\in \mathcal{X}}g(\gamma)$ gives a bijection between the set of $D$-laminations and $\mathbb{Z}^n$. 
\end{theorem}

We especially consider certain $D$-laminations.
A $D$-lamination $\mathcal{X}$ is said to be reduced if it consists of pairwise distinct non-closed $D$-laminates up to isotopy, and complete if it is reduced and maximal as a set. 
We denote by $\mathcal{F}(D)$ a collection of $C(\mathcal{X})$ of all reduced $D$-laminations $\mathcal{X}$, where $C(\mathcal{X})$ is a cone in $\mathbb{R}^n$ spanned by $g(\gamma)$ for all $\gamma \in \mathcal{X}$.
We prove that 
\begin{equation*}
    \vert \mathcal{F}(D)\vert  := \bigcup_{C\in \mathcal{F}(D)} C
\end{equation*}
is dense in $\mathbb{R}^n$. Namely,

\begin{theorem} \label{dense}
   For a dissection $D$ of a $\circ\bullet$-marked surface $(S,M)$, we have 
   \[  
   \overline{\vert \mathcal{F}(D)\vert } 
   = \mathbb{R}^{n}.
   \]
\end{theorem}

On the other hand, we show in Section \ref{Secrep} that the surface model realizes a fan of $g$-vector cones for a complete gentle algebra $A(D)$ of $D$. It completes a proof of Theorem \ref{comp SB gtame}. 

\begin{theorem}[Theorem \ref{D-twosilt}]\label{introtwosilt}
   Let $D$ be a dissection of a $\circ\bullet$-marked surface $(S,M)$ and $A(D)$ the complete gentle algebra associated with $D$. Then there are bijections 
   \[
   T_{(-)}\colon \{\text{reduced $D$-laminations}\} \to \twopresilt A(D), \ \text{and}
   \]
   \[
   \{\text{complete $D$-laminations}\} \to \twosilt A(D)
   \] 
   such that $C(\mathcal{X}) = C(T_{\mathcal{X}})$. 
   In particular, we have $\mathcal{F}(A(D))=\mathcal{F}(D)$. 
\end{theorem}

It is shown via representation theory that $\mathcal{F}(A(D))=\mathcal{F}(D)$ is a (not necessarily finite) non-singular fan (Corollary \ref{FD is non-singular}).

A main ingredient of our proof of Theorem \ref{dense} is the asymptotic behavior of $g$-vectors under Dehn twists. This proof is inspired by the proof of \cite[Theorem 1.5]{Yurikusa20}. Recently, Plamondon and the second author \cite{PY} showed that finite
dimensional tame algebras are $g$-tame. Although finite dimensional special biserial algebras are tame, our result includes infinite dimensional case. Moreover, since our method is based on a geometric model of complete gentle algebras $A(D)$, it gives much more information about $g$-vectors, e.g.\,in the forthcoming paper \cite{Aoki}, it plays a key role for analyzing the polytope associated with the fan $\mathcal{F}(A(D))$.

This paper is organized as follows.
Through to Section \ref{Secproofdense}, we study the geometric and combinatorial aspects of our results.
In Section \ref{Secprel}, we recall the notions and results of \cite{APS,PPP} in terms of our notations. 
Before proving our results, we give some examples in Section \ref{Secexample}.
By using the examples, we prove Theorem \ref{introbij} in Section \ref{Secg}.
In Sections \ref{SecDehn} and \ref{Secproofdense}, we study $g$-vectors of $D$-laminates and their asymptotic behavior under Dehn twists, and prove Theorem \ref{dense}.

In Section \ref{Secrep}, we study the algebraic aspects of our results.
First, we recall two-term silting theory over module-finite algebras, in particular, they include complete gentle algebras.
Second, we give a geometric model of two-term silting theory over complete gentle algebras, and prove Theorem \ref{introtwosilt}.
Finally, we prove Theorem \ref{comp SB gtame} and also give a relation with a special class of special biserial algebras containing Brauer graph algebras (see Section \ref{appfindim}).
These examples are given in Section \ref{Sec example rep}.

\section{Preliminary} \label{Secprel}
In this section, we recall the notions and results of \cite{APS,PPP} (see also \cite{OPS}). 
Our notations are slightly different from theirs for the convenience of our purpose.

\subsection{$\circ\bullet$-marked surfaces}\label{secms}

\begin{definition}
 A {\it $\circ\bullet$-marked surface} is the pair $(S,M)$ consisting of the following data:
\begin{enumerate}
 \item[(a)] $S$ is a connected compact oriented Riemann surface with (possibly empty) boundary $\partial S$.
 \item[(b)] $M=M_{\circ} \sqcup M_{\bullet}$ is a non-empty finite set of marked points on $S$ such that
\begin{enumerate}
 \item[$\bullet$] both $M_{\circ}$ and $M_{\bullet}$ are not empty;
 \item[$\bullet$] each component of $\partial S$ has at least one marked point;
 \item[$\bullet$] the points of $M_{\circ}$ and $M_{\bullet}$ alternate on each boundary component.
\end{enumerate}
\end{enumerate}
 Any marked point in the interior of $S$ is called a {\it puncture}.
\end{definition}

 Let $(S,M)$ be a $\circ\bullet$-marked surface. Throughout this paper, when we consider intersections of curves, we assume that they intersect transversally in a minimum number of points in $S \setminus M$.

\begin{definition}
$(1)$ A {\it $\circ$-arc} (resp., $\bullet$-arc) $\gamma$ of $(S,M)$ is a curve in $S$ with endpoints in $M_{\circ}$ (resp., $M_{\bullet}$), considered up to isotopy, such that the following conditions are satisfied (see Figure \ref{Fig nonarc}):
\begin{enumerate}
 \item[$\bullet$] $\gamma$ does not intersect itself except at its endpoints;
 \item[$\bullet$] $\gamma$ is disjoint from $M$ and $\partial S$ except at its endpoints;
 \item[$\bullet$] $\gamma$ does not cut out a monogon without punctures.
\end{enumerate}

$(2)$ A {\it $\circ$-dissection} (resp., {\it $\bullet$-dissection}) is a maximal set of pairwise non-intersecting $\circ$-arcs (resp., $\bullet$-arcs) on $(S,M)$ which does not cut out a subsurface without marked points in $M_{\bullet}$ (resp., $M_{\circ}$).
\end{definition}

\begin{figure}[htp]
\begin{center}
\begin{tikzpicture}
\coordinate(ru)at(30:1); \coordinate(u)at(90:1); \coordinate(lu)at(150:1);
\coordinate(rd)at(-30:1); \coordinate(d)at(-90:1); \coordinate(ld)at(-150:1);
\draw(0,0)circle(10mm);
\draw(lu)--(70:0.4) (70:0.4)arc(70:-250:0.4) (-250:0.4)--(ru);
\fill(u)circle(1mm); \fill(ld)circle(1mm); \fill(rd)circle(1mm); \fill(0,0)circle(1mm);
\draw[fill=white](ru)circle(1mm); \draw[fill=white](lu)circle(1mm); \draw[fill=white](d)circle(1mm);
\end{tikzpicture}
 \hspace{10mm}
\begin{tikzpicture}
\coordinate(ru)at(30:1); \coordinate(u)at(90:1); \coordinate(lu)at(150:1);
\coordinate(rd)at(-30:1); \coordinate(d)at(-90:1); \coordinate(ld)at(-150:1);
\draw(0,0)circle(10mm);
\draw(lu)..controls(-0.5,0.1)and(-0.3,0)..(0,0);
\draw(ru)..controls(0.5,0.1)and(0.3,0)..(0,0);
\fill(u)circle(1mm); \fill(ld)circle(1mm); \fill(rd)circle(1mm); \fill(0,0)circle(1mm);
\draw[fill=white](ru)circle(1mm); \draw[fill=white](lu)circle(1mm); \draw[fill=white](d)circle(1mm);
\end{tikzpicture}
 \hspace{10mm}
\begin{tikzpicture}
\coordinate(ru)at(30:1); \coordinate(u)at(90:1); \coordinate(lu)at(150:1);
\coordinate(rd)at(-30:1); \coordinate(d)at(-90:1); \coordinate(ld)at(-150:1);
\draw(0,0)circle(10mm);
\draw(lu)..controls(90:0.6)and(130:0.2)..(150:0.2);
\draw(lu)..controls(210:0.6)and(170:0.2)..(150:0.2);
\fill(u)circle(1mm); \fill(ld)circle(1mm); \fill(rd)circle(1mm); \fill(0,0)circle(1mm);
\draw[fill=white](ru)circle(1mm); \draw[fill=white](lu)circle(1mm); \draw[fill=white](d)circle(1mm);
\end{tikzpicture}
\end{center}
   \caption{Examples of a curve which is not a $\circ$-arc}
   \label{Fig nonarc}
\end{figure}

\begin{remark} \label{invariant}
 Let $g$ be the genus of $S$, $b$ be the number of boundary components and $p_{\circ}$ (resp., $p_{\bullet}$) be the number of punctures in $M_{\circ}$ (resp., $M_{\bullet}$). By \cite[Proposition 1.11]{APS}, a $\circ$-dissection (resp., $\bullet$-dissection) of $(S,M)$ consists of $\vert M_{\circ}\vert +p_{\bullet}+b+2g-2=\vert M_{\bullet}\vert +p_{\circ}+b+2g-2$ $\circ$-arcs (resp., $\bullet$-arcs).
\end{remark}

By symmetry, the claims in this paper hold if we permute the symbols $\circ$ and $\bullet$. Thus we state only one side of each claim. 
A dissection divides $(S,M)$ into polygons with exactly one marked point.

\begin{proposition}\cite[Proposition 1.12]{APS}\label{poly}
 For a $\bullet$-dissection $D$ of $(S,M)$, each connected component of $S \setminus D$ is homeomorphic to one of the following:
\begin{enumerate}
 \item[$\bullet$] an open disk with precisely one marked point in $M_{\circ} \cap \partial S$;
 \item[$\bullet$] an open disk with precisely one marked point in $M_{\circ}$, but not in $\partial S$.
\end{enumerate}
\end{proposition}

For a $\bullet$-dissection $D$ of $(S,M)$, the closure of a connected component of $S \setminus D$ is called a {\it polygon of $D$}. Proposition \ref{poly} implies that any polygon of $D$ has exactly one marked point in $M_{\circ}$. We denote by $\triangle_{v}$ the polygon with marked point $v \in M_{\circ}$ (see Figure \ref{Fig polygon}).

\begin{figure}[htp]
\begin{center}
\begin{tikzpicture}[baseline=0mm,scale=1]
\coordinate(l)at(180:1); \coordinate(lu)at(120:1); \coordinate(ru)at(60:1);
\coordinate(r)at(0:1); \coordinate(ld)at(-120:1); \coordinate(rd)at(-60:1);
\draw(l)--(lu)--(ru)--(r)--(rd)--(ld)--(l);
\fill(l)circle(1mm); \fill(lu)circle(1mm); \fill(ru)circle(1mm);
\fill(r)circle(1mm); \fill(ld)circle(1mm); \fill(rd)circle(1mm);
\draw(0,0)circle(1mm);
\node at(0.2,0.2){$v$};
\end{tikzpicture}
 \hspace{20mm}
\begin{tikzpicture}[baseline=0mm,scale=1]
\coordinate(l)at(180:1); \coordinate(lu)at(120:1); \coordinate(ru)at(60:1);
\coordinate(r)at(0:1); \coordinate(ld)at(-120:1); \coordinate(rd)at(-60:1);
\draw(l)--(lu)--(ru)--(r)--(rd)--(ld)--(l);
\fill(l)circle(1mm); \fill(lu)circle(1mm); \fill(ru)circle(1mm);
\fill(r)circle(1mm); \fill(ld)circle(1mm); \fill(rd)circle(1mm);
\fill[pattern=north east lines](ld)--(-0.8,-1.2)--(0.8,-1.2)--(rd);
\draw[fill=white]($(ld)!0.5!(rd)$)circle(1mm);
\node at(0.2,-0.7){$v$};
\end{tikzpicture}
\end{center}
   \caption{Polygon $\triangle_v$ for a marked point $v \in M_{\circ}$}
   \label{Fig polygon}
\end{figure}

\begin{defprop}\cite[Proposition 3.6]{PPP}\label{dual}
 For a $\bullet$-dissection $D$ of $(S,M)$, there is a unique $\circ$-dissection $D^{\ast}$ whose each $\circ$-arc intersects exactly one $\bullet$-arc of $D$. We have $D^{\ast\ast}=D$. 
 We call $D^{\ast}$ the {\it dual dissection} of $D$. For $d \in D$, we write the corresponding $\circ$-arc by $d^{\ast} \in D^{\ast}$.
\end{defprop}

\subsection{$g$-vectors of $D$-laminates and $D$-laminations}

We fix a $\bullet$-dissection $D$ of $(S,M)$. 

\begin{definition}\label{$D$-laminate}
 $(1)$ A {\it $\circ$-laminate} of $(S,M)$ is a curve $\gamma$ in $S$, considered up to isotopy relative to $M$, that is either
\begin{enumerate}
 \item[$\bullet$] a closed curve, or
 \item[$\bullet$] a curve whose ends are unmarked points on $\partial S$ or spirals around punctures in $M_{\circ}$ either clockwise or counterclockwise (see Figure \ref{Fig laminate}).
\end{enumerate}

 $(2)$ 
 A {\it $D$-laminate} is a non-self-intersecting $\circ$-laminate $\gamma$ of $(S,M)$ intersecting at least one $\bullet$-arc of $D$ such that the following condition is satisfied:
\begin{center}
 $(\ast)$ Whenever $\gamma$ intersects $d \in D$, the endpoints $v$ and $v'$ of $d^{\ast}$ lie on opposite sides of $\gamma$ in $\triangle_{v} \cup \triangle_{v'}$.
\end{center}
 Here, we consider that the point $v$ lies on the right (resp., left) to $\gamma$ if an end of $\gamma$ is a spiral around $v$ clockwise (resp., counterclockwise) and we take its orientation as approaching $v$ (see Figures \ref{Fig Seg of D-lam} and \ref{Fig Seg of nonD-lam}).
\end{definition}

\begin{figure}[htp]
\begin{center}
\begin{tikzpicture}[baseline=0mm]
\coordinate(u)at(90:0.5); \coordinate(l)at(0:0.5); \coordinate(ll)at(180:0.5); \coordinate(d)at(-90:0.5);
\coordinate(r)at(2.5,0);
\draw(l)..controls(1.2,0)and(1.8,0.3)..(2.5,0.3);
\draw(2.5,0.3)arc(90:-90:2.6mm); \draw(2.5,-0.22)arc(-90:-270:2mm);
\draw[pattern=north east lines] (u) arc (90:-90:5mm); \draw[dotted] (u) arc (90:270:5mm);
\draw[fill=white](r)circle(1mm);
\draw[fill=white](40:0.5)circle(1mm); \fill(-40:0.5)circle(1mm);
\end{tikzpicture}
\hspace{10mm}
\begin{tikzpicture}[baseline=0mm]
\coordinate(u)at(90:0.5); \coordinate(l)at(0,0); \coordinate(d)at(-90:0.5);
\coordinate(r)at(2.5,0);
\draw(0,0.3)arc(90:270:2.6mm); \draw(0,-0.22)arc(-90:90:2mm);
\draw(0,0.3)--(2.5,0.3);
\draw(2.5,0.3)arc(90:-90:2.6mm); \draw(2.5,-0.22)arc(-90:-270:2mm);
\draw[fill=white](r)circle(1mm);
\draw[fill=white](l)circle(1mm);
\node at(1.25,0){$\gamma$}; \node at(0,-0.4){$u$}; \node at(2.5,-0.4){$v$};
\end{tikzpicture}
\end{center}
   \caption{Examples of a $\circ$-laminate, where an end of $\gamma$ is a spiral around $u$ counterclockwise and the other is a spiral around $v$ clockwise}
   \label{Fig laminate}
\end{figure}

\begin{figure}[htp]
\begin{center}
\begin{tikzpicture}[baseline=0mm]
\coordinate(u)at(0,0.8); \coordinate(d)at(0,-0.8);
\coordinate(llu)at(-1.5,1); \coordinate(lu)at(-0.6,1); \coordinate(rru)at(1.5,1); \coordinate(ru)at(0.6,1);
\coordinate(lld)at(-1.5,-1); \coordinate(ld)at(-0.6,-1); \coordinate(rrd)at(1.5,-1); \coordinate(rd)at(0.6,-1);
\coordinate(lllu)at(-2,0.5); \coordinate(llld)at(-2,-0.5); \coordinate(rrru)at(2,0.5); \coordinate(rrrd)at(2,-0.5);
\draw(llu)--(lu)--(u)--(ru)--(rru)--(rrru)--(rrrd)--(rrd)--(rd)--(d)--(ld)--(lld)--(llld)--(lllu)--(llu) (u)--(d);
\draw(2,0)..controls(1,0.5)and(0.5,0)..(-1.75,-0.75);
\draw(1.05,1)..controls(1,0.5)and(-0.5,-0.3)..(-1.2,-0.3);
\draw(-1.2,-0.3)arc(270:90:2.6mm); \draw(-1.2,0.22)arc(90:-90:2mm);
\fill(lu)circle(1mm); \fill(ru)circle(1mm); \fill(ld)circle(1mm); \fill(rd)circle(1mm);
\fill(u)circle(1mm); \fill(d)circle(1mm);
\fill(llu)circle(1mm); \fill(lld)circle(1mm); \fill(rru)circle(1mm); \fill(rrd)circle(1mm);
\fill(lllu)circle(1mm); \fill(llld)circle(1mm); \fill(rrru)circle(1mm); \fill(rrrd)circle(1mm);
\draw[fill=white](-1.2,0)circle(1mm); \draw[fill=white](1.2,0)circle(1mm);
\node at(-1.2,0.4){$v$}; \node at(1.2,-0.3){$v'$};
\end{tikzpicture}
\end{center}
   \caption{An example of segments of a $D$-laminate in $\triangle_{v} \cup \triangle_{v'}$}
   \label{Fig Seg of D-lam}
\end{figure}
\begin{figure}[htp]
\begin{center}
\begin{tikzpicture}[baseline=0mm]
\coordinate(u)at(0,0.8); \coordinate(d)at(0,-0.8);
\coordinate(llu)at(-1.5,1); \coordinate(lu)at(-0.6,1); \coordinate(rru)at(1.5,1); \coordinate(ru)at(0.6,1);
\coordinate(lld)at(-1.5,-1); \coordinate(ld)at(-0.6,-1); \coordinate(rrd)at(1.5,-1); \coordinate(rd)at(0.6,-1);
\coordinate(lllu)at(-2,0.5); \coordinate(llld)at(-2,-0.5); \coordinate(rrru)at(2,0.5); \coordinate(rrrd)at(2,-0.5);
\draw(llu)--(lu)--(u)--(ru)--(rru)--(rrru)--(rrrd)--(rrd)--(rd)--(d)--(ld)--(lld)--(llld)--(lllu)--(llu) (u)--(d);
\draw(2,0)..controls(1,-0.5)and(0,0)..(-1.75,-0.75);
\draw(1.05,1)..controls(1,0.5)and(-0.5,0.3)..(-1.2,0.3);
\draw(-1.2,0.3)arc(-270:-90:2.6mm); \draw(-1.2,-0.22)arc(-90:90:2mm);
\fill(lu)circle(1mm); \fill(ru)circle(1mm); \fill(ld)circle(1mm); \fill(rd)circle(1mm);
\fill(u)circle(1mm); \fill(d)circle(1mm);
\fill(llu)circle(1mm); \fill(lld)circle(1mm); \fill(rru)circle(1mm); \fill(rrd)circle(1mm);
\fill(lllu)circle(1mm); \fill(llld)circle(1mm); \fill(rrru)circle(1mm); \fill(rrrd)circle(1mm);
\draw[fill=white](-1.2,0)circle(1mm); \draw[fill=white](1.2,0)circle(1mm);
\node at(-1.2,0.45){$v$}; \node at(1.2,0.3){$v'$};
\end{tikzpicture}
\end{center}
   \caption{An example of segments of a $\circ$-laminate in $\triangle_{v} \cup \triangle_{v'}$, neither of which satisfy the condition $(\ast)$}
   \label{Fig Seg of nonD-lam}
\end{figure}

 A $D$-laminate is called a {\it closed D-laminate} if it is a closed curve. Remark that non-closed $D$-laminates coincide with $D$-slaloms in \cite{PPP}. 
Now, we treat a certain collection of $D$-laminates, which is central in this paper. 

\begin{definition}
 We say that two $D$-laminates are {\it compatible} if they do not intersect. A finite multi-set of pairwise compatible $D$-laminates is called a {\it $D$-lamination}. 
A $D$-lamination is said to be 
\begin{enumerate}
\item[$\bullet$] {\it reduced} if it consists of pairwise distinct non-closed $D$-laminates, and  
\item[$\bullet$] {\it complete} if it is reduced and is maximal as a set.
\end{enumerate}
\end{definition}

 Let $\gamma$ be a $D$-laminate. Using the notations in the condition $(\ast)$, let $p$ be an intersection point of $\gamma$ and $d$ such that $\gamma$ leaves $\triangle_{v}$ to enter $\triangle_{v'}$ via $p$. Then $p$ is said to be {\it positive} (resp., {\it negative}) if $v$ is to its right (resp., left), or equivalently, $v'$ is to its left (resp., right). See Figure \ref{Fig pos neg}. For $d \in D$, we define an integer 
{\setlength\arraycolsep{0.5mm}
\begin{eqnarray}\label{def g-vect} 
 g(\gamma)_d &:=& \#\{\text{positive intersection points of $\gamma$ and $d$}\}\\
 &&- \#\{\text{negative intersection points of $\gamma$ and $d$}\}.\nonumber
\end{eqnarray}}
\noindent The {\it $g$-vector} $g(\gamma)$ of $\gamma$ is given by $\bigl(g(\gamma)_d\bigr)_{d \in D} \in \mathbb{Z}^{\vert D\vert }$, where $\vert D\vert $ is the number of $\bullet$-arcs of $D$.
We remark that if $\gamma$ and $d$ intersect twice, then their intersection points are either positive or negative simultaneously. 
Thus, the absolute value of $g(\gamma)_d$ just counts the number of intersection points of $\gamma$ and $d$.

\begin{figure}[htp]
\begin{center}
\begin{tikzpicture}[baseline=0mm]
\coordinate(u)at(0,0.8); \coordinate(d)at(0,-0.8);
\coordinate(llu)at(-1.5,1); \coordinate(lu)at(-0.6,1); \coordinate(rru)at(1.5,1); \coordinate(ru)at(0.6,1);
\coordinate(lld)at(-1.5,-1); \coordinate(ld)at(-0.6,-1); \coordinate(rrd)at(1.5,-1); \coordinate(rd)at(0.6,-1);
\draw(llu)--(lu)--(u)--(ru)--(rru) (lld)--(ld)--(d)--(rd)--(rrd) (u)--(d);
\draw(-1.5,0.5)..controls(-0.5,0.5)and(0.5,-0.5)..(1.5,-0.5);
\fill(lu)circle(1mm); \fill(ru)circle(1mm); \fill(ld)circle(1mm); \fill(rd)circle(1mm);
\fill(u)circle(1mm); \fill(d)circle(1mm);
\draw[fill=white](-1.4,0)circle(1mm); \draw[fill=white](1.4,0)circle(1mm);
\node at(0.2,0.15){$p$};
\end{tikzpicture}
 \hspace{10mm}
\begin{tikzpicture}[baseline=0mm]
\coordinate(u)at(0,0.8); \coordinate(d)at(0,-0.8);
\coordinate(llu)at(-1.5,1); \coordinate(lu)at(-0.6,1); \coordinate(rru)at(1.5,1); \coordinate(ru)at(0.6,1);
\coordinate(lld)at(-1.5,-1); \coordinate(ld)at(-0.6,-1); \coordinate(rrd)at(1.5,-1); \coordinate(rd)at(0.6,-1);
\draw(llu)--(lu)--(u)--(ru)--(rru) (lld)--(ld)--(d)--(rd)--(rrd) (u)--(d);
\draw(-1.5,-0.5)..controls(-0.5,-0.5)and(0.5,0.5)..(1.5,0.5);
\fill(lu)circle(1mm); \fill(ru)circle(1mm); \fill(ld)circle(1mm); \fill(rd)circle(1mm);
\fill(u)circle(1mm); \fill(d)circle(1mm);
\draw[fill=white](-1.4,0)circle(1mm); \draw[fill=white](1.4,0)circle(1mm);
\node at(-0.2,0.15){$q$};
\end{tikzpicture}
\end{center}
   \caption{A positive intersection point $p$ and a negative intersection point $q$}
   \label{Fig pos neg}
\end{figure}

\begin{definition}
For a $D$-lamination $\mathcal{X}$, let $g(\mathcal{X}):=\sum_{\gamma\in \mathcal{X}}g(\gamma) \in \mathbb{Z}^{\vert D\vert }$, where $g(\emptyset):=0$. 
We call $g(\mathcal{X})$ the \emph{$g$-vector} of $\mathcal{X}$. 
In addition, let $C(\mathcal{X})$ be a cone in $\mathbb{R}^{\vert D\vert }$ spanned by $g(\gamma)$ for all $\gamma\in \mathcal{X}$, that is,  
\[
 C(\mathcal{X}) := \biggl\{\sum_{\gamma\in \mathcal{X}}a_{\gamma}g(\gamma)\mid a_{\gamma}\in \mathbb{R}_{\ge0}\biggl\} \subseteq \mathbb{R}^{\vert D\vert }.
\]
We call $C(\mathcal{X})$ the \emph{$g$-vector cone} of $\mathcal{X}$. We denote by $\mathcal{F}(D)$ a collection of $g$-vector cones of all reduced $D$-laminations. Let 
\[
 \vert \mathcal{F}(D)\vert  := \bigcup_{\mathcal{X}} C(\mathcal{X}) \subseteq \mathbb{R}^{\vert D\vert },
\]
where $\mathcal{X}$ runs over all reduced $D$-laminations.
\end{definition}

The invariants, $g$-vectors and $g$-vector cones, have good properties.

\begin{theorem} \cite[Theorems 5.12 and 6.12]{PPP} \label{g base}
   \begin{enumerate} 
   \item Every $g$-vector cone $C(\mathcal{X})$ of a complete $D$-lamination $\mathcal{X}$ is non-singular, that is, the set $\{g(\gamma)\mid \gamma\in \mathcal{X}\}$ forms a free basis of $\mathbb{Z}^{\vert D\vert }$. 
   \item A reduced $D$-lamination is complete if and only if it has precisely $\vert D\vert $ elements. 
   \end{enumerate}
\end{theorem}

\begin{theorem}\cite[Theorem 6.14]{PPP}\label{comp simp fan}
If $\mathcal{F}(D)$ is finite, then all $D$-laminates are non-closed.
In this case, $\mathcal{F}(D)$ is a non-singular fan whose maximal faces correspond to complete $D$-laminations and such that $\vert \mathcal{F}(D)\vert =\mathbb{R}^{\vert D\vert }$.
\end{theorem}
In Corollary \ref{FD is non-singular}, it is shown via representation theory that $\mathcal{F}(D)$ is a non-singular fan even when it is not finite.

\section{Examples} \label{Secexample}
In this section, we examine our notions defined in the previous section. Here, for each integer $1\leq i\leq n$, let $e_i$ be an integer vector $(x_1,\ldots,x_n)\in \mathbb{Z}^n$ such that $x_i=1$ and $x_j=0$ for all $j\neq i$.

(1) 
Let $(S_1(3),M_1(3))$ be a disk with $\vert M_1(3)\vert =8$ such that all marked points lie on $\partial S_1(3)$. For a $\bullet$-dissection of $(S_1(3),M_1(3))$
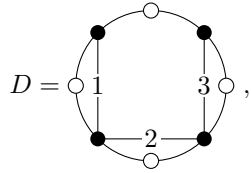
\[
D=D_1(3)=
\begin{tikzpicture}[baseline=-1mm]
\coordinate(r)at(0:1); \coordinate(ru)at(45:1); \coordinate(u)at(90:1); \coordinate(lu)at(135:1);
\coordinate(l)at(180:1); \coordinate(rd)at(-45:1); \coordinate(d)at(-90:1); \coordinate(ld)at(-135:1);
\draw(0,0)circle(10mm);
\draw(lu)--node[fill=white,inner sep=1]{$1$}(ld)--node[fill=white,inner sep=1]{$2$}(rd)--node[fill=white,inner sep=1]{$3$}(ru);
\fill(lu)circle(1mm); \fill(ru)circle(1mm); \fill(ld)circle(1mm); \fill(rd)circle(1mm);
\draw[fill=white](r)circle(1mm); \draw[fill=white](u)circle(1mm); 
\draw[fill=white](l)circle(1mm); \draw[fill=white](d)circle(1mm);
\end{tikzpicture}\ ,
\]
all $D$-laminates and the corresponding $g$-vectors are given as follows:
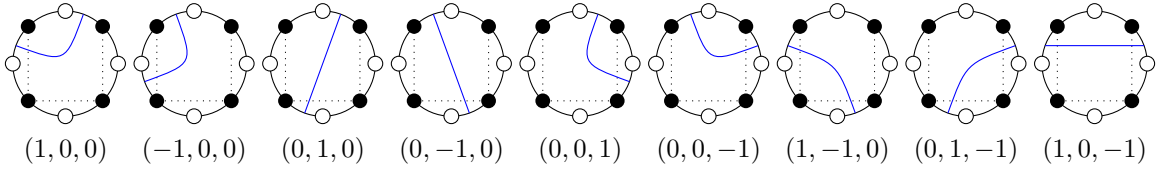
\[
\renewcommand{\arraystretch}{1.5}
\setlength{\tabcolsep}{0.5mm}
\begin{tabular}{ccccccccc}
\begin{tikzpicture}[baseline=0mm,scale=0.7]
\coordinate(r)at(0:1); \coordinate(ru)at(45:1); \coordinate(u)at(90:1); \coordinate(lu)at(135:1);
\coordinate(l)at(180:1); \coordinate(rd)at(-45:1); \coordinate(d)at(-90:1); \coordinate(ld)at(-135:1);
\draw(0,0)circle(10mm); \draw[dotted](lu)--(ld)--(rd)--(ru);
\fill(lu)circle(1.4mm); \fill(ru)circle(1.4mm); \fill(ld)circle(1.4mm); \fill(rd)circle(1.4mm);
\draw[fill=white](r)circle(1.4mm); \draw[fill=white](u)circle(1.4mm); 
\draw[fill=white](l)circle(1.4mm); \draw[fill=white](d)circle(1.4mm);
\draw[blue](70:1)..controls(0,0)..(160:1);
\end{tikzpicture}
 &
\begin{tikzpicture}[baseline=0mm,scale=0.7]
\coordinate(r)at(0:1); \coordinate(ru)at(45:1); \coordinate(u)at(90:1); \coordinate(lu)at(135:1);
\coordinate(l)at(180:1); \coordinate(rd)at(-45:1); \coordinate(d)at(-90:1); \coordinate(ld)at(-135:1);
\draw(0,0)circle(10mm); \draw[dotted](lu)--(ld)--(rd)--(ru);
\fill(lu)circle(1.4mm); \fill(ru)circle(1.4mm); \fill(ld)circle(1.4mm); \fill(rd)circle(1.4mm);
\draw[fill=white](r)circle(1.4mm); \draw[fill=white](u)circle(1.4mm); 
\draw[fill=white](l)circle(1.4mm); \draw[fill=white](d)circle(1.4mm);
\draw[blue](110:1)..controls(0,0)..(200:1);
\end{tikzpicture}
 &
\begin{tikzpicture}[baseline=0mm,scale=0.7]
\coordinate(r)at(0:1); \coordinate(ru)at(45:1); \coordinate(u)at(90:1); \coordinate(lu)at(135:1);
\coordinate(l)at(180:1); \coordinate(rd)at(-45:1); \coordinate(d)at(-90:1); \coordinate(ld)at(-135:1);
\draw(0,0)circle(10mm); \draw[dotted](lu)--(ld)--(rd)--(ru);
\fill(lu)circle(1.4mm); \fill(ru)circle(1.4mm); \fill(ld)circle(1.4mm); \fill(rd)circle(1.4mm);
\draw[fill=white](r)circle(1.4mm); \draw[fill=white](u)circle(1.4mm); 
\draw[fill=white](l)circle(1.4mm); \draw[fill=white](d)circle(1.4mm);
\draw[blue](70:1)--(-110:1);
\end{tikzpicture}
 &
\begin{tikzpicture}[baseline=0mm,scale=0.7]
\coordinate(r)at(0:1); \coordinate(ru)at(45:1); \coordinate(u)at(90:1); \coordinate(lu)at(135:1);
\coordinate(l)at(180:1); \coordinate(rd)at(-45:1); \coordinate(d)at(-90:1); \coordinate(ld)at(-135:1);
\draw(0,0)circle(10mm); \draw[dotted](lu)--(ld)--(rd)--(ru);
\fill(lu)circle(1.4mm); \fill(ru)circle(1.4mm); \fill(ld)circle(1.4mm); \fill(rd)circle(1.4mm);
\draw[fill=white](r)circle(1.4mm); \draw[fill=white](u)circle(1.4mm); 
\draw[fill=white](l)circle(1.4mm); \draw[fill=white](d)circle(1.4mm);
\draw[blue](110:1)..controls(0,0)..(-70:1);
\end{tikzpicture}
 &
\begin{tikzpicture}[baseline=0mm,scale=0.7]
\coordinate(r)at(0:1); \coordinate(ru)at(45:1); \coordinate(u)at(90:1); \coordinate(lu)at(135:1);
\coordinate(l)at(180:1); \coordinate(rd)at(-45:1); \coordinate(d)at(-90:1); \coordinate(ld)at(-135:1);
\draw(0,0)circle(10mm); \draw[dotted](lu)--(ld)--(rd)--(ru);
\fill(lu)circle(1.4mm); \fill(ru)circle(1.4mm); \fill(ld)circle(1.4mm); \fill(rd)circle(1.4mm);
\draw[fill=white](r)circle(1.4mm); \draw[fill=white](u)circle(1.4mm); 
\draw[fill=white](l)circle(1.4mm); \draw[fill=white](d)circle(1.4mm);
\draw[blue](70:1)..controls(0,0)..(-20:1);
\end{tikzpicture}\\
$(1,0,0)$&$(-1,0,0)$&$(0,1,0)$&$(0,-1,0)$&$(0,0,1)$\\
\begin{tikzpicture}[baseline=0mm,scale=0.7]
\coordinate(r)at(0:1); \coordinate(ru)at(45:1); \coordinate(u)at(90:1); \coordinate(lu)at(135:1);
\coordinate(l)at(180:1); \coordinate(rd)at(-45:1); \coordinate(d)at(-90:1); \coordinate(ld)at(-135:1);
\draw(0,0)circle(10mm); \draw[dotted](lu)--(ld)--(rd)--(ru);
\fill(lu)circle(1.4mm); \fill(ru)circle(1.4mm); \fill(ld)circle(1.4mm); \fill(rd)circle(1.4mm);
\draw[fill=white](r)circle(1.4mm); \draw[fill=white](u)circle(1.4mm); 
\draw[fill=white](l)circle(1.4mm); \draw[fill=white](d)circle(1.4mm);
\draw[blue](110:1)..controls(0,0)..(20:1);
\end{tikzpicture}
 &
\begin{tikzpicture}[baseline=0mm,scale=0.7]
\coordinate(r)at(0:1); \coordinate(ru)at(45:1); \coordinate(u)at(90:1); \coordinate(lu)at(135:1);
\coordinate(l)at(180:1); \coordinate(rd)at(-45:1); \coordinate(d)at(-90:1); \coordinate(ld)at(-135:1);
\draw(0,0)circle(10mm); \draw[dotted](lu)--(ld)--(rd)--(ru);
\fill(lu)circle(1.4mm); \fill(ru)circle(1.4mm); \fill(ld)circle(1.4mm); \fill(rd)circle(1.4mm);
\draw[fill=white](r)circle(1.4mm); \draw[fill=white](u)circle(1.4mm); 
\draw[fill=white](l)circle(1.4mm); \draw[fill=white](d)circle(1.4mm);
\draw[blue](160:1)..controls(0,0)..(-70:1);
\end{tikzpicture}
 &
\begin{tikzpicture}[baseline=0mm,scale=0.7]
\coordinate(r)at(0:1); \coordinate(ru)at(45:1); \coordinate(u)at(90:1); \coordinate(lu)at(135:1);
\coordinate(l)at(180:1); \coordinate(rd)at(-45:1); \coordinate(d)at(-90:1); \coordinate(ld)at(-135:1);
\draw(0,0)circle(10mm); \draw[dotted](lu)--(ld)--(rd)--(ru);
\fill(lu)circle(1.4mm); \fill(ru)circle(1.4mm); \fill(ld)circle(1.4mm); \fill(rd)circle(1.4mm);
\draw[fill=white](r)circle(1.4mm); \draw[fill=white](u)circle(1.4mm); 
\draw[fill=white](l)circle(1.4mm); \draw[fill=white](d)circle(1.4mm);
\draw[blue](20:1)..controls(0,0)..(-110:1);
\end{tikzpicture}
 &
\begin{tikzpicture}[baseline=0mm,scale=0.7]
\coordinate(r)at(0:1); \coordinate(ru)at(45:1); \coordinate(u)at(90:1); \coordinate(lu)at(135:1);
\coordinate(l)at(180:1); \coordinate(rd)at(-45:1); \coordinate(d)at(-90:1); \coordinate(ld)at(-135:1);
\draw(0,0)circle(10mm); \draw[dotted](lu)--(ld)--(rd)--(ru);
\fill(lu)circle(1.4mm); \fill(ru)circle(1.4mm); \fill(ld)circle(1.4mm); \fill(rd)circle(1.4mm);
\draw[fill=white](r)circle(1.4mm); \draw[fill=white](u)circle(1.4mm); 
\draw[fill=white](l)circle(1.4mm); \draw[fill=white](d)circle(1.4mm);
\draw[blue](160:1)--(20:1);
\end{tikzpicture}\\
$(0,0,-1)$&$(1,-1,0)$&$(0,1,-1)$&$(1,0,-1)$
\end{tabular}
\]
There are 14 complete $D$-laminations. The corresponding fan $\mathcal{F}(D)$ of $g$-vector cones is given as in the left diagram of Figure \ref{Fig g-cones}. 

The above construction is naturally generalized. For a positive integer $n$, we define a $\circ\bullet$-marked surface $(S_1,M_1):=(S_1(n),M_1(n))$ as follows: Let $(S_1,M_1)$ be a disk with $\vert M_1\vert =2n+2$ such that all marked points lie on $\partial S_1$. Let $D_1:=D_1(n)$ be a $\bullet$-dissection of $(S_1,M_1)$ as in the following diagram. 
\[
\begin{tikzpicture}[baseline=-10mm]
\coordinate(0)at(90:1); \coordinate(1)at(120:1); \coordinate(2)at(150:1); \coordinate(3)at(180:1);
\coordinate(4)at(210:1); \coordinate(5)at(240:1); \coordinate(6)at(270:1); \coordinate(7)at(300:1);
\coordinate(n)at(60:1); \coordinate(8)at(-30:1);
\draw(0,0)circle(10mm);
\draw(1)..controls(150:0.7)..node[fill=white,inner sep=1]{$1$}(3)..controls(210:0.7)..node[fill=white,inner sep=1]{$2$}(5)..controls(270:0.7)..node[fill=white,inner sep=1]{$3$}(7)--(-40:0.8) (n)--node[fill=white,inner sep=1]{$n$}(12:0.65);
\fill(1)circle(1mm); \fill(3)circle(1mm); \fill(5)circle(1mm); \fill(7)circle(1mm); \fill(n)circle(1mm);
\draw[fill=white](0)circle(1mm); \draw[fill=white](2)circle(1mm); 
\draw[fill=white](4)circle(1mm); \draw[fill=white](6)circle(1mm); \draw[fill=white](8)circle(1mm);
\node at(-5:0.65){$\vdots$};
\end{tikzpicture}
\]
Now, we classify all $D_1$-laminates in terms of their $g$-vectors. In the same way as above, it is easily shown that the $g$-vector of any $D_1$-laminate lies in the set 
\begin{equation}\label{triangulation}
    \{e_i - e_j \mid 1\leq i < j \leq n\} \cup \{\pm e_i\mid 1\leq i \leq n\}. 
\end{equation}
Moreover, each integer vector in \eqref{triangulation} is the $g$-vector of exactly one $D_1$-laminate. Thus the correspondence $\gamma\mapsto g(\gamma)$ gives a bijection between the set of $D_1$-laminates and the set \eqref{triangulation}. 
In particular, there are precisely $\frac{1}{2}n(n+3)$ $D_1$-laminates.

(2) 
Let $(S_2(3),M_2(3))$ be a disk with $\vert M_2(3)\vert =7$ such that all marked points lie on $\partial S_2(3)$. For a $\bullet$-dissection of $(S_2(3),M_2(3))$
\[
D=D_2(3)=
\begin{tikzpicture}[baseline=-1mm]
\coordinate(ru)at(30:1); \coordinate(u)at(90:1); \coordinate(lu)at(150:1);
\coordinate(rd)at(-30:1); \coordinate(d)at(-90:1); \coordinate(ld)at(-150:1);
\draw(0,0)circle(10mm);
\draw(u)--node[fill=white,inner sep=1]{$3$}(rd)--node[fill=white,inner sep=1]{$2$}(ld)--node[fill=white,inner sep=1]{$1$}(u);
\fill(u)circle(1mm); \fill(ld)circle(1mm); \fill(rd)circle(1mm);
\draw[fill=white](ru)circle(1mm); \draw[fill=white](lu)circle(1mm); 
\draw[fill=white](d)circle(1mm); \draw[fill=white](0,0)circle(1mm);
\end{tikzpicture}\ ,
\]
all $D$-laminates and the corresponding $g$-vectors are given as follows:
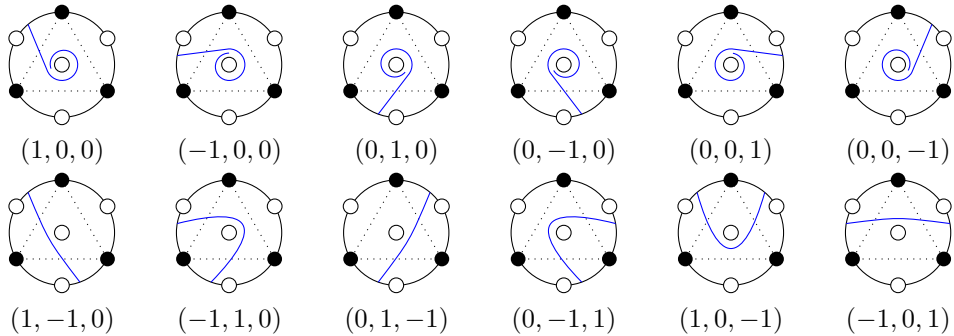
\[
\renewcommand{\arraystretch}{1.5}
\setlength{\tabcolsep}{4mm}
\begin{tabular}{cccccc}
\begin{tikzpicture}[baseline=0mm,scale=0.7]
\coordinate(ru)at(30:1); \coordinate(u)at(90:1); \coordinate(lu)at(150:1);
\coordinate(rd)at(-30:1); \coordinate(d)at(-90:1); \coordinate(ld)at(-150:1);
\draw(0,0)circle(10mm); \draw[dotted](u)--(rd)--(ld)--(u);
\fill(u)circle(1.4mm); \fill(ld)circle(1.4mm); \fill(rd)circle(1.4mm);
\draw[fill=white](ru)circle(1.4mm); \draw[fill=white](lu)circle(1.4mm); 
\draw[fill=white](d)circle(1.4mm); \draw[fill=white](0,0)circle(1.4mm);
\draw[blue](130:1)--(-160:0.3) (-160:0.3)arc(-160:20:0.3) (20:0.3)arc(20:200:0.26);
\end{tikzpicture}
 &
\begin{tikzpicture}[baseline=0mm,scale=0.7]
\coordinate(ru)at(30:1); \coordinate(u)at(90:1); \coordinate(lu)at(150:1);
\coordinate(rd)at(-30:1); \coordinate(d)at(-90:1); \coordinate(ld)at(-150:1);
\draw(0,0)circle(10mm); \draw[dotted](u)--(rd)--(ld)--(u);
\fill(u)circle(1.4mm); \fill(ld)circle(1.4mm); \fill(rd)circle(1.4mm);
\draw[fill=white](ru)circle(1.4mm); \draw[fill=white](lu)circle(1.4mm); 
\draw[fill=white](d)circle(1.4mm); \draw[fill=white](0,0)circle(1.4mm);
\draw[blue](170:1)--(90:0.3) (90:0.3)arc(90:-90:0.3) (-90:0.3)arc(-90:-270:0.26);
\end{tikzpicture}
 &
\begin{tikzpicture}[baseline=0mm,scale=0.7]
\coordinate(ru)at(30:1); \coordinate(u)at(90:1); \coordinate(lu)at(150:1);
\coordinate(rd)at(-30:1); \coordinate(d)at(-90:1); \coordinate(ld)at(-150:1);
\draw(0,0)circle(10mm); \draw[dotted](u)--(rd)--(ld)--(u);
\fill(u)circle(1.4mm); \fill(ld)circle(1.4mm); \fill(rd)circle(1.4mm);
\draw[fill=white](ru)circle(1.4mm); \draw[fill=white](lu)circle(1.4mm); 
\draw[fill=white](d)circle(1.4mm); \draw[fill=white](0,0)circle(1.4mm);
\draw[blue](250:1)--(-40:0.3) (-40:0.3)arc(-40:140:0.3) (140:0.3)arc(140:320:0.26);
\end{tikzpicture}
 &
\begin{tikzpicture}[baseline=0mm,scale=0.7]
\coordinate(ru)at(30:1); \coordinate(u)at(90:1); \coordinate(lu)at(150:1);
\coordinate(rd)at(-30:1); \coordinate(d)at(-90:1); \coordinate(ld)at(-150:1);
\draw(0,0)circle(10mm); \draw[dotted](u)--(rd)--(ld)--(u);
\fill(u)circle(1.4mm); \fill(ld)circle(1.4mm); \fill(rd)circle(1.4mm);
\draw[fill=white](ru)circle(1.4mm); \draw[fill=white](lu)circle(1.4mm); 
\draw[fill=white](d)circle(1.4mm); \draw[fill=white](0,0)circle(1.4mm);
\draw[blue](290:1)--(210:0.3) (210:0.3)arc(210:30:0.3) (30:0.3)arc(30:-150:0.26);
\end{tikzpicture}
\\
$(1,0,0)$&$(-1,0,0)$&$(0,1,0)$&$(0,-1,0)$
\\
\begin{tikzpicture}[baseline=0mm,scale=0.7]
\coordinate(ru)at(30:1); \coordinate(u)at(90:1); \coordinate(lu)at(150:1);
\coordinate(rd)at(-30:1); \coordinate(d)at(-90:1); \coordinate(ld)at(-150:1);
\draw(0,0)circle(10mm); \draw[dotted](u)--(rd)--(ld)--(u);
\fill(u)circle(1.4mm); \fill(ld)circle(1.4mm); \fill(rd)circle(1.4mm);
\draw[fill=white](ru)circle(1.4mm); \draw[fill=white](lu)circle(1.4mm); 
\draw[fill=white](d)circle(1.4mm); \draw[fill=white](0,0)circle(1.4mm);
\draw[blue](10:1)--(80:0.3) (80:0.3)arc(80:260:0.3) (260:0.3)arc(260:440:0.26);
\end{tikzpicture}
 &
\begin{tikzpicture}[baseline=0mm,scale=0.7]
\coordinate(ru)at(30:1); \coordinate(u)at(90:1); \coordinate(lu)at(150:1);
\coordinate(rd)at(-30:1); \coordinate(d)at(-90:1); \coordinate(ld)at(-150:1);
\draw(0,0)circle(10mm); \draw[dotted](u)--(rd)--(ld)--(u);
\fill(u)circle(1.4mm); \fill(ld)circle(1.4mm); \fill(rd)circle(1.4mm);
\draw[fill=white](ru)circle(1.4mm); \draw[fill=white](lu)circle(1.4mm); 
\draw[fill=white](d)circle(1.4mm); \draw[fill=white](0,0)circle(1.4mm);
\draw[blue](50:1)--(-30:0.3) (-30:0.3)arc(330:150:0.3) (150:0.3)arc(150:-30:0.26);
\end{tikzpicture}
 &
\begin{tikzpicture}[baseline=0mm,scale=0.7]
\coordinate(ru)at(30:1); \coordinate(u)at(90:1); \coordinate(lu)at(150:1);
\coordinate(rd)at(-30:1); \coordinate(d)at(-90:1); \coordinate(ld)at(-150:1);
\draw(0,0)circle(10mm); \draw[dotted](u)--(rd)--(ld)--(u);
\fill(u)circle(1.4mm); \fill(ld)circle(1.4mm); \fill(rd)circle(1.4mm);
\draw[fill=white](ru)circle(1.4mm); \draw[fill=white](lu)circle(1.4mm); 
\draw[fill=white](d)circle(1.4mm); \draw[fill=white](0,0)circle(1.4mm);
\draw[blue](130:1)..controls(-150:0.3)..(-70:1);
\end{tikzpicture}
 &
\begin{tikzpicture}[baseline=0mm,scale=0.7]
\coordinate(ru)at(30:1); \coordinate(u)at(90:1); \coordinate(lu)at(150:1);
\coordinate(rd)at(-30:1); \coordinate(d)at(-90:1); \coordinate(ld)at(-150:1);
\draw(0,0)circle(10mm); \draw[dotted](u)--(rd)--(ld)--(u);
\fill(u)circle(1.4mm); \fill(ld)circle(1.4mm); \fill(rd)circle(1.4mm);
\draw[fill=white](ru)circle(1.4mm); \draw[fill=white](lu)circle(1.4mm); 
\draw[fill=white](d)circle(1.4mm); \draw[fill=white](0,0)circle(1.4mm);
\draw[blue](170:1)..controls(50:0.7)and(10:0.7)..(-110:1);
\end{tikzpicture}
\\
$(0,0,1)$&$(0,0,-1)$&$(1,-1,0)$&$(-1,1,0)$
\\
\begin{tikzpicture}[baseline=0mm,scale=0.7]
\coordinate(ru)at(30:1); \coordinate(u)at(90:1); \coordinate(lu)at(150:1);
\coordinate(rd)at(-30:1); \coordinate(d)at(-90:1); \coordinate(ld)at(-150:1);
\draw(0,0)circle(10mm); \draw[dotted](u)--(rd)--(ld)--(u);
\fill(u)circle(1.4mm); \fill(ld)circle(1.4mm); \fill(rd)circle(1.4mm);
\draw[fill=white](ru)circle(1.4mm); \draw[fill=white](lu)circle(1.4mm); 
\draw[fill=white](d)circle(1.4mm); \draw[fill=white](0,0)circle(1.4mm);
\draw[blue](-110:1)..controls(-30:0.3)..(50:1);
\end{tikzpicture}
 &
\begin{tikzpicture}[baseline=0mm,scale=0.7]
\coordinate(ru)at(30:1); \coordinate(u)at(90:1); \coordinate(lu)at(150:1);
\coordinate(rd)at(-30:1); \coordinate(d)at(-90:1); \coordinate(ld)at(-150:1);
\draw(0,0)circle(10mm); \draw[dotted](u)--(rd)--(ld)--(u);
\fill(u)circle(1.4mm); \fill(ld)circle(1.4mm); \fill(rd)circle(1.4mm);
\draw[fill=white](ru)circle(1.4mm); \draw[fill=white](lu)circle(1.4mm); 
\draw[fill=white](d)circle(1.4mm); \draw[fill=white](0,0)circle(1.4mm);
\draw[blue](-70:1)..controls(170:0.7)and(130:0.7)..(10:1);
\end{tikzpicture}
 &
\begin{tikzpicture}[baseline=0mm,scale=0.7]
\coordinate(ru)at(30:1); \coordinate(u)at(90:1); \coordinate(lu)at(150:1);
\coordinate(rd)at(-30:1); \coordinate(d)at(-90:1); \coordinate(ld)at(-150:1);
\draw(0,0)circle(10mm); \draw[dotted](u)--(rd)--(ld)--(u);
\fill(u)circle(1.4mm); \fill(ld)circle(1.4mm); \fill(rd)circle(1.4mm);
\draw[fill=white](ru)circle(1.4mm); \draw[fill=white](lu)circle(1.4mm); 
\draw[fill=white](d)circle(1.4mm); \draw[fill=white](0,0)circle(1.4mm);
\draw[blue](50:1)..controls(-70:0.7)and(-110:0.7)..(130:1);
\end{tikzpicture}
 &
\begin{tikzpicture}[baseline=0mm,scale=0.7]
\coordinate(ru)at(30:1); \coordinate(u)at(90:1); \coordinate(lu)at(150:1);
\coordinate(rd)at(-30:1); \coordinate(d)at(-90:1); \coordinate(ld)at(-150:1);
\draw(0,0)circle(10mm); \draw[dotted](u)--(rd)--(ld)--(u);
\fill(u)circle(1.4mm); \fill(ld)circle(1.4mm); \fill(rd)circle(1.4mm);
\draw[fill=white](ru)circle(1.4mm); \draw[fill=white](lu)circle(1.4mm); 
\draw[fill=white](d)circle(1.4mm); \draw[fill=white](0,0)circle(1.4mm);
\draw[blue](10:1)..controls(90:0.3)..(170:1);
\end{tikzpicture}
\\
$(0,1,-1)$&$(0,-1,1)$&$(1,0,-1)$&$(-1,0,1)$
\end{tabular}
\]
There are 20 complete $D$-laminations. The fan $\mathcal{F}(D)$ is given as in the center diagram of Figure \ref{Fig g-cones}.

The above construction is naturally generalized. For a positive integer $n$, we define a $\circ\bullet$-marked surface $(S_2,M_2):=(S_2(n),M_2(n))$ as follows: 
Let $(S_2,M_2)$ be a disk with $\vert M_2\vert =2n+1$ such that one marked point in $(M_2)_{\circ}$ is a puncture and the others lie on $\partial S_2$. Let $D_2$ be a $\bullet$-dissection of $(S_2,M_2)$ as in the following diagram.
\[
\begin{tikzpicture}[baseline=0mm]
\coordinate(0)at(90:1); \coordinate(1)at(120:1); \coordinate(2)at(150:1); \coordinate(3)at(180:1);
\coordinate(4)at(210:1); \coordinate(5)at(240:1); \coordinate(6)at(270:1); \coordinate(7)at(300:1);
\coordinate(n)at(60:1); \coordinate(8)at(-30:1);
\draw(0,0)circle(10mm);
\draw(1)..controls(150:0.7)..node[fill=white,inner sep=1.5]{$1$}(3)..controls(210:0.7)..node[fill=white,inner sep=1.5]{$2$}(5)..controls(270:0.7)..node[fill=white,inner sep=1.5]{$3$}(7)--(-40:0.8) (1)..controls(90:0.7)..node[fill=white,inner sep=1.5]{$n$}(n)--(40:0.8);
\fill(1)circle(1mm); \fill(3)circle(1mm); \fill(5)circle(1mm); \fill(7)circle(1mm); \fill(n)circle(1mm);
\draw[fill=white](0,0)circle(1mm); \draw[fill=white](0)circle(1mm); \draw[fill=white](2)circle(1mm); 
\draw[fill=white](4)circle(1mm); \draw[fill=white](6)circle(1mm); \draw[fill=white](8)circle(1mm);
\node at(10:0.65){$\vdots$};
\end{tikzpicture}
\]
In the same way as above, it is not difficult to see that there are precisely $n(n+1)$ $D_2$-laminates each of which is determined by its $g$-vector. More precisely, the map $\gamma\mapsto g(\gamma)$ gives a bijection from the set of $g$-vectors of all $D_2$-laminates 
to the set
\begin{equation} \label{ptriangulation}
    \{\pm (e_i - e_j) \mid 1\leq i < j \leq n\} \cup \{\pm e_i\mid 1\leq i \leq n\}. 
\end{equation}

(3) Consider a torus $S=T^2$ with $\partial S=\emptyset$ and $\vert M\vert =2$ (i.e., both marked points are punctures). 
Let $D$ be a $\bullet$-dissection of $(S,M)$ given by 
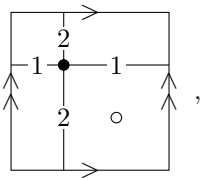
\[
\begin{tikzpicture}[baseline=-1mm,scale=0.7]
\coordinate(b)at(-0.5,0.5); \coordinate(w)at(0.5,-0.5);
\draw (-1.5,-1.5) rectangle (1.5,1.5);
\draw[fill=black](b)circle(1.4mm);
\draw[fill=white](w)circle(1.4mm);
\draw(-1.5,0.5)--node[fill=white,inner sep=1]{$1$}(b)--node(1)[fill=white,inner sep=1]{$1$}(1.5,0.5); 
\draw(-0.5,-1.5)--node(1)[fill=white,inner sep=1]{$2$}(b)--node(1)[fill=white,inner sep=1]{$2$}(-0.5,1.5); 
\node(ll) at(-1.5,0) {\rotatebox{90}{$>>$}};
\node(rr) at(1.5,0) {\rotatebox{90}{$>>$}};
\node(uu) at(0,1.5) {\rotatebox{0}{$>$}};
\node(dd) at(0,-1.5) {\rotatebox{0}{$>$}};
\end{tikzpicture}\ ,
\]
where we identify the opposite sides of the square in the same direction.
All $D$-laminates and the corresponding $g$-vectors are given as follows:
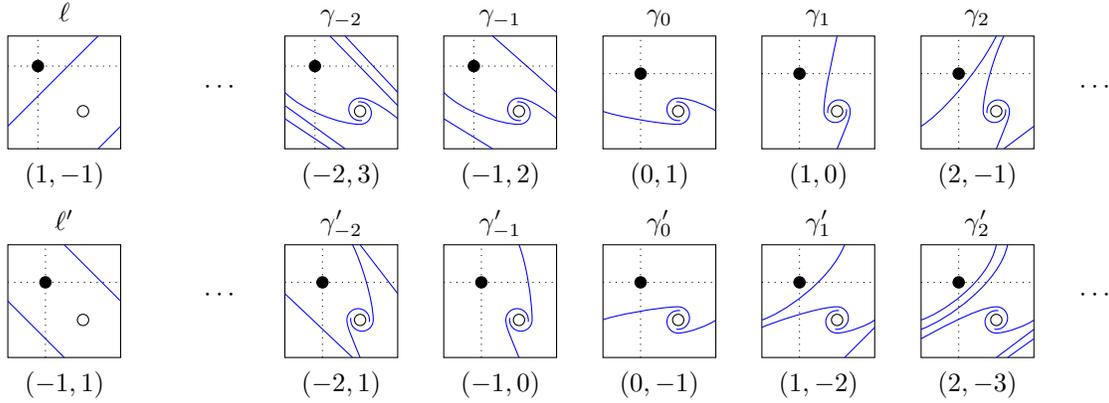
\[ 
\renewcommand{\arraystretch}{1.5}
\setlength{\tabcolsep}{1mm}
\begin{tabular}{cccccccc}
   $\ell$ && $\gamma_{-2}$ & $\gamma_{-1}$ & $\gamma_0$ & $\gamma_{1}$ & $\gamma_{2}$\\
   \begin{tikzpicture}[baseline=0mm,scale=0.5]
      \coordinate(b)at(-0.7,0.7); \coordinate(w)at(0.5,-0.5);
      \draw (-1.5,-1.5) rectangle (1.5,1.5);
      \draw[fill=black](b)circle(1.5mm);
      \draw[fill=white](w)circle(1.5mm);
      \draw[dotted](-1.5,0.7)--(1.5,0.7); 
      \draw[dotted](-0.7,-1.5)--(-0.7,1.5); 
      \draw[blue](-1.5,-0.9)--(0.9,1.5);
      \draw[blue](0.9,-1.5)--(1.5,-0.9);
   \end{tikzpicture}
   &
   \hspace{5mm}$\cdots$
   &
   \begin{tikzpicture}[baseline=0mm,scale=0.5]
      \coordinate(b)at(-0.7,0.7); \coordinate(w)at(0.5,-0.5);
      \draw (-1.5,-1.5) rectangle (1.5,1.5);
      \draw[fill=black](b)circle(1.5mm);
      \draw[fill=white](w)circle(1.5mm);
      \draw[dotted](-1.5,0.7)--(1.5,0.7); 
      \draw[dotted](-0.7,-1.5)--(-0.7,1.5); 
      \draw[blue](0.1,1.5)--(1.5,0.0);
      \draw[blue](-0.3,1.5)--(1.5,-0.35);
      \draw[blue](1.5,-0.7)..controls(1.4,-0.6)and(0.8,-0.15)..(0.5,-0.15)arc(90:280:0.3);
      \draw[blue](-1.5,0.0)..controls(-1,-0.5)and(0,-0.85)..(0.5,-0.85)arc(-90:100:0.3);
      \draw[blue](-1.5,-0.35)--(0.1,-1.5);
      \draw[blue](-1.5,-0.7)--(-0.3,-1.5);
   \end{tikzpicture}
   &
   \begin{tikzpicture}[baseline=0mm,scale=0.5]
      \coordinate(b)at(-0.7,0.7); \coordinate(w)at(0.5,-0.5);
      \draw (-1.5,-1.5) rectangle (1.5,1.5);
      \draw[fill=black](b)circle(1.5mm);
      \draw[fill=white](w)circle(1.5mm);
      \draw[dotted](-1.5,0.7)--(1.5,0.7); 
      \draw[dotted](-0.7,-1.5)--(-0.7,1.5); 
      \draw[blue](-0.2,1.5)--(1.5,0);
      \draw[blue](1.5,-0.7)..controls(1.4,-0.6)and(0.8,-0.15)..(0.5,-0.15)arc(90:280:0.3);
      \draw[blue](-1.5,0.0)..controls(-1,-0.5)and(0,-0.85)..(0.5,-0.85)arc(-90:100:0.3);
      \draw[blue](-1.5,-0.7)--(-0.2,-1.5);
   \end{tikzpicture}
   &        
   \begin{tikzpicture}[baseline=0mm,scale=0.5]
      \coordinate(b)at(-0.5,0.5); \coordinate(w)at(0.5,-0.5);
      \draw (-1.5,-1.5) rectangle (1.5,1.5);
      \draw[fill=black](b)circle(1.5mm);
      \draw[fill=white](w)circle(1.5mm);
      \draw[dotted](-1.5,0.5)--(1.5,0.5); 
      \draw[dotted](-0.5,-1.5)--(-0.5,1.5); 
      \draw[blue](1.5,-0.5)..controls(1.4,-0.4)and(0.8,-0.15)..(0.5,-0.15)arc(90:280:0.3);
      \draw[blue](-1.5,-0.5)..controls(-1.3,-0.6)and(0,-0.85)..(0.5,-0.85)arc(-90:100:0.3);
   \end{tikzpicture}
   &
   \begin{tikzpicture}[baseline=0mm,scale=0.5]
      \coordinate(b)at(-0.5,0.5); \coordinate(w)at(0.5,-0.5);
      \draw (-1.5,-1.5) rectangle (1.5,1.5);
      \draw[fill=black](b)circle(1.5mm);
      \draw[fill=white](w)circle(1.5mm);
      \draw[dotted](-1.5,0.5)--(1.5,0.5); 
      \draw[dotted](-0.5,-1.5)--(-0.5,1.5); 
      \draw[blue](0.5,1.5)..controls(0.4,1)and(0.15,0)..(0.15,-0.5)arc(-180:10:0.3);
      \draw[blue](0.5, -1.5)..controls(0.7,-1)and(0.85,-0.7)..(0.85,-0.5)arc(0:190:0.3);
   \end{tikzpicture}
   &
   \begin{tikzpicture}[baseline=0mm,scale=0.5]
      \coordinate(b)at(-0.5,0.5); \coordinate(w)at(0.5,-0.5);
      \draw (-1.5,-1.5) rectangle (1.5,1.5);
      \draw[fill=black](b)circle(1.5mm);
      \draw[fill=white](w)circle(1.5mm);
      \draw[dotted](-1.5,0.5)--(1.5,0.5); 
      \draw[dotted](-0.5,-1.5)--(-0.5,1.5); 
      \draw[blue](0.5,1.5)..controls(0,0.5)and(-1,-0.6)..(-1.5,-0.9);
      \draw[blue](1.5,-0.9)--(0.7,-1.5);
      \draw[blue](0.7,1.5)..controls(0.45,1)and(0.15,0)..(0.15,-0.5)arc(-180:10:0.3);
      \draw[blue](0.5, -1.5)..controls(0.7,-1)and(0.85,-0.7)..(0.85,-0.5)arc(0:190:0.3);
   \end{tikzpicture} 
   & $\cdots$\\
   $(1,-1)$ && $(-2,3)$ & $(-1,2)$ & $(0,1)$ & $(1,0)$ & $(2,-1)$\\
   $\ell'$ && $\gamma'_{-2}$ & $\gamma'_{-1}$ & $\gamma'_0$ & $\gamma'_{1}$ & $\gamma'_{2}$\\ 
      \begin{tikzpicture}[baseline=0mm,scale=0.5]
         \coordinate(b)at(-0.5,0.5); \coordinate(w)at(0.5,-0.5);
         \draw (-1.5,-1.5) rectangle (1.5,1.5);
         \draw[fill=black](b)circle(1.5mm);
         \draw[fill=white](w)circle(1.5mm);
         \draw[dotted](-1.5,0.5)--(1.5,0.5); 
         \draw[dotted](-0.5,-1.5)--(-0.5,1.5); 
         \draw[blue](-1.5,0)--(0,-1.5);
         \draw[blue](0,1.5)--(1.5,0);
      \end{tikzpicture}
      &
      \hspace{5mm}$\cdots$
      &
      \begin{tikzpicture}[baseline=0mm,scale=0.5]
         \coordinate(b)at(-0.5,0.5); \coordinate(w)at(0.5,-0.5);
         \draw (-1.5,-1.5) rectangle (1.5,1.5);
         \draw[fill=black](b)circle(1.5mm);
         \draw[fill=white](w)circle(1.5mm);
         \draw[dotted](-1.5,0.5)--(1.5,0.5); 
         \draw[dotted](-0.5,-1.5)--(-0.5,1.5); 
         \draw[blue](-1.5,0.2)--(0.3,-1.5);
         \draw[blue](0.5,1.5)--(1.5,0.2);
         \draw[blue](0.3,1.5)..controls(0.55,1)and(0.85,0)..(0.85,-0.5)arc(0:-190:0.3);
         \draw[blue](0.5, -1.5)..controls(0.3,-1)and(0.15,-0.7)..(0.15,-0.5)arc(180:-10:0.3);
      \end{tikzpicture}
      &
      \begin{tikzpicture}[baseline=0mm,scale=0.5]
         \coordinate(b)at(-0.5,0.5); \coordinate(w)at(0.5,-0.5);
         \draw (-1.5,-1.5) rectangle (1.5,1.5);
         \draw[fill=black](b)circle(1.5mm);
         \draw[fill=white](w)circle(1.5mm);
         \draw[dotted](-1.5,0.5)--(1.5,0.5); 
         \draw[dotted](-0.5,-1.5)--(-0.5,1.5); 
         \draw[blue](0.5,1.5)..controls(0.7,1)and(0.85,0)..(0.85,-0.5)arc(0:-190:0.3);
         \draw[blue](0.5, -1.5)..controls(0.3,-1)and(0.15,-0.7)..(0.15,-0.5)arc(180:-10:0.3);
      \end{tikzpicture}
      &
      \begin{tikzpicture}[baseline=0mm,scale=0.5]
         \coordinate(b)at(-0.5,0.5); \coordinate(w)at(0.5,-0.5);
         \draw (-1.5,-1.5) rectangle (1.5,1.5);
         \draw[fill=black](b)circle(1.5mm);
         \draw[fill=white](w)circle(1.5mm);
         \draw[dotted](-1.5,0.5)--(1.5,0.5); 
         \draw[dotted](-0.5,-1.5)--(-0.5,1.5); 
         \draw[blue](1.5,-0.5)..controls(1.4,-0.6)and(0.8,-0.85)..(0.5,-0.85)arc(270:80:0.3);
         \draw[blue](-1.5,-0.5)..controls(-1.3,-0.4)and(0,-0.15)..(0.5,-0.15)arc(90:-100:0.3);
      \end{tikzpicture}
      &
      \begin{tikzpicture}[baseline=0mm,scale=0.5]
         \coordinate(b)at(-0.5,0.5); \coordinate(w)at(0.5,-0.5);
         \draw (-1.5,-1.5) rectangle (1.5,1.5);
         \draw[fill=black](b)circle(1.5mm);
         \draw[fill=white](w)circle(1.5mm);
         \draw[dotted](-1.5,0.5)--(1.5,0.5); 
         \draw[dotted](-0.5,-1.5)--(-0.5,1.5); 
         \draw[blue](0.7,1.5)..controls(0.3,0.5)and(-1,-0.35)..(-1.5,-0.5);
         \draw[blue](1.5,-0.7)--(0.7,-1.5);
         \draw[blue](1.5,-0.5)..controls(1.4,-0.6)and(0.8,-0.85)..(0.5,-0.85)arc(270:80:0.3);
         \draw[blue](-1.5,-0.7)..controls(-1.3,-0.6)and(0,-0.15)..(0.5,-0.15)arc(90:-100:0.3);
      \end{tikzpicture}
      &
      \begin{tikzpicture}[baseline=0mm,scale=0.5]
         \coordinate(b)at(-0.5,0.5); \coordinate(w)at(0.5,-0.5);
         \draw (-1.5,-1.5) rectangle (1.5,1.5);
         \draw[fill=black](b)circle(1.5mm);
         \draw[fill=white](w)circle(1.5mm);
         \draw[dotted](-1.5,0.5)--(1.5,0.5); 
         \draw[dotted](-0.5,-1.5)--(-0.5,1.5); 
         \draw[blue](0.5,1.5)..controls(0.3,0.5)and(-1,-0.35)..(-1.5,-0.5);
         \draw[blue](0.8,1.5)..controls(0.6,0.5)and(-0.8,-0.5)..(-1.5,-0.75);
         \draw[blue](-1.5,-1)..controls(-1.3,-0.9)and(0,-0.15)..(0.5,-0.15)arc(90:-100:0.3);
         \draw[blue](1.5,-0.75)--(0.5,-1.5);
         \draw[blue](1.5,-1)--(0.8,-1.5);
         \draw[blue](1.5,-0.5)..controls(1.4,-0.6)and(0.8,-0.85)..(0.5,-0.85)arc(270:80:0.3);
      \end{tikzpicture}
      & $\cdots$\\
      $(-1,1)$ && $(-2,1)$ & $(-1,0)$ & $(0,-1)$ & $(1,-2)$ &  $(2,-3)$
\end{tabular}
\]
where $\ell, \ell'$ are closed $D$-laminates and $\gamma_m, \gamma'_m$ are non-closed $D$-laminates for all $m\in \mathbb{Z}$. 
We find that the set $\{\{\gamma_m, \gamma_{m+1}\}, \{\gamma'_m, \gamma'_{m+1}\} \mid m \in \mathbb{Z}\}$ provides all complete $D$-laminations. 
Then $\mathcal{F}(D)$ is a non-singular fan given as in the right diagram of Figure \ref{Fig g-cones}. 

For the closed $D$-laminate $\ell$, its $g$-vector $g(\ell)=(1,-1)\in \mathbb{Z}^2$ does not contained in $\vert \mathcal{F}(D)\vert $. It will be approximated by using the Dehn twist $\mathsf{T}_{\ell}$ along $\ell$ (we refer to Section \ref{SecDehn} for the details). 
In fact, we have $\mathsf{T}_{\ell}(\gamma_i)=\gamma_{i+1}$ for any $i\in \mathbb{Z}_{>0}$ and hence
\[
g(\ell)=(1,-1)\in\overline{\bigcup_{m \ge 0} C(\mathsf{T}_{\ell}^m(\{\gamma_1\}))}. 
\] 

\begin{figure}[htp]
\[
\renewcommand{\arraystretch}{1.5}
\setlength{\tabcolsep}{3mm}
\begin{tabular}{ccc}
\begin{tikzpicture}[baseline=0mm,scale=0.7]
\coordinate(0)at(0,0); \coordinate(x)at(-0.7,-1.1); \coordinate(-x)at(0.7,1.1);
\coordinate(y)at(0:2); \coordinate(-y)at(180:2); \coordinate(z)at(90:2); \coordinate(-z)at(-90:2);
\coordinate(x-y)at(-2.7,-1.1); \coordinate(x-z)at(-0.7,-3.1); \coordinate(y-z)at(2,-2);
\draw[->,dashed](-x)--(x); \draw[->,dashed](-y)--(y); \draw[->,dashed](-z)--(z);
\draw[->,dashed](0)--(x-y); \draw[->,dashed](0)--(x-z); \draw[->,dashed](0)--(y-z);
\draw(x)--(y)--(z)--(-y)--(x-y)--(z)--(x)--(x-y)--(x-z)--(x)--(y-z)--(x-z) (y-z)--(y);
\draw[dotted](z)--(-x)--(y) (-z)--(-x)--(-y)--(-z)--(x-z) (x-y)--(-z)--(y-z)--(-x);
\end{tikzpicture}
 &
\begin{tikzpicture}[baseline=0mm,scale=0.7]
\coordinate(0)at(0,0); \coordinate(x)at(-0.7,-1.1); \coordinate(-x)at(0.7,1.1);
\coordinate(y)at(0:2); \coordinate(-y)at(180:2); \coordinate(z)at(90:2); \coordinate(-z)at(-90:2);
\coordinate(x-y)at(-2.7,-1.1); \coordinate(x-z)at(-0.7,-3.1); \coordinate(y-z)at(2,-2);
\coordinate(-xy)at(2.7,1.1); \coordinate(x-z)at(-0.7,-3.1); \coordinate(y-z)at(2,-2);
\coordinate(-xz)at(0.7,3.1); \coordinate(-yz)at(-2,2);
\draw[->,dashed](-x)--(x); \draw[->,dashed](-y)--(y); \draw[->,dashed](-z)--(z);
\draw[->,dashed](0)--(x-y); \draw[->,dashed](0)--(x-z); \draw[->,dashed](0)--(y-z);
\draw[->,dashed](0)--(-xy); \draw[->,dashed](0)--(x-z); \draw[->,dashed](0)--(y-z);
\draw[->,dashed](0)--(-xz); \draw[->,dashed](0)--(-yz);
\draw(x)--(y)--(z)--(-yz)--(x-y)--(z)--(x)--(x-y)--(x-z)--(x)--(y-z)--(x-z) (y)--(y-z)--(-xy)--(-xz)--(y)--(-xy) (z)--(-xz)--(-yz);
\draw[dotted](-z)--(-x)--(-y)--(-xz)--(-x)--(-xy) (-yz)--(-y)--(x-y)--(-z)--(x-z) (-y)--(-z)--(y-z)--(-x);
\end{tikzpicture}
 &
\begin{tikzpicture}[baseline=0mm,scale=0.7]
\coordinate(0)at(0,0); \coordinate(p)at(2,-2); \coordinate(-p)at(-2,2);
\coordinate(x)at(0:2); \coordinate(-x)at(180:2); \coordinate(y)at(90:2); \coordinate(-y)at(-90:2);
\draw[->](-x)--(x); \draw[->](-y)--(y);
\draw[dashed](2,-2)--(-2,2);
\draw($(-x)!0.5!(-p)$)--(0)--($(-x)!0.67!(-p)$) ($(-x)!0.75!(-p)$)--(0)--($(-x)!0.8!(-p)$) (0)--($(y)!0.833!(-p)$) ($(-x)!0.86!(-p)$)--(0)--($(-x)!0.875!(-p)$) ($(-x)!0.89!(-p)$)--(0)--($(-x)!0.905!(-p)$) (0)--($(-x)!0.92!(-p)$) ($(-x)!0.93!(-p)$)--(0)--($(-x)!0.94!(-p)$) ($(-x)!0.95!(-p)$)--(0)--($(-x)!0.96!(-p)$);
\draw($(y)!0.5!(-p)$)--(0)--($(y)!0.67!(-p)$) ($(y)!0.75!(-p)$)--(0)--($(y)!0.8!(-p)$) (0)--($(y)!0.833!(-p)$) ($(y)!0.86!(-p)$)--(0)--($(y)!0.875!(-p)$) ($(y)!0.89!(-p)$)--(0)--($(y)!0.905!(-p)$) (0)--($(y)!0.92!(-p)$) ($(y)!0.93!(-p)$)--(0)--($(y)!0.94!(-p)$) ($(y)!0.95!(-p)$)--(0)--($(y)!0.96!(-p)$);
\draw ($(x)!0.5!(p)$)--(0)--($(x)!0.67!(p)$) ($(x)!0.75!(p)$)--(0)--($(x)!0.8!(p)$) (0)--($(-y)!0.833!(p)$) ($(x)!0.86!(p)$)--(0)--($(x)!0.875!(p)$) ($(x)!0.89!(p)$)--(0)--($(x)!0.905!(p)$) (0)--($(x)!0.92!(p)$) ($(x)!0.93!(p)$)--(0)--($(x)!0.94!(p)$) ($(x)!0.95!(p)$)--(0)--($(x)!0.96!(p)$);
\draw($(-y)!0.5!(p)$)--(0)--($(-y)!0.67!(p)$) ($(-y)!0.75!(p)$)--(0)--($(-y)!0.8!(p)$) (0)--($(-y)!0.833!(p)$) ($(-y)!0.86!(p)$)--(0)--($(-y)!0.875!(p)$) ($(-y)!0.89!(p)$)--(0)--($(-y)!0.905!(p)$) (0)--($(-y)!0.92!(p)$) ($(-y)!0.93!(p)$)--(0)--($(-y)!0.94!(p)$) ($(-y)!0.95!(p)$)--(0)--($(-y)!0.96!(p)$);
\end{tikzpicture}\\
(1)&(2)&(3)
\end{tabular}
\]
   \caption{A fan $\mathcal{F}(D)$ of $g$-vector cones for Examples (1)-(3)}
   \label{Fig g-cones}
\end{figure}

\section{$g$-vectors and lattice points}\label{Secg}

The aim of this section is to prove the following result.

\begin{theorem} \label{lattice points}
Let $D$ be a $\bullet$-dissection of $(S,M)$.
Then there is a bijection
$$
\{\text{$D$-laminations}\} \rightarrow \mathbb{Z}^{\vert D\vert }
$$
given by the map $\mathcal{X} \mapsto g(\mathcal{X}):=\sum_{\gamma\in \mathcal{X}}g(\gamma)$, where $g(\emptyset):=0$.
\end{theorem}

Let $n$ be a positive integer. 
To prove Theorem \ref{lattice points}, we first consider triples \begin{equation}
(S_i,M_i,D_i):=(S_i(n),M_i(n),D_i(n))    
\end{equation} 
for $i=1,2$ defined in Section \ref{Secexample}(1) and (2) respectively. By definition, we have $\vert D_i\vert =n$ for $i=1,2$. 

\begin{proposition}\label{Di fan}
Under the above notation, for $i \in\{1,2\}$, $\mathcal{F}(D_i)$ is a non-singular fan which is finite. In particular, we have $\vert \mathcal{F}(D_i)\vert =\mathbb{R}^{n}$. 
\end{proposition}

\begin{proof} 
It is already shown in Section \ref{Secexample}(1) (resp, (2)) that the number of $D_1$-laminates (resp., $D_2$-laminates) is equal to $\frac{1}{2}n(n+3)$ (resp., $n(n+1)$), in particular, it is finite. Then, the assertion follows from Theorem \ref{comp simp fan}.
\end{proof}

\begin{corollary} \label{for stars}
Theorem \ref{lattice points} holds for $D=D_1$ or $D=D_2$.
\end{corollary}

\begin{proof}
By Proposition \ref{Di fan}, for $i \in \{1,2\}$, $\mathcal{F}(D_i)$ is a finite non-singular fan satisfying 
\begin{equation}\label{Diall}
    \vert \mathcal{F}(D_i)\vert =\mathbb{R}^{n}.
\end{equation} 
We show that the map $\mathcal{X}\mapsto g(\mathcal{X})$ gives a one-to-one correspondence between the set of $D_i$-laminations and $\mathbb{Z}^{n}$. Here, we remark that every $D_i$-lamination consists only of non-closed $D_i$-laminates by Proposition \ref{comp simp fan}. 
Let $g\in \mathbb{Z}^{n}$ be an integer vector. By \eqref{Diall}, there is a complete $D_i$-lamination $\mathcal{X}'$ such that $g\in C(\mathcal{X}')$. 
Since $C(\mathcal{X}')$ is non-singular, $g$ is uniquely written by $g=\sum_{\gamma\in \mathcal{X}'}a_{\gamma}g(\gamma)$ for non-negative integers $a_{\gamma}$. Consider a $D_i$-lamination $\mathcal{X}$ consisting of $a_{\gamma}$ copies of $\gamma$ for all $\gamma\in\mathcal{X}'$. 
Then, it satisfies $g=g(\mathcal{X})$. 
Moreover, this is a unique $D_i$-lamination whose $g$-vector is $g$. In fact, we see in Section \ref{Secexample} that the map $\gamma\mapsto g(\gamma)$ is a injective map from the set of $D_i$-laminates to $\mathbb{Z}^n$ whose image is given by the sets \eqref{triangulation} and \eqref{ptriangulation} for $i=1,2$ respectively, and it naturally gives rise to an injective map from the set of $D_i$-laminations to $\mathbb{Z}^n$. It completes a proof.  
\end{proof}



Now, we are ready to prove Theorem \ref{lattice points}.
\begin{proof}[Proof of Theorem \ref{lattice points}]
Let $D$ be a $\bullet$-dissection of $(S,M)$ and $g = (g_d)_{d\in D}$ an arbitrary integer vector in $\mathbb{Z}^{\vert D\vert }$.
In the following, we construct a $D$-lamination $\mathcal{X}$ such that $g=g(\mathcal{X})$.

Recall that $(S, M)$ is divided into polygons $\triangle_v$ for all $v \in M_{\circ}$.
For $v\in M_{\circ}$, we can naturally embed $\triangle_v$ into the above $\bullet$-dissection $D_i$ of $(S_i,M_i)$ for $i=1$ or $2$.
More precisely, $(S_i,M_i)$ is obtained from $\triangle_v$ by gluing a digon with one $\circ$-marked point on each $\bullet$-arc of $D \cap \triangle_v$, where $D \cap \triangle_v$ form $D_i$ in $(S_i,M_i)$.
By Corollary \ref{for stars}, there is a unique $D_i$-lamination $\mathcal{X}_v$ such that $g(\mathcal{X}_v)=(g_d)_{d\in D \cap \triangle_v}$.
We regard $\mathcal{X}_v \cap \triangle_v$ as a set of pairwise non-intersecting curves in $(S,M)$ with $\vert g_d\vert $ endpoints on $d \in D \cap \triangle_v$.
Then we can glue the curves of $\mathcal{X}_v \cap \triangle_v$ for all $v \in M_{\circ}$ at their endpoints on $D$ (see Figure \ref{Fig gluing}).
As a result, we obtain a set $\mathcal{X}$ of pairwise non-intersecting $\circ$-laminates of $(S,M)$. 
From our construction, every $\circ$-laminate of $\mathcal{X}$ is a $D$-laminate, and hence $\mathcal{X}$ forms a $D$-lamination such that $g(\mathcal{X})=g$ as desired.

On the other hand, the uniqueness of $\mathcal{X}$ follows from that of $\mathcal{X}_v$ for any $v \in M_{\circ}$.
\end{proof}

\begin{figure}[htp]
\begin{center}
\begin{tikzpicture}[baseline=-1mm]
\coordinate(lu)at(-0.3,0.8); \coordinate(ld)at(-0.3,-0.8); \coordinate(ru)at(0.3,0.8); \coordinate(rd)at(0.3,-0.8);
\coordinate(llu)at(-0.9,1); \coordinate(lld)at(-0.9,-1); \coordinate(rru)at(0.9,1); \coordinate(rrd)at(0.9,-1);
\draw(-1.8,1)--(llu)--(lu)--(ld)--(lld)--(-1.8,-1) (1.8,1)--(rru)--(ru)--(rd)--(rrd)--(1.8,-1);
\draw(-1.8,-0.4)--(-0.3,0.2) (-1.8,-0.6)--(-0.3,0) (-1.8,-0.8)--(-0.3,-0.2);
\draw(1.8,0.8)--(0.3,0.2) (1.8,0.6)--(0.3,0) (1.8,0.4)--(0.3,-0.2);
\fill(llu)circle(1mm); \fill(lu)circle(1mm); \fill(ru)circle(1mm); \fill(rru)circle(1mm);
\fill(lld)circle(1mm); \fill(ld)circle(1mm); \fill(rd)circle(1mm); \fill(rrd)circle(1mm);
\draw[fill=white](-1.7,0)circle(1mm); \draw[fill=white](1.7,0)circle(1mm);
\end{tikzpicture}
 \hspace{10mm}$\Longrightarrow$\hspace{10mm}
\begin{tikzpicture}[baseline=-1mm]
\coordinate(u)at(0,0.8); \coordinate(d)at(0,-0.8);
\coordinate(llu)at(-1.5,1); \coordinate(lu)at(-0.6,1); \coordinate(rru)at(1.5,1); \coordinate(ru)at(0.6,1);
\coordinate(lld)at(-1.5,-1); \coordinate(ld)at(-0.6,-1); \coordinate(rrd)at(1.5,-1); \coordinate(rd)at(0.6,-1);
\draw(llu)--(lu)--(u)--(ru)--(rru) (lld)--(ld)--(d)--(rd)--(rrd) (u)--(d);
\draw (-1.5,-0.4)--(1.5,0.8) (-1.5,-0.6)--(1.5,0.6) (-1.5,-0.8)--(1.5,0.4);
\fill(lu)circle(1mm); \fill(ru)circle(1mm); \fill(ld)circle(1mm); \fill(rd)circle(1mm);
\fill(u)circle(1mm); \fill(d)circle(1mm);
\draw[fill=white](-1.4,0)circle(1mm); \draw[fill=white](1.4,0)circle(1mm);
\end{tikzpicture}
\end{center}
   \caption{An example of gluing process}
   \label{Fig gluing}
\end{figure}

\begin{example}\label{example gluing}
We consider the $\circ\bullet$-marked surface $(S,M)$ and the $\bullet$-dissection $D=\{1,2\}$ in Section \ref{Secexample}(3). The $\bullet$-dissection $D$ has exactly one polygon $\triangle_v$ and we can naturally embed $\triangle_v$ into the $\bullet$-dissection $D_2$ of $(S_2(4),S_2(4))$, where we also denote one side $i$ of $\triangle_v$ by $i'$ for $i \in \{1,2\}$ as follows:
\[
\triangle_v=
\begin{tikzpicture}[baseline=-1mm,scale=0.7]
\coordinate(lu)at(-1,1); \coordinate(ru)at(1,1); \coordinate(ld)at(-1,-1); \coordinate(rd)at(1,-1);
\draw (lu)--node[above]{$1$}(ru)--node[fill=white,inner sep=2]{$2'=2$}(rd)--node[below]{$1'=1$}(ld)--node[fill=white,inner sep=2]{$2$}(lu);
\draw[fill=black](lu)circle(1.4mm); \draw[fill=black](ru)circle(1.4mm);
\draw[fill=black](ld)circle(1.4mm); \draw[fill=black](rd)circle(1.4mm); 
\draw[fill=white](0,0)circle(1.4mm);
\end{tikzpicture}
\hspace{5mm}\hookrightarrow\hspace{5mm}
D_2=
\begin{tikzpicture}[baseline=-1mm,scale=0.7]
\coordinate(lu)at(-1,1); \coordinate(ru)at(1,1); \coordinate(ld)at(-1,-1); \coordinate(rd)at(1,-1);
\draw(0,0)circle(1.4);
\draw (lu)--node[fill=white,inner sep=2]{$1$}(ru)--node[fill=white,inner sep=2]{$2'$}(rd)--node[fill=white,inner sep=2]{$1'$}(ld)--node[fill=white,inner sep=2]{$2$}(lu);
\draw[fill=black](lu)circle(1.4mm); \draw[fill=black](ru)circle(1.4mm);
\draw[fill=black](ld)circle(1.4mm); \draw[fill=black](rd)circle(1.4mm); 
\draw[fill=white](0,0)circle(1.4mm);
\end{tikzpicture}
\]

For an integer vector $(g_1,g_2):=(3,-1) \in \mathbb{Z}^{\vert D \vert}$, we construct a $D$-lamination whose $g$-vector is $(3,-1)$ by our method in a proof of Theorem \ref{lattice points}. We consider the corresponding vector $(g_1,g_2,g_{1'},g_{2'})=(3,-1,3,-1)\in \mathbb{Z}^{\vert D_2 \vert}$, and there is a unique $D_2$-lamination $\mathcal{X}_v$ with $g(\mathcal{X}_v)=(3,-1,3,-1)$ given by the left picture of Figure \ref{Fig gluing example}. By gluing curves of $\mathcal{X}_v$ along $1=1'$ and $2=2'$, we obtain a $D$-lamination $\mathcal{X}$ of $(S,M)$ satisfying $g(\mathcal{X})=(3,1)$ and described in the right picture of Figure \ref{Fig gluing example}.

\begin{figure}[htp]
\[
\mathcal{X}_v=
\begin{tikzpicture}[baseline=-1mm,scale=0.7]
\coordinate(lu)at(-1,1); \coordinate(ru)at(1,1); \coordinate(ld)at(-1,-1); \coordinate(rd)at(1,-1);
\draw(0,0)circle(1.4);
\draw[dotted] (lu)--(ru)--(rd)--(ld)--(lu);
\draw[blue](75:1.4)--(195:1.4);
\draw[blue](67:1.4)..controls(-0.3,0.5)and(-0.45,0.3)..(-0.45,0)arc(-180:10:0.35);
\draw[blue](60:1.4)..controls(-0.2,0.4)and(-0.35,0.2)..(-0.35,0)arc(-180:10:0.27);
\draw[blue](-105:1.4)--(15:1.4);
\draw[blue](-113:1.4)..controls(0.3,-0.5)and(0.45,-0.3)..(0.45,0)arc(0:190:0.35);
\draw[blue](-120:1.4)..controls(0.2,-0.4)and(0.35,-0.2)..(0.35,0)arc(0:190:0.27);
\draw[fill=black](lu)circle(1.4mm); \draw[fill=black](ru)circle(1.4mm);
\draw[fill=black](ld)circle(1.4mm); \draw[fill=black](rd)circle(1.4mm); 
\draw[fill=white](0,0)circle(1.4mm); \draw[fill=white](0,1.4)circle(1.4mm);
\draw[fill=white](0,-1.4)circle(1.4mm); \draw[fill=white](1.4,0)circle(1.4mm);
\draw[fill=white](-1.4,0)circle(1.4mm);
\end{tikzpicture}
\hspace{10mm} \mathcal{X}=
\begin{tikzpicture}[baseline=-1mm,scale=0.7]
\coordinate(b)at(-0.5,0.5); \coordinate(w)at(0.5,-0.5);
\draw (-1.5,-1.5) rectangle (1.5,1.5);
\draw[fill=black](b)circle(1.4mm); \draw[fill=white](w)circle(1.4mm);
\draw[dotted](-1.5,0.5)--(1.5,0.5); \draw[dotted](-0.5,-1.5)--(-0.5,1.5);
\draw[blue](0.5,1.5)..controls(0,0.5)and(-1,-0.6)..(-1.5,-0.9);
\draw[blue](0.7,1.5)..controls(0.2,0.7)and(0,0)..(0,-0.5)arc(-180:10:0.38);
\draw[blue](0.9,1.5)..controls(0.5,1)and(0.1,0)..(0.1,-0.5)arc(-180:10:0.3);
\draw[blue](1.5,-0.9)--(0.9,-1.5);
\draw[blue](0.5, -1.5)..controls(0.8,-1)and(0.9,-0.7)..(0.9,-0.5)arc(0:190:0.3);
\draw[blue](0.7, -1.5)..controls(0.9,-1)and(1,-0.7)..(1,-0.5)arc(0:190:0.38);
\end{tikzpicture} 
\]
   \caption{A $D_2$-lamination $\mathcal{X}_v$ and a $D$-lamination $\mathcal{X}$ in Example \ref{example gluing}}
   \label{Fig gluing example}
\end{figure}
\end{example}

\section{Positive position and Dehn twists}\label{SecDehn}

In this section, we fix a $\bullet$-dissection $D$ of $(S,M)$ and make preparations for proving Theorem \ref{dense}.
The proof of Theorem \ref{dense} appears in the next section.

\subsection{Dehn twist along a closed $D$-laminate}
{\color{black}\ \hspace{-3mm}}
Let $V$ be a regular neighborhood of a simple closed curve $\ell$, which is homeomorphic to an annulus $S^1 \times [0,1]$. The \emph{Dehn twist $\mathsf{T}_{\ell}$ along $\ell$} is a self-homeomorphism of $S$ which is given by sending $(s,t)$ to $(\exp(-2\pi i t)s,t)$ on $V \simeq S^1 \times [0,1]$ and by fixing all points on $S \setminus V$ (see e.g. \cite{M}). In particular, $\mathsf{T}_{\ell}$ is oriented as follows:

\[
\begin{tikzpicture}[baseline=0mm]
 \coordinate (0) at (0,0);
 \draw (0,1)--(3,1) (0,-1)--(3,-1); \draw[blue] (0.4,0)--(3.4,0);
 \draw[red] (1.5,1) arc [start angle = 90, end angle = -90, x radius=4mm, y radius=10mm];
 \draw[red,dotted] (1.5,1) arc [start angle = 90, end angle = 270, x radius=4mm, y radius=10mm];
 \draw (0) circle [x radius=4mm, y radius=10mm];
 \draw (3,1) arc [start angle = 90, end angle = -90, x radius=4mm, y radius=10mm];
 \draw[dotted] (3,1) arc [start angle = 90, end angle = 270, x radius=4mm, y radius=10mm];
 \node[red] at(2,0.7) {$\ell$};
\end{tikzpicture}
 \hspace{7mm} \xrightarrow{\mathsf{T}_{\ell}} \hspace{7mm}
\begin{tikzpicture}[baseline=0mm]
 \coordinate (0) at (0,0);
 \draw (0,1)--(3,1) (0,-1)--(3,-1);
 \draw[blue] (0.4,0) .. controls (1.2,0) and (1.2,-1) .. (1.4,-1);
 \draw[blue] (1.6,1) .. controls (1.8,1) and (1.8,0) .. (3.4,0);
 \draw[blue,dotted] (1.4,-1) .. controls (1.6,-1) and (1.2,0.8) .. (1.6,1);
 \draw (0) circle [x radius=4mm, y radius=10mm];
 \draw (3,1) arc [start angle = 90, end angle = -90, x radius=4mm, y radius=10mm];
 \draw[dotted] (3,1) arc [start angle = 90, end angle = 270, x radius=4mm, y radius=10mm];
\end{tikzpicture}
\]\vspace{1mm}

\noindent In general, for a given $D$-laminate $\gamma$, $\mathsf{T}_{\ell}(\gamma)$ is a non-self-intersecting $\circ$-laminate, but it may not be a $D$-laminate because it does not necessarily satisfy the condition $(\ast)$ in Definition \ref{$D$-laminate}(2). We will give a condition that $\mathsf{T}_{\ell}(\gamma)$ becomes a $D$-laminate (see Lemma \ref{Dehn twist at l}).

Let $\gamma$ and $\delta$ be $D$-laminates.
For each intersection point $p$ of $\gamma$ and $\delta$, we can assume that $p$ lies in $S \setminus D$, thus $p \in \triangle_{v}$ for some $v \in M_{\circ}$. 
We set orientations of the segments of $\gamma$ and $\delta$ in $\triangle_v$ such that $v$ lies on the right to them.
We say that $\gamma$ {\it is in positive position for $\delta$} if $\gamma$ and $\delta$ do not intersect or $\gamma$ intersects $\delta$ from right to left at each intersection point (see Figure \ref{Fig pos pos}).

\begin{figure}[htp]
\begin{center}
\begin{tikzpicture}[baseline=0mm,scale=1]
\coordinate(l)at(180:1); \coordinate(lu)at(120:1); \coordinate(ru)at(60:1);
\coordinate(r)at(0:1); \coordinate(ld)at(-120:1); \coordinate(rd)at(-60:1);
\draw(l)--(lu)--(ru)--(r)--(rd)--(ld)--(l);
\fill(l)circle(1mm); \fill(lu)circle(1mm); \fill(ru)circle(1mm);
\fill(r)circle(1mm); \fill(ld)circle(1mm); \fill(rd)circle(1mm);
\draw(0,0)circle(1mm);
\draw(0,0.3)node[above]{$\gamma$}--node[pos=0.3]{\rotatebox{10}{$>$}}($(r)!0.4!(ru)$);
\draw(0,0.3)arc(95:270:2.6mm); \draw(0,-0.22)arc(-90:90:2mm);
\draw($(lu)!0.6!(ru)$)..controls(0.6,0.7)and(0.6,-0.7)..node[right,pos=0.6]{$\delta$}node[pos=0.7]{\rotatebox{-100}{$>$}}($(ld)!0.6!(rd)$);
\end{tikzpicture}
 \hspace{20mm}
\begin{tikzpicture}[baseline=0mm,scale=1]
\coordinate(l)at(180:1); \coordinate(lu)at(120:1); \coordinate(ru)at(60:1);
\coordinate(r)at(0:1); \coordinate(ld)at(-120:1); \coordinate(rd)at(-60:1);
\draw(l)--(lu)--(ru)--(r)--(rd)--(ld)--(l);
\fill(l)circle(1mm); \fill(lu)circle(1mm); \fill(ru)circle(1mm);
\fill(r)circle(1mm); \fill(ld)circle(1mm); \fill(rd)circle(1mm);
\fill[pattern=north east lines](ld)--(-0.8,-1.2)--(0.8,-1.2)--(rd);
\draw[fill=white]($(ld)!0.5!(rd)$)circle(1mm);
\draw($(l)!0.4!(lu)$)--node[below,pos=0.3]{$\gamma$}node[pos=0.5]{$>$}($(r)!0.4!(ru)$);
\draw($(lu)!0.8!(ru)$)--node[left,pos=0.7]{$\delta$}node[pos=0.6]{\rotatebox{-90}{$>$}}($(ld)!0.8!(rd)$);
\end{tikzpicture}
\end{center}
   \caption{A $D$-laminate $\gamma$ is in positive position for a $D$-laminate $\delta$}
   \label{Fig pos pos}
\end{figure}

We observe the case that $\ell:=\delta$ is a closed $D$-laminate and $\gamma$ is in positive position for $\ell$. First, we make $\#(\gamma\cap\ell)$ copies of $\ell$, in particular, there are $(\#(\gamma\cap\ell))^2$ intersection points of $\gamma$ and them. For each intersection point $p$, let $s$ and $t$ be the oriented segments of $\gamma$ and $\ell$ in $\triangle_v$ containing $p$ respectively, as in the previous paragraph. We replace $s$ and $t$ with two oriented segments: one is from the starting point of $s$ to the ending point of $t$, and the other is from the starting point of $t$ to the ending point of $s$, where the starting or ending point of $s$ may be a spiral (see Figure \ref{Fig resolving}). By applying this process for all intersection points, the curve $\mathsf{T}_{\ell}(\gamma)$ is obtained from $\gamma$ and $\ell$ (see Figure \ref{Fig process of Dehn twist}).

\begin{figure}[htp]
\begin{center}
\begin{tikzpicture}[baseline=-1mm,scale=1]
\coordinate(l)at(180:1); \coordinate(lu)at(120:1); \coordinate(ru)at(60:1);
\coordinate(r)at(0:1); \coordinate(ld)at(-120:1); \coordinate(rd)at(-60:1);
\draw(l)--(lu)--(ru)--(r)--(rd)--(ld)--(l);
\fill(l)circle(1mm); \fill(lu)circle(1mm); \fill(ru)circle(1mm);
\fill(r)circle(1mm); \fill(ld)circle(1mm); \fill(rd)circle(1mm);
\draw(0,0)circle(1mm);
\draw(0,0.3)node[above]{$s$}--node[pos=0.3]{\rotatebox{10}{$>$}}($(r)!0.4!(ru)$);
\draw(0,0.3)arc(95:270:2.6mm); \draw(0,-0.22)arc(-90:90:2mm);
\draw($(lu)!0.6!(ru)$)..controls(0.6,0.7)and(0.6,-0.7)..node[right,pos=0.6]{$t$}node[pos=0.7]{\rotatebox{-100}{$>$}}($(ld)!0.6!(rd)$);
\end{tikzpicture}
 \hspace{10mm}$\Longrightarrow$\hspace{10mm}
\begin{tikzpicture}[baseline=-1mm,scale=1]
\coordinate(l)at(180:1); \coordinate(lu)at(120:1); \coordinate(ru)at(60:1);
\coordinate(r)at(0:1); \coordinate(ld)at(-120:1); \coordinate(rd)at(-60:1);
\draw(l)--(lu)--(ru)--(r)--(rd)--(ld)--(l);
\fill(l)circle(1mm); \fill(lu)circle(1mm); \fill(ru)circle(1mm);
\fill(r)circle(1mm); \fill(ld)circle(1mm); \fill(rd)circle(1mm);
\draw(0,0)circle(1mm);
\draw(0,0.3)..controls(0.6,0.4)and(0.6,-0.3)..node[pos=0.7]{\rotatebox{-110}{$>$}}($(ld)!0.6!(rd)$);
\draw(0,0.3)arc(95:270:2.6mm); \draw(0,-0.22)arc(-90:90:2mm);
\draw($(lu)!0.6!(ru)$)..controls(60:0.6)and(30:0.5)..node[pos=0.4]{\rotatebox{-55}{$>$}}($(r)!0.4!(ru)$);
\end{tikzpicture}
\end{center}
   \caption{Resolving an intersection point of two oriented segments $s$ and $t$}
   \label{Fig resolving}
\end{figure}

\begin{figure}[htp]
\[
\renewcommand{\arraystretch}{2}
\setlength{\tabcolsep}{3mm}
\begin{tabular}{ccc}
\begin{tikzpicture}[baseline=0mm]
 \draw(0,1)--node[pos=0.1]{$>$}(3,1) (0,-1)--node[pos=0.1]{$>$}(3,-1);
 \draw[blue](0,0.5)--(3,0.5) (0,0)--(3,0) (0,-0.5)--(3,-0.5);
 \draw[red](1.5,1)--(1.5,-1); \node[red]at(1.7,0.75){$\ell$};
\end{tikzpicture}
 & $\xrightarrow{\mathsf{T}_{\ell}}$ &
\begin{tikzpicture}[baseline=0mm]
 \draw(0,1)--node[pos=0.1]{$>$}(3,1) (0,-1)--node[pos=0.1]{$>$}(3,-1);
 \draw[blue](0,0.5)..controls(1.4,0.5)..(1.6,-1); \draw[blue](1.6,1)..controls(1.8,0.5)..(3,0.5);
 \draw[blue](0,0)..controls(1.3,0)..(1.5,-1); \draw[blue](1.5,1)..controls(1.7,0)..(3,0);
 \draw[blue](0,-0.5)..controls(1.2,-0.5)..(1.4,-1); \draw[blue](1.4,1)..controls(1.6,-0.5)..(3,-0.5);
\end{tikzpicture}
\\
$\downarrow$ && $\rotatebox{90}{$=$}$
\\
\begin{tikzpicture}[baseline=0mm]
 \draw(0,1)--node[pos=0.1]{$>$}(3,1) (0,-1)--node[pos=0.1]{$>$}(3,-1);
 \draw[blue](0,0.5)--(3,0.5) (0,0)--(3,0) (0,-0.5)--(3,-0.5);
 \draw[red](1.5,1)--(1.5,-1); \draw[red](1,1)--(1,-1); \draw[red](2,1)--(2,-1);
\end{tikzpicture}
 & $\xrightarrow{\rm resolving}$ &
\begin{tikzpicture}[baseline=0mm]
 \draw(0,1)--node[pos=0.1]{$>$}(3,1) (0,-1)--node[pos=0.1]{$>$}(3,-1);
 \draw[blue](0,0.5)--(0.75,0.5) (0,0)--(0.75,0) (0,-0.5)--(0.75,-0.5);
 \draw[blue](3,0.5)--(2.25,0.5) (3,0)--(2.25,0) (3,-0.5)--(2.25,-0.5);
 \draw[blue](1,1)--(1,0.75) (1.5,1)--(1.5,0.75) (2,1)--(2,0.75);
 \draw[blue](1,-1)--(1,-0.75) (1.5,-1)--(1.5,-0.75) (2,-1)--(2,-0.75);
 \draw[blue](2,0.75)arc(180:270:0.25);
 \draw[blue](1.5,0.75)arc(180:270:0.25)arc(90:0:0.25)arc(180:270:0.25);
 \draw[blue](1,0.75)arc(180:270:0.25)arc(90:0:0.25)arc(180:270:0.25)arc(90:0:0.25)arc(180:270:0.25);
 \draw[blue](0.75,0.5)arc(90:0:0.25)arc(180:270:0.25)arc(90:0:0.25)arc(180:270:0.25)arc(90:0:0.25);
 \draw[blue](0.75,0)arc(90:0:0.25)arc(180:270:0.25)arc(90:0:0.25);
 \draw[blue](0.75,-0.5)arc(90:0:0.25);
\end{tikzpicture}
\end{tabular}
\]
   \caption{A process of the Dehn twist along a simple closed curve $\ell$, where we identify the top and bottom horizontal lines in each figure}
   \label{Fig process of Dehn twist}
\end{figure}

\begin{lemma} \label{Dehn twist at l}
Let $\ell$ be a closed $D$-laminate and $\gamma$ a non-closed $D$-laminate which is in positive position for $\ell$. Then
\begin{enumerate}
\item[(1)] $\mathsf{T}_{\ell}(\gamma)$ is a non-closed $D$-laminate;
\item[(2)] $g(\mathsf{T}_{\ell}(\gamma))= g(\gamma) + \#(\gamma\cap\ell)g(\ell)$;
\item[(3)] if a $D$-laminate $\gamma'$ does not intersect $\ell$, then 
\[
\#(\gamma' \cap \gamma) = \#(\gamma' \cap \mathsf{T}_{\ell}(\gamma)).
\]
\end{enumerate}
\end{lemma}

\begin{proof}
Since $\gamma$ is a $D$-laminate, $\mathsf{T}_{\ell}(\gamma)$ is a non-self-intersecting $\circ$-laminate. Moreover, it satisfies the condition $(\ast)$ in Definition \ref{$D$-laminate}(2) except for a neighborhood of each intersection of $\gamma$ and $\ell$ since both $\gamma$ and $\ell$ satisfy $(\ast)$. On a neighborhood of each intersection of $\gamma$ and $\ell$, it follows that $(\ast)$ holds from a behavior of resolving the intersection as in Figure \ref{Fig resolving}. Therefore, the assertion (1) holds. Considering the numbers of intersections with each $d \in D$, we get the equation of (2) by the above observation. The assertion (3) also follows from the above observation.
\end{proof}

In the situation of Lemma \ref{Dehn twist at l}, we can repeat the Dehn twist $\mathsf{T}_{\ell}$. 
Moreover, Lemma \ref{Dehn twist at l} is generalized for $D$-laminations.
For closed curves $\ell_1,\ldots,\ell_k$ and $m_1, \ldots, m_k \in \mathbb{Z}_{\ge 0}$, we write
\[
 \mathsf{T}_{(\ell_1,\ldots,\ell_k)}^{(m_1,\ldots,m_k)}:=\mathsf{T}_{\ell_1}^{m_1}\cdots\mathsf{T}_{\ell_k}^{m_k}.
\]
 Note that if $\ell_1,\ldots,\ell_k$ are pairwise non-intersecting, then all $\mathsf{T}_{\ell_i}$ are commutative.

\begin{proposition} \label{Dehn twists at l_1,..., l_k}
Let $\{\ell_1,\ldots,\ell_k\}$ be a $D$-lamination consisting only of closed $D$-laminates and $\{\gamma_1,\ldots,\gamma_h\}$ a $D$-lamination consisting only of non-closed $D$-laminates which are in positive position for any $\ell_i$. Then for any $m_1,\ldots,m_k \in \mathbb{Z}_{\ge 0}$ and $\mathsf{T}:=\mathsf{T}^{(m_1,\ldots,m_k)}_{(\ell_1,\ldots,\ell_k)}$,
\begin{enumerate}
\item[(1)] $\{\mathsf{T}(\gamma_1), \ldots, \mathsf{T}(\gamma_h)\}$ is a $D$-lamination consisting only of non-closed $D$-laminates;
\item[(2)] we have the equality 
\[
 \sum_{i=1}^h g(\mathsf{T}(\gamma_i)) = \sum_{i=1}^h g(\gamma_i) + \sum_{i=1}^h\sum_{j=1}^k m_j \#(\gamma_i \cap \ell_j)g(\ell_j).
\]
\end{enumerate}
\end{proposition}

\begin{proof} 
(1) Let $\mathcal{X}:=\{\gamma_1,\ldots, \gamma_h\}$ and $\mathcal{Y}:=\{\ell_1,\ldots, \ell_k\}$. For any $\gamma \in \mathcal{X}$ and $\ell \in \mathcal{Y}$, by Lemma \ref{Dehn twist at l}(1), $\mathsf{T}_{\ell}(\gamma)$ is a non-closed $D$-laminate. 
Lemma \ref{Dehn twist at l}(3) says that $\mathsf{T}_{\ell}(\gamma) \cap \ell'$ is naturally identified with $\gamma \cap \ell'$ for any $\ell' \in \mathcal{Y}$. 
In particular, $\mathsf{T}_{\ell}(\gamma)$ is also in positive position for $\ell'$, thus $\mathsf{T}_{\ell'}\mathsf{T}_{\ell}(\gamma)$ is a non-closed $D$-laminate. 
Repeating this process, $\mathsf{T}(\gamma)$ is a non-closed $D$-laminate. 
Since $\mathsf{T}(\gamma)$ and $\mathsf{T}(\gamma')$ do not intersect for any $\gamma, \gamma' \in \mathcal{X}$, the assertion holds.

(2) The equality is calculated from Lemma \ref{Dehn twist at l}(2) since Lemma \ref{Dehn twist at l}(3) says that the number of all intersection points of $\mathcal{X}$ and $\mathcal{Y}$ is not changed by the Dehn twists.
\end{proof}

\subsection{Non-closed $D$-laminate $\ell^d$ for a closed $D$-laminate $\ell$} \label{construction}

Let $\mathcal{X}$ be a $D$-lamination consisting only of non-closed $D$-laminates. 
We assume that there is a closed $D$-laminate $\ell$ such that $\mathcal{X} \sqcup \{\ell\}$ is a $D$-lamination.
By the definition of closed $D$-laminates, there exists $d \in D$ such that $g(\ell)_d >0$. 
From now, we construct a non-closed $D$-laminate $\ell^d$ such that 
\begin{enumerate}
\item[(a)] $\ell^d$ is a non-closed $D$-laminate which is compatible with any $D$-laminate of $\mathcal{X}$;
\item[(b)] $\ell^d$ intersects with $\ell$ so that $\ell^d$ is in positive position for $\ell$.
\end{enumerate} 
It plays an important role to prove Theorem \ref{dense} in the next section.

First, for $d \in D$, we define a $D$-laminate $d^{\ast}_+$ (resp., $d^{\ast}_-$) as follows (see Figure \ref{Fig hat check}):
\begin{enumerate}
 \item[$\bullet$] $d^{\ast}_+$ (resp., $d^{\ast}_-$) is a laminate running along $d^{\ast}$ in a small neighborhood of it;
 \item[$\bullet$] If $d^{\ast}$ has an endpoint $v \in M_{\circ}$ on a component $C$ of $\partial S$, 
then the corresponding endpoint of $d^{\ast}_+$ (resp., $d^{\ast}_-$) is located near $v$ on $C$ in the counterclockwise (resp., clockwise) direction;
 \item[$\bullet$] If $d^{\ast}$ has an endpoint at a puncture $p \in M_{\circ}$, 
then the corresponding end of $d^{\ast}_+$ (resp., $d^{\ast}_-$) is a spiral around $p$ counterclockwise (resp., clockwise).
\end{enumerate}

\begin{figure}[htp]
\begin{center}
\begin{tikzpicture}[baseline=0mm]
\coordinate(u)at(0,0.8); \coordinate(d)at(0,-0.8);
\coordinate(llu)at(-1.5,1); \coordinate(lu)at(-0.6,1); \coordinate(rru)at(1.5,1); \coordinate(ru)at(0.6,1);
\coordinate(lld)at(-1.5,-1); \coordinate(ld)at(-0.6,-1); \coordinate(rrd)at(1.5,-1); \coordinate(rd)at(0.6,-1);
\draw(llu)--(lu)--(u)--(ru)--(rru) (lld)--(ld)--(d)--(rd)--(rrd) (u)--(d) (llu)--(lld);
\draw[blue](-1.5,0.5)..controls(-0.5,0.5)and(0.5,-0.3)..(1,-0.3);
\draw[blue](1,-0.3)arc(-90:90:2.6mm)node[above,pos=0.8]{$d^{\ast}_+$};
\draw[blue](1,0.22)arc(90:270:2mm);
\draw(-1.5,0)--node[below,pos=0.3]{$d^{\ast}$}(1,0);
\fill[pattern=north east lines](llu)--(-1.8,1.2)--(-1.8,-1.2)--(lld);
\fill(lu)circle(1mm); \fill(ru)circle(1mm); \fill(ld)circle(1mm); \fill(rd)circle(1mm);
\fill(u)circle(1mm); \fill(d)circle(1mm); \fill(llu)circle(1mm); \fill(lld)circle(1mm);
\draw[fill=white](-1.5,0)circle(1mm); \draw[fill=white](1,0)circle(1mm);
\end{tikzpicture}
 \hspace{10mm}
\begin{tikzpicture}[baseline=0mm]
\coordinate(u)at(0,0.8); \coordinate(d)at(0,-0.8);
\coordinate(llu)at(-1.5,1); \coordinate(lu)at(-0.6,1); \coordinate(rru)at(1.5,1); \coordinate(ru)at(0.6,1);
\coordinate(lld)at(-1.5,-1); \coordinate(ld)at(-0.6,-1); \coordinate(rrd)at(1.5,-1); \coordinate(rd)at(0.6,-1);
\draw(llu)--(lu)--(u)--(ru)--(rru) (lld)--(ld)--(d)--(rd)--(rrd) (u)--(d) (llu)--(lld);
\draw[blue](-1.5,-0.5)..controls(-0.5,-0.5)and(0.5,0.3)..(1,0.3);
\draw[blue](1,0.3)arc(90:-90:2.6mm)node[above,pos=0.1]{$d^{\ast}_-$};
\draw[blue](1,-0.22)arc(-90:-270:2mm);
\draw(-1.5,0)--node[above,pos=0.3]{$d^{\ast}$}(1,0);
\fill[pattern=north east lines](llu)--(-1.8,1.2)--(-1.8,-1.2)--(lld);
\fill(lu)circle(1mm); \fill(ru)circle(1mm); \fill(ld)circle(1mm); \fill(rd)circle(1mm);
\fill(u)circle(1mm); \fill(d)circle(1mm); \fill(llu)circle(1mm); \fill(lld)circle(1mm);
\draw[fill=white](-1.5,0)circle(1mm); \draw[fill=white](1,0)circle(1mm);
\end{tikzpicture}
\end{center}
   \caption{Two $D$-laminates $d^{\ast}_+$ and $d^{\ast}_-$}
   \label{Fig hat check}
\end{figure}

\noindent That is, $g(d^{\ast}_+)_{e}=\delta_{ed}$ (resp., $g(d^{\ast}_-)_{e}=-\delta_{ed}$) for $e\in D$, where $\delta$ is the Kronecker delta. 

On this notation, $g(\ell)_d>0$ implies $\ell \cap d^{\ast}_+ \neq \emptyset$ and $d^{\ast}_+$ is in positive position for $\ell$. 
Without loss of generality, we can assume that $p \in \ell \cap d^{\ast}_+$ lie on $d$ as in the left diagram of Figure \ref{Fig ell v} .

Second, for each endpoint $v$ of $d^{\ast}$, we define a curve $\ell_v$ of $S$ as follows: Consider the segment $\alpha:=d^{\ast}_+\cap\triangle_{v}$.
\begin{enumerate}
\item[$\bullet$] If $\alpha$ intersects none of $\mathcal{X}$, then let $\ell_v:=\alpha$ (see the center diagram of Figure \ref{Fig ell v});
\item[$\bullet$] Otherwise, we take a non-closed $D$-laminate $\gamma$ of $\mathcal{X}$ such that it gives the nearest intersection point of $\alpha$ and $\mathcal{X}$ from $p$, in which case let $p_v\in \alpha \cap \gamma$ be this nearest intersection point.
We denote by $q$ an endpoint of a connected segment in $\gamma \cap \triangle_v$ containing $p_v$ such that the intersection point $q$ is negative.
Then $\ell_v$ is a curve obtained by gluing the following two curves at $p_v$ (see the right diagram of Figure \ref{Fig ell v}):
\begin{enumerate}
 \item[(i)] a segment of $\alpha$ between $p$ and $p_v$;
 \item[(ii)] a segment of $\gamma$ obtained by cutting $\gamma$ at $p_v$, that contains $q$.
\end{enumerate}
\end{enumerate}
\begin{figure}[htp]
\begin{center}
\begin{tikzpicture}[baseline=0mm,scale=1.4]
\coordinate(u)at(0,0.8); \coordinate(d)at(0,-0.8);
\coordinate(llu)at(-1.2,1); \coordinate(lu)at(-0.6,1); \coordinate(rru)at(1.2,1); \coordinate(ru)at(0.6,1);
\coordinate(lld)at(-1.2,-1); \coordinate(ld)at(-0.6,-1); \coordinate(rrd)at(1.2,-1); \coordinate(rd)at(0.6,-1);
\draw(llu)--(lu)--(u)--(ru)--(rru) (lld)--(ld)--(d)--(rd)--(rrd) (u)--node[left,pos=0.7]{$d$}(d);
\draw(-1,0.3)..controls(-0.5,0.3)and(0.5,-0.3)..(1,-0.3);
\draw(-1.2,0.8)..controls(-0.5,0.5)and(0.5,-0.5)..node[above,pos=0.3]{$\ell$}(1.2,-0.8);
\draw[dotted](1,-0.3)arc(-90:90:2.6mm)node[above,pos=0.8]{$d^{\ast}_+$};
\draw[dotted](1,0.22)arc(90:270:2mm);
\draw[dotted](-1,0.3)arc(90:270:2.6mm); \draw[dotted](-1,-0.22)arc(-90:50:2mm);
\fill(lu)circle(0.7mm); \fill(ru)circle(0.7mm); \fill(ld)circle(0.7mm); \fill(rd)circle(0.7mm);
\fill(u)circle(0.7mm); \fill(d)circle(0.7mm);
\draw[fill=white](-1,0)circle(0.7mm); \draw[fill=white](1,0)circle(0.7mm);
\node at(0.15,0.1){$p$};
\end{tikzpicture}
 \hspace{10mm}
\begin{tikzpicture}[baseline=0mm,scale=1.4]
\coordinate(u)at(0,0.8); \coordinate(d)at(0,-0.8);
\coordinate(llu)at(-1.2,1); \coordinate(lu)at(-0.6,1);
\coordinate(lld)at(-1.2,-1); \coordinate(ld)at(-0.6,-1);
\draw(llu)--(lu)--(u)--(d)--(ld)--(lld) (u)--(d);
\draw[blue](-1,0.3)..controls(-0.8,0.3)and(-0.3,0.1)..(0,-0);
\draw[blue,dotted](-1,0.3)arc(90:270:2.6mm)node[below]{$\ell_v=\alpha$}; 
\draw[blue,dotted](-1,-0.22)arc(-90:50:2mm);
\fill(lu)circle(0.7mm); \fill(ld)circle(0.7mm);
\fill(u)circle(0.7mm); \fill(d)circle(0.7mm);
\draw[fill=white](-1,0)circle(0.7mm);
\node at(-1,0.17){$v$}; \node at(0.15,0){$p$};
\end{tikzpicture}
 \hspace{10mm}
\begin{tikzpicture}[baseline=0mm,scale=1.4]
\coordinate(u)at(0,0.8); \coordinate(d)at(0,-0.8);
\coordinate(llu)at(-1.2,1); \coordinate(lu)at(-0.6,1);
\coordinate(lld)at(-1.2,-1); \coordinate(ld)at(-0.6,-1);
\draw(llu)--(lu)--(u)--(d)--(ld)--(lld) (u)--(d) (llu)--(-1.6,0.3);
\draw(-1,0.3)..controls(-0.8,0.3)and(-0.3,0.1)..(0,0);
\draw[dotted](-1,0.3)arc(90:270:2.6mm)node[below]{$\alpha$}; 
\draw[dotted](-1,-0.22)arc(-90:50:2mm);
\draw($(llu)!0.3!(-1.6,0.3)$)node[left]{$q$}..controls(-0.5,0.3)and(-0.4,0)..node[left,pos=0.8]{$\gamma$}($(ld)!0.7!(d)$);
\draw[blue]($(llu)!0.3!(-1.6,0.3)$)..controls(-0.7,0.5)and(-0.1,0.05)..node[above,pos=0.5]{$\ell_v$}(0,0);
\fill(lu)circle(0.7mm); \fill(ld)circle(0.7mm);
\fill(u)circle(0.7mm); \fill(d)circle(0.7mm); \fill(llu)circle(0.7mm);
\draw[fill=white](-1,0)circle(0.7mm);
\node at(-1,0.17){$v$}; \node at(0.15,0){$p$}; \node at(-0.63,0.05){$p_v$};
\end{tikzpicture}
\end{center}
   \caption{A closed $D$-laminate $\ell$ and $d \in D$ with $g(\ell)_d >0$ (left), constructions of a curve $\ell_v$ (center, right)}
   \label{Fig ell v}
\end{figure}

Finally, we define $\ell^d$ as a curve obtained by gluing $\ell_v$ and $\ell_{v'}$ at $p$ for endpoints $v$ and $v'$ of $d^{\ast}$.
A segment of $\ell^d$ between $p_v$ and $p_{v'}$ is called its {\it center segment}, where $p_v$ is a point on $\ell_v$ sufficiently close to $v$ if $\ell_v=\alpha$.
It follows from the construction that $\ell^d$ satisfies (a) and (b) above.
Moreover, (b) is generalized as follows.

\begin{lemma}\label{pp for elld}
In the above situations, if a $D$-laminate $\gamma$ is compatible with $\mathcal{X}\cup\{\ell\}$, then $\ell^d$ is in positive position for $\gamma$.
\end{lemma}

\begin{proof}
If $\gamma$ intersects $\ell^{d}$, then all the intersection points lie on the center segment of $\ell^{d}$, thus the assertion holds.
\end{proof}

\section{Proof of Theorem \ref{dense}}\label{Secproofdense}

In this section, we prove Theorem \ref{dense}. The idea of its proof comes from \cite{Yurikusa20}. Fix a $\bullet$-dissection $D$ of $(S,M)$.
Let $g\in\mathbb{Z}^{\vert D\vert }$. By Theorem \ref{lattice points}, there is a $D$-lamination $\mathcal{X}$ such that $g=g(\mathcal{X})=\sum_{\gamma\in\mathcal{X}}g(\gamma)$.
It is sufficient to construct $D$-laminations $\{\mathcal{X}_m\}_{m \in \mathbb{Z}_{\ge 0}}$ consisting only of non-closed $D$-laminates such that
\[
 \mathcal{X}^{\rm nc}\subseteq\mathcal{X}_m\ \text{ and }\ g \in \overline{\bigcup_{m \in \mathbb{Z}_{\ge 0}}C(\mathcal{X}_m)}.
\]
where $\mathcal{X}=\mathcal{X}^{\rm nc}\sqcup\mathcal{X}^{\rm cl}$ is a decomposition such that $\mathcal{X}^{\rm nc}$ (resp., $\mathcal{X}^{\rm cl}$) consists of all non-closed $D$-laminates (resp., closed $D$-laminates) in $\mathcal{X}$.

If $\mathcal{X}^{\rm cl}=\emptyset$, then a family of $\mathcal{X}_m:=\mathcal{X}$ for all $m\in \mathbb{Z}_{>0}$ is the desired one. Assume that $\mathcal{X}^{\rm cl}$ is non-empty.
For $\ell_1 \in \mathcal{X}^{\rm cl}$ and $d_1 \in D$ with $g(\ell_1)_{d_1}>0$, we obtain a non-closed $D$-laminate $\ell_1^{d_1}$ by the construction of Section \ref{construction} for $\mathcal{X}^{\rm nc}$.
By Lemma \ref{pp for elld}, $\ell_1^{d_1}$ is in positive position for every $\ell \in \mathcal{X}^{\rm cl}$.

If the set $\{\ell \in \mathcal{X}^{\rm cl} \mid \ell \cap \ell_1^{d_1}=\emptyset\}$ is non-empty, then we take $\ell_2 \in \{\ell \in \mathcal{X}^{\rm cl} \mid \ell \cap \ell_1^{d_1}=\emptyset\}$ and $d_2 \in D$ with $g(\ell_2)_{d_2}>0$.
By the construction of Section \ref{construction} for $\mathcal{X}^{\rm nc} \sqcup \{\ell_1^{d_1}\}$, we obtain a non-closed $D$-laminate $\ell_2^{d_2}$.
Notice that $\ell_2^{d_2}$ consists of some of segments of $D$-laminates in $\mathcal{X}^{\rm nc}$, the center segment of $\ell_2^{d_2}$ and one of $\ell_1^{d_1}$, where the third type may not appear. 
In the same way as Lemma \ref{pp for elld}, we can show that $\ell_2^{d_2}$ is in positive position for every $\ell \in \mathcal{X}^{\rm cl}$. Repeating this process, we finally get an integer $h \in \{1,\ldots,k=\vert \mathcal{X}^{\rm cl}\vert \}$ and non-closed $D$-laminates $\ell_1^{d_1},\ldots,\ell_h^{d_h}$ such that
\[
 \{\ell \in \mathcal{X}^{\rm cl} \mid \ell \cap \ell_1^{d_1} = \cdots = \ell \cap \ell_h^{d_h} = \emptyset\} = \emptyset.
\]
 Moreover, our construction provides the following properties:
\begin{enumerate}
\item[$\bullet$] $\mathcal{X}^{\rm nc} \cup \{\ell_1^{d_1},\ldots,\ell_h^{d_h}\}$ is a $D$-lamination consisting only of non-closed $D$-laminates;
\item[$\bullet$] For $i \in \{1,\ldots,h\}$, $\ell_i^{d_i}$ is in positive position for every $\ell \in \mathcal{X}^{\rm cl}$.  
\end{enumerate}

 We set $\mathcal{X}^{\rm cl} = \{\ell_1,\ldots,\ell_h\}\sqcup\{\ell_{h+1},\ldots,\ell_k\}$ and fix the notations $n_j^{(i)} := \#(\ell_i^{d_i} \cap \ell_j)$, $n_j := \sum_{i=1}^h n_j^{(i)}$, and $N := n_1 \cdots n_k$. Set
\[
 \mathsf{T} := \mathsf{T}_{(\ell_1,\ldots,\ell_k)}^{(\frac{N}{n_1},\ldots,\frac{N}{n_k})}.
\]
 By Proposition \ref{Dehn twists at l_1,..., l_k}, 
$\mathsf{T}^m(\ell_i^{d_i})$ are non-closed $D$-laminates for all $m \in \mathbb{Z}_{\geq 0}$ and $i \in \{1,\ldots,h\}$, and we get the equalities
{\setlength\arraycolsep{0.5mm}
\begin{eqnarray*}
 \sum_{i=1}^h g(\mathsf{T}^m(\ell_i^{d_i})) &=& \sum_{i=1}^h g(\ell_i^{d_i}) + m\sum_{i=1}^h\sum_{j=1}^k\frac{N}{n_j}n_j^{(i)}g(\ell_j)\\
    &=& \sum_{i=1}^h g(\ell_i^{d_i}) + mN\sum_{j=1}^k g(\ell_j)\\
    &=& \sum_{i=1}^h g(\ell_i^{d_i}) + mN g(\mathcal{X}^{\rm cl}).
\end{eqnarray*}}
 It gives
\[
 \lim_{m \rightarrow \infty} \frac{\sum_{i=1}^h g(\mathsf{T}^m(\ell_i^{d_i}))}{m} = N g(\mathcal{X}^{\rm cl}).
\]
Then the $D$-lamination
\[
 \mathcal{X}_m := \mathcal{X}^{\rm nc} \cup \{\mathsf{T}^m(\ell_i^{d_i})\}_{i=1}^h
\]
is the desired one because
\[
 g = g(\mathcal{X}^{\rm nc}) + g(\mathcal{X}^{\rm cl}) \in C(\mathcal{X}^{\rm nc}) + \overline{\bigcup_{m \in \mathbb{Z}_{\ge 0}}C(\{\mathsf{T}^m(\ell_i^{d_i})\}_{i=1}^h)} \subseteq \overline{\bigcup_{m \in \mathbb{Z}_{\ge 0}}C(\mathcal{X}_m)}.
\]

\section{Representation theory} \label{Secrep}

In this section, we study the algebraic aspects of our results.
We can see their examples in the next section.
Throughout, let $R:=k[[t]]$ be the formal power series ring in one valuable over $k$. 
Then $R$ has the unique maximal ideal $\mathfrak{m}:=tR$. 

\subsection{Two-term silting complexes for module-finite algebras}
Let $A$ be a basic $R$-algebra which is module-finite (i.e., $A$ is finitely generated as an $R$-module). 
We denote by $\proj A$ the category of finitely generated projective right $A$-modules, 
by $\Kb(\proj A)$ the homotopy category of bounded complexes of $\proj A$. In particular, $\Kb(\proj A)$ is an $R$-linear category and $\Hom_{\Kb(\proj A)}(X,Y)$ is a finitely generated $R$-module for any $X,Y\in \Kb(\proj A)$. 

We begin with the following observation. 

\begin{proposition}\label{Kb KS}
   The category $\Kb(\proj A)$ is a Krull-Schmidt triangulated category.
\end{proposition}

\begin{proof}
For any $X \in \Kb(\proj A)$, $E=\End_{\Kb(\proj A)}(X)$ is a module-finite algebra over the complete local noetherian ring $R$. Therefore, $E$ is semiperfect by \cite[p.132]{CR} and hence $\Kb(\proj A)$ is Krull-Schmidt.
\end{proof}   

Now, we study two-term silting theory for a module-finite $R$-algebra $A$. We refer to \cite{AIR,Aihara13,DIJ} for two-term silting theory of finite dimensional algebras, and to \cite{Kimura} for that of module-finite algebras.

\begin{definition}
   Let $P=(P^i, f^i)$ be a complex in $\Kb(\proj A)$.  
   \begin{enumerate}
   \item We say that $P$ is {\it two-term} if $P^i=0$ for any integer $i\neq 0, -1$.
   \item We say that $P$ is {\it presilting} if $\Hom_{\Kb(\proj A)}(P, P[m])=0$ for any positive integer $m$.  
   \item We say that $P$ is {\it silting} if it is presilting and $\thick P = \Kb(\proj A)$, where 
   $\thick P$ is the smallest triangulated subcategory of $\Kb(\proj A)$ which contains $P$ and is closed under taking direct summands.
   \end{enumerate}
\end{definition}

We denote by $\twoips A$ (resp., $\twopresilt A$, $\twosilt A$) the set of isomorphism classes of indecomposable two-term presilting (resp., basic two-term presilting, basic two-term silting) complexes for $A$.
Here, we say that a complex $P$ is {\it basic} if all indecomposable direct summands of $P$ are pairwise non-isomorphic. 
We denote by $\vert P\vert $ the number of non-isomorphic indecomposable direct summands of $P$ and by $\add P$ the category of all direct summands of finite direct sums of copies of $P$.

Let $A=\bigoplus_{i=1}^{\vert A\vert } P_i$ be a decomposition of $A$, where $P_i$ is an indecomposable projective $A$-module. 

\begin{definition}
Let $P=[P^{-1}\overset{f}{\to} P^0]$ be a two-term complex in $\Kb(\proj A)$.
\begin{enumerate}
   \item The {\it $g$-vector} of $P$ is defined by 
   $$g(P):=(m_1-n_1, \ldots, m_{\vert A\vert }-n_{\vert A\vert }) \in \mathbb{Z}^{\vert A\vert }$$ 
   where $m_i$ (resp., $n_i$) is the multiplicity of $P_i$ as indecomposable direct summands of $P^0$ (resp., $P^{-1}$). 
   \item The {\it $g$-vector cone} $C(P)$ is defined to be a cone in $\mathbb{R}^{\vert A\vert }$ spanned by $g$-vectors of all indecomposable direct summands of $P$, that is,  
\[
 C(P) := \biggl\{\sum_{X} a_X g(X) \mid a_X \in \mathbb{R}_{\ge0}\biggr\} \subseteq \mathbb{R}^{\vert A\vert },
\]
where $X$ runs over all indecomposable direct summands of $P$. 
\end{enumerate} 
We denote by $\mathcal{F}(A)$ a collection of $g$-vector cones of all basic two-term presilting complexes for $A$. 
\end{definition} 

The Grothendieck group of $\Kb(\proj A)$, which is naturally identified with the Grothendieck group $K_0(\proj A)$ of $\proj A$, is a free abelian group with basis given by $P_1,\ldots,P_{\vert A\vert }$. The $g$-vector of $P$ is identified with the corresponding element $[P]$ in $K_0(\proj A)$ with respect to the free basis $[P_1],\ldots,[P_{\vert A\vert }]$.
The following are basic properties of two-term presilting complexes 
(see \cite{AiI,Kimura}).  

\begin{proposition} \cite[Theorem 2.27]{AiI}
\begin{enumerate}
\item For any silting complex $P$ for $A$, the set of isomorphism classes of indecomposable direct summands of $P$ gives a free basis of $K_0(\proj A)$. In particular, $\vert P\vert =\vert A\vert $ holds.
\item For any presilting complex $P$ for $A$ which is a direct
summand of a silting complex, we have $\vert P\vert  \leq \vert A\vert $. In
particular, it is silting if and only if $\vert P\vert =\vert A\vert $ holds.
\end{enumerate}
\end{proposition}

\begin{proposition}\label{basis}
Let $P=[P^{-1}\overset{f}{\to} P^0]\in\twopresilt A$. Then the following hold:
\begin{enumerate}
\item\cite[Example 2.3]{Kimura} $P$ is a direct summand of a two-term silting complex. In particular, it is silting if and only if $\vert P\vert =\vert A\vert $.
\item\cite[Lemma 2.9]{Kimura} $\add P^0 \cap \add P^{-1} = 0$.
\end{enumerate}
\end{proposition}

Now, we consider a collection $\mathcal{F}(A)$ of $g$-vector cones, which is shown to be a non-singular fan in Proposition \ref{prop fan}. We focus on a realization of $\mathcal{F}(A)$ in $\mathbb{R}^{\vert A\vert }$. Namely, let 
\begin{equation}\nonumber
    \vert \mathcal{F}(A)\vert  := \bigcup_{P \in \twopresilt A} C(P) = \bigcup_{P \in \twosilt A} C(P) \ \subseteq\  \mathbb{R}^{\vert A\vert }.
\end{equation}
We say that $A$ is \emph{$g$-finite} if the set $\twosilt A$ is finite, in other words, $\mathcal{F}(A)$ is finite. Inspired from the result in Theorem \ref{def:g-finite}, we define the $g$-tameness of algebras as follows.

\begin{definition}
Let $A$ be a module-finite $R$-algebra, where $R:=k[[t]]$. Then, we say that $A$ is \emph{$g$-tame} if it satisfies 
\begin{equation}
    \overline{\vert \mathcal{F}(A)\vert } = \mathbb{R}^{\vert A\vert }, 
\end{equation}
where $\overline{(-)}$ is the closure with respect to the natural topology on $\mathbb{R}^{\vert A\vert }$.
\end{definition}

Let $I$ be a two-sided ideal in $A$ and $B:=A/I$. In particular, $B$ is also module-finite over $R$. The functor $-\otimes_A B \colon \proj A \to \proj B$ induces a triangle functor 
\[
    \overline{(-)}:= -\otimes_A B \colon \Kb(\proj A) \to \Kb(\proj B).
\]

\begin{proposition}\cite[Lemma 2.6]{Kimura}\label{factor-twosilt}
If $P$ is a two-term presilting complex for $A$, then $\overline{P}$ is a two-term presilting complex for $B$.
\end{proposition}

\begin{corollary}\label{factor-g-tame}
If $A$ is g-tame, so is $B$.
\end{corollary}

\begin{proof}
We consider $g$-vectors as elements of the Grothendieck groups. The functor $\overline{(-)}$ naturally induces a continuous surjective map from $K_0(\proj A)\otimes_{\mathbb{Z}}\mathbb{R}$ to $K_0(\proj B)\otimes_{\mathbb{Z}}\mathbb{R}$. By Proposition \ref{factor-twosilt}, the image of $\vert \mathcal{F}(A)\vert $ in $K_0(\proj B)\otimes_{\mathbb{Z}}\mathbb{R}$ is contained in $\vert \mathcal{F}(B)\vert $. On the other hand, since $\vert \mathcal{F}(A)\vert $ is dense in $K_0(\proj A)\otimes_{\mathbb{Z}}\mathbb{R}$, its image is dense in $K_0(\proj B)\otimes_{\mathbb{Z}}\mathbb{R}$. Thus $\vert \mathcal{F}(B)\vert $ is also dense in $K_0(\proj B)\otimes_{\mathbb{Z}}\mathbb{R}$. It finishes a proof.
\end{proof}

Finally, we consider the case that $I\subseteq \mathfrak{m}A$, where we recall that $\mathfrak{m}=tR$ is the unique maximal ideal in $R$. 
Then, a canonical surjection $A\to B=A/I$ provides a decomposition $B=\bigoplus_{i=1}^{\vert A\vert } \overline{P_i}$, where $\overline{P_i}$ is an indecomposable projective $B$-module corresponding to $P_i$ for every $i\in \{1,\ldots,\vert A\vert \}$. Thus, it gives a free basis $[\overline{P_1}],\ldots,[\overline{P_{\vert A\vert }}]$ of the Grothendieck group $K_0(\proj B)$, and $g(P)=g(\overline{P})$ holds for any two-term complex $P$ in $\Kb(\proj A)$ with respect to this basis. We get the following bijections (see also \cite[Corollary 5.21]{DIRRT}, \cite[Theorem 1]{EJR}, and \cite[Lemma 2.6]{vG21})

\begin{proposition}\cite[Proposition 4.2 and Theorem 4.3]{Kimura}\label{central nilpotent}
In the above, we assume that $I$ is a two-sided ideal contained in $\mathfrak{m}A$. 
Then, for two-term complexes $P$ and $Q$ in $\Kb(\proj A)$, the following hold:
\begin{enumerate}
   \item $\Hom_{\Kb(\proj A)}(P,Q[1])=0$ if and only if $\Hom_{\Kb(\proj
B)}(\overline{P},\overline{Q}[1])=0$.
   \item If $P$ is presilting, then $\overline{P}$ is presilting and
$\vert P\vert =\vert \overline{P}\vert $.
\end{enumerate}
Therefore, the correspondence $P \mapsto \overline{P}$ induces
bijections
\[
   \twopresilt A \to \twopresilt B \quad \text{and} \quad \twosilt A \to \twosilt B. 
\]
In particular, we have $\mathcal{F}(A)=\mathcal{F}(B)$. 
\end{proposition}

\begin{proposition}\label{prop fan}
Let $A$ be a module-finite $R$-algebra. 
Then $\mathcal{F}(A)$ is a non-singular fan whose maximal faces correspond to basic two-term silting complexes for $A$.
\end{proposition}

\begin{proof}
By \cite[Corollary 6.7]{DIJ}, the statement holds for any finite dimensional $k$-algebra. In particular, it holds true for $B:=A/\mathfrak{m}A$ since it is finite dimensional. Applying Proposition \ref{central nilpotent} as $I:=\mathfrak{m}A$, we have $\mathcal{F}(A)=\mathcal{F}(B)$ and the assertion for $A$.
\end{proof}

\subsection{Complete special biserial algebras}\label{secgentle}

We define complete special biserial algebras and complete gentle algebras. 
For a given finite connected quiver $Q$, let $\widehat{kQ}$ be the complete path algebra, that is, the completion of a path algebra $kQ$ of $Q$ with respect to $kQ_+$-adic topology, where $kQ_+$ is the arrow ideal. For arrows $\alpha$ and $\beta$, we denote by $s(\alpha)$ and $t(\alpha)$ the starting point and the terminal point of $\alpha$, respectively.
Also we write $\alpha\beta$ for the path from $s(\alpha)$ to $t(\beta)$.

\begin{definition}\label{Def gentle}
Let $Q$ be a finite connected quiver and $I$ an ideal in the path algebra $kQ$ of $Q$.
We say that $\widehat{kQ}/\overline{I}$ is a \textit{complete special biserial algebra}, where $\overline{I}$ is the closure of $I$, if all the following conditions are satisfied: 
\begin{enumerate}
\item[(SB1)] For each vertex $i$ of $Q$, there are at most two arrows starting at $i$ and there are at most two arrows ending at $i$.
\item[(SB2)] For every arrow $\alpha$ in $Q$, there exists at most one arrow $\beta$ such that $t(\alpha)=s(\beta)$ and $\alpha\beta\notin I$. 
\item[(SB3)] For every arrow $\alpha$ in $Q$, there exists at most one arrow $\gamma$ such that $s(\alpha)=t(\gamma)$ and $\gamma\alpha\notin I$. 
\end{enumerate}
It is called \textit{complete gentle algebra} if in addition:
\begin{enumerate}
\item[(SB4)] For every arrow $\alpha$ in $Q$, there exists at most one arrow $\beta$ such that $t(\alpha)=s(\beta)$ and $\alpha\beta\in I$. 
\item[(SB5)] For every arrow $\alpha$ in $Q$, there exists at most one arrow $\gamma$ such that $s(\alpha)=t(\gamma)$ and $\gamma\alpha\in I$. 
\item[(SB6)] The ideal $I$ is generated by paths of length $2$.  
\end{enumerate}
\end{definition}

Here, we do not assume that complete special biserial algebras are finite dimensional. 
Notice that finite dimensional special biserial (resp., gentle) algebras are complete special biserial (resp., gentle) algebras. 
The following observation is basic. 

\begin{proposition} \label{SB-gentle}
Complete special biserial algebras are precisely factor algebras of complete gentle algebras. 
\end{proposition}

\begin{proof}
It is immediate from the definition. 
\end{proof}

Therefore, to prove Theorem \ref{comp SB gtame}, it suffices to show the $g$-tameness of complete gentle algebras by Corollary \ref{factor-g-tame}.

\subsection{Gentle algebras from dissections}
It is known in \cite[Theorem 4.10]{PPP} that complete gentle algebras are precisely obtained by the following construction. 

\begin{definition}\label{def QD}
For a $\bullet$-dissection $D$ of $(S,M)$, 
we define a quiver $Q(D)$ and an ideal $I(D)$ in $kQ(D)$ as follows: 
\begin{enumerate}
\item[$\bullet$] The set of vertices of $Q(D)$ corresponds bijectively with $D$.
\item[$\bullet$] The set of arrows of $Q(D)$ is a disjoint union of sets of arrows in $C_v$ for all $v\in M_{\circ}$ defined as follows (see Figure \ref{Fig dis-quiver}):
\begin{enumerate}
\item[--] If $v$ is a puncture and $d_1, \ldots, d_m \in D$ are sides of $\triangle_v$ in counterclockwise order, then there is a cycle $C_v \colon d_1 \overset{a_1}{\to} d_2 \overset{a_2}{\to} \cdots \overset{a_{m-1}}{\to} d_m \overset{a_m}{\to} d_1$ in $Q(D)$, which is uniquely determined up to cyclic permutation. 
\item[--] If $v$ lies on a boundary segment, and $d_1,\ldots, d_m \in D$ are sides of $\triangle_v$ in counterclockwise order, then there is a path $C_v \colon d_1\overset{a_1}{\to} d_2 \cdots \overset{a_{m-1}}{\to} d_m$ in $Q(D)$.
\end{enumerate}
\item[$\bullet$] $I(D)$ is generated by all paths of length $2$ which is not a sub-path of any $C_v$.
\end{enumerate}
We denote $A(D):=\widehat{kQ(D)}/\overline{I(D)}$. 
\end{definition}

\begin{figure}[htp]
\begin{tikzpicture}[baseline=0mm,scale=1.3]
\coordinate(l)at(-150.5:2); \coordinate(lu)at(170.5:1.8); \coordinate(ru)at(10:1.8);
\coordinate(r)at(-30:2); \coordinate(ld)at(-115:2); \coordinate(rd)at(-65:2);
\draw(ru)--node(3)[fill=white,inner sep=2]{$d_3$}(r)--node(2)[fill=white,inner sep=2]{$d_2$}(rd)--node(1)[fill=white,inner sep=2]{$d_1$}(ld)--node(m)[fill=white,inner sep=2]{$d_m$}(l)--node(m-1)[fill=white,inner sep=2]{$d_{m-1}$}(lu);
\fill(l)circle(0.8mm); \fill(lu)circle(0.8mm); \fill(ru)circle(0.8mm);
\fill(r)circle(0.8mm); \fill(ld)circle(0.8mm); \fill(rd)circle(0.8mm);
\draw(0,-0.1)circle(1mm)node[right=1mm]{$v$};
\draw[->](m-1)..controls(-160:1)and(-140:1)..node[fill=white,inner sep=2]{$a_{m-1}$}(m); 
\draw[->](m)..controls(-125:1)and(-105:1)..node[fill=white,inner sep=2]{$a_m$}(1);
\draw[->](1)..controls(-75:1)and(-55:1)..node[fill=white,inner sep=2]{$a_1$}(2); 
\draw[->](2)..controls(-40:1)and(-20:1)..node[fill=white,inner sep=2]{$a_2$}(3);
\draw[->](150:1.5)..controls(155:1)and(175:1)..node[fill=white,inner sep=2]{$a_{m-2}$}(m-1); 
\draw[->](3)..controls(5:1)and(25:1)..node[fill=white,inner sep=2]{$a_3$}(30:1.5);
\draw[dotted](lu)..controls(135:2)and(45:2)..(ru);
\end{tikzpicture}
 \hspace{7mm}
\begin{tikzpicture}[baseline=0mm,scale=1.3]
\coordinate(l)at(-150.5:2); \coordinate(lu)at(170.5:1.8); \coordinate(ru)at(10:1.8);
\coordinate(r)at(-30:2); \coordinate(ld)at(-115:2); \coordinate(rd)at(-65:2);
\draw(ru)--node(3)[fill=white,inner sep=2]{$d_2$}(r)--node(2)[fill=white,inner sep=2]{$d_1$}(rd)--(ld)--node(m)[fill=white,inner sep=2]{$d_m$}(l)--node(m-1)[fill=white,inner sep=2]{$d_{m-1}$}(lu);
\fill[pattern=north east lines](ld)--(-1.2,-2.1)--(1.2,-2.1)--(rd);
\fill(l)circle(0.8mm); \fill(lu)circle(0.8mm); \fill(ru)circle(0.8mm);
\fill(r)circle(0.8mm); \fill(ld)circle(0.8mm); \fill(rd)circle(0.8mm);
\draw[->](m-1)..controls(-160:1)and(-140:1)..node[fill=white,inner sep=2]{$a_{m-1}$}(m); 
\draw[->](2)..controls(-40:1)and(-20:1)..node[fill=white,inner sep=2]{$a_1$}(3);
\draw[->](150:1.5)..controls(155:1)and(175:1)..node[fill=white,inner sep=2]{$a_{m-2}$}(m-1); 
\draw[->](3)..controls(5:1)and(25:1)..node[fill=white,inner sep=2]{$a_2$}(30:1.5);
\draw[dotted](lu)..controls(135:2)and(45:2)..(ru);
\draw[fill=white]($(ld)!0.5!(rd)$)circle(1mm); \node at(-85:1.5){$v$};
\end{tikzpicture}
   \caption{Sub-quiver of $Q(D)$ in $\triangle_v$}
   \label{Fig dis-quiver}
\end{figure}

\begin{proposition} \cite[Theorem 4.10]{PPP} \label{loc. gentle}
For a $\bullet$-dissection $D$ of $(S,M)$, the algebra $A(D)$ is a complete gentle algebra, and any complete gentle algebra arises in this way.  
\end{proposition}

We prepare a few terminology. 

\begin{definition}
   Let $Q(D)$ be the quiver in Definition \ref{def QD}.
   \begin{enumerate}
   \item[$\bullet$] For a puncture $v\in M_{\circ}$, a cycle $C_v$ is called a {\it special cycle} at $v$. If it is a representative of its cyclic permutation class starting and ending at $d\in D$, then we call it a {\it special $d$-cycle} at $v$. 
   \item[$\bullet$] For $v\in M_{\circ}$, every non-constant sub-path of $C_v$ is called a {\it short path}.
   \end{enumerate}
\end{definition}

\subsection{Two-term silting complexes for $A(D)$ via $D$-laminates}
Let $D$ be a $\bullet$-dissection of $(S,M)$ and $A(D):=\widehat{kQ(D)}/\overline{I(D)}$ the complete gentle algebra associated with $D$.  
In this subsection, we establish a geometric model of two-term silting theory for $A(D)$ and prove Theorem \ref{comp SB gtame}. 
To do it, we need some preparation.

Let $t$ be a sum of all special cycles of $Q(D)$, which is in the center of $A(D)$.

\begin{proposition} \label{modfin}
The complete gentle algebra $A(D)$ is a module-finite $k[[t]]$-algebra. 
\end{proposition}

\begin{proof}
It follows from the fact that $A(D)$ is generated by all short paths and constant paths as an $k[[t]]$-module. 
\end{proof}

We would discuss a class of complexes in $\Kb(\proj A(D))$ obtained from $D$-laminates. 
Let $P_d$ be an indecomposable projective $A(D)$-module corresponding to $d\in D$. 
For $d, e \in D$, every short path in $Q(D)$ from $d$ to $e$ provides a non-vanishing homomorphism $P_e\to P_d$ in $\proj A(D)$, that we call {\it short map}. 

\begin{definition}\rm \label{short string cpx}
An indecomposable two-term complex in $\Kb(\proj A(D))$ is called a {\it two-term string complex} if it can be written as one of the following forms:  
\[
\renewcommand{\arraystretch}{1.5}
\setlength{\tabcolsep}{2.5mm}
\begin{tabular}{ccccccccc}
(1) & (2) & (3) \\
\begin{xy}
(-18,5)*+{P_{d_1}}="e1", (-18,-5)*+{P_{d_3}}="e2", (-20,-20)*{P_{d_{m-2}}}="em-1", (-20,-30)*{P_{d_m}}="em",
(0,0)*+{P_{d_2}}="d1", (0,-10)*+{P_{d_4}}="d2", (2.5,-25)*{P_{d_{m-1}}}="dm-1",
{ "e1"\ar@{->}|{f_{21}} "d1"}, 
{ "e2"\ar@{->}|{f_{23}} "d1"},
{ "e2"\ar@{->}|{f_{43}} "d2"},
{ "em-1"\ar@{->}|{} "dm-1"},
{ "em"\ar@{->}|{f_{m-1\,m}} "dm-1"},
{(-10,-12.5) \ar@{.} (-10, -17.5)},
\end{xy}
&
\begin{xy}
(-18,5)*+{P_{d_1}}="e1", (-18,-5)*+{P_{d_3}}="e2", (-20,-25)*{P_{d_{m-1}}}="em-1", 
(0,0)*+{P_{d_2}}="d1", (0,-10)*+{P_{d_4}}="d2", (2, -20)*{P_{d_{m-2}}}="dm-1", (2.5,-30)*{P_{d_{m}}}="dm",
{ "e1"\ar@{->}|{f_{21}} "d1"}, 
{ "e2"\ar@{->}|{f_{23}} "d1"},
{ "e2"\ar@{->}|{f_{43}} "d2"},
{ "em-1"\ar@{->}|{} "dm-1"}, 
{ "em-1"\ar@{->}|{f_{m\, m-1}} "dm"}, 
{(-10,-12.5) \ar@{.} (-10, -17.5)},
\end{xy}
&
\begin{xy}
(18,5)*+{P_{d_1}}="e1", (18,-5)*+{P_{d_3}}="e2", (20,-20)*{P_{d_{m-2}}}="em-1", (20,-30)*{P_{d_m}}="em",
(0,0)*+{P_{d_2}}="d1", (0,-10)*+{P_{d_4}}="d2", (-2,-25)*{P_{d_{m-1}}}="dm-1",
{ "e1"\ar@{<-}|{f_{12}} "d1"}, 
{ "e2"\ar@{<-}|{f_{32}} "d1"},
{ "e2"\ar@{<-}|{f_{34}} "d2"},
{ "em-1"\ar@{<-}|{} "dm-1"},
{ "em"\ar@{<-}|{f_{m\,m-1}} "dm-1"},
{(7.5,-12.5) \ar@{.} (7.5, -17.5)},
\end{xy}
\end{tabular}
\]
where each $f_{ij}$ is of the form $f_{ij}=t^hf'$ for a short map $f' \colon P_j \to P_i$ and a non-negative integer $h$. 
It is called {\it two-term short string complex} if all $f_{ij}$ are short maps.
\end{definition}

We denote by $\twoscx A(D)$ the set of isomorphism classes of two-term short string complexes $P$ such that $\mathrm{add} P^0 \cap \mathrm{add} P^{-1}=0$. To prove the following proposition, we use $\tau$-tilting theory (see \cite{AIR} for details).

\begin{proposition} \label{2-ips2-scx}
Any indecomposable two-term presilting complex in $\Kb(\proj A(D))$ is a two-term short string complex, that is, $\twoips A(D) \subseteq \twoscx A(D)$. 
\end{proposition}

\begin{proof}
Recall from Proposition \ref{modfin} that $A:=A(D)$ is module-finite
over $R=k[[t]]$.
Then, $B:=A/\mathfrak{m}A$ is a finite dimensional $k$-algebra, where
$\mathfrak{m}=tR$.
Moreover, $B$ is a string algebra. We refer to \cite{BR87,WW} about string algebras and their representations.

Let $\overline{(-)}:=-\otimes_{A}B\colon \Kb(\proj A) \to \Kb(\proj B)$.
For two indecomposable projective $A$-modules $P,Q$ with $P\ncong Q$, we
have an isomorphism
\begin{equation}
     \Hom_A(P,Q)/\mathfrak{m}\Hom_A(P,Q) \cong
\Hom_B(\overline{P},\overline{Q})
\end{equation}
as $k$-vector spaces.
Thus, the set of short maps from $P$ to $Q$ induces a $k$-basis of
$\Hom_B(\overline{P},\overline{Q})$.
In particular, every two-term string complex $P=[P^{-1} \to P^{0}]$ of
projective $B$-modules satisfying $\add P^0\cap \add P^{-1}=0$ are
short.
Here, two-term string complexes and two-term short string complexes of
projective $B$-modules are defined by Definition \ref{short string cpx}
similarly.

Let $T=[T^{-1}\to T^0]\in \twoips A$.
If $T$ is a stalk complex, then the assertion is clear. In the
following, we assume that $T$ is not a stalk complex.
We consider $\overline{T}\in \twoips B$ and $M:=H^0(\overline{T})$.
By \cite[Lemma 3.4]{AIR}, we have $\Hom_{B}(M,\tau M)=0$, where $\tau$
is the Auslander-Reiten translation for $B$.
Since $B$ is a string algebra, $M$ is either a string module or a band
module.
If $M$ is a band module, then it satisfies $M\cong \tau M$ and
\begin{equation*}
     \Hom_B(M,\tau M) = \Hom_B(M,M) \neq 0,
\end{equation*}
a contradiction.
Therefore, $M$ is a string module.
Thus, $\overline{T}$ is a minimal projective presentation of a string
module and hence it is a two-term short string complex.
We can take a two-term short string complex $T'$ of projective
$A$-modules such that $\overline{T'}\cong \overline{T}$.
By Proposition \ref{central nilpotent}, $T'$ is a presilting complex and
$T\cong  T'$ in $\Kb(\proj A)$. It finishes a proof.
\end{proof}

From now on, we establish a geometric model of two-term silting theory for $A(D)$.
First, we give a geometric model of short maps in $\proj A(D)$. 

\begin{definition}
A {\it $D$-segment} is a non-self-intersecting curve, considered up to isotopy relative to $M$, in a polygon $\triangle_v$ of $D$ for some $v \in M_{\circ}$ whose ends are unmarked points on sides of $\triangle_v$ or spirals around $v$. 
\end{definition}

Let $\eta$ be a $D$-segment in $\triangle_v$ whose endpoints lie on $d, e \in \triangle_v \cap D$. 
We orient it to satisfy that $v$ is to its right and its starting point lies on $d$.
Then it corresponds to a short path $d_1 \overset{a_1}{\to} \cdots \overset{a_{s-1}}{\to}d_s$ in $Q(D)$, where $d_1,\ldots, d_s \in D$ are sides of $\triangle_v$ in counterclockwise order with $d_1=e$ and $d_s=d$.
It induces a short map $\sigma(\eta) \colon P_d \to P_e$ in $\proj A(D)$. 

\begin{proposition}\label{bij seg short}
The map $\sigma$ induces a bijection
$$
\sigma : \{\text{$D$-segments whose endpoints lie on $D$}\} \rightarrow  \{\text{short maps in $\proj A(D)$}\}.
$$
\end{proposition}

\begin{proof}
The assertion immediately follows from the definition of $\sigma$.
\end{proof}

Next, we give a geometric model of two-term short string complexes in $\Kb(\proj A(D))$. 

\begin{definition}\label{gen $D$-lam}
A {\it generalized $D$-laminate} is a $\circ$-laminate $\gamma$ intersecting at least one $\bullet$-arc of $D$ such that the condition $(\ast)$ in Definition \ref{$D$-laminate}(2) and the following conditions are satisfied:
\begin{enumerate}
 \item[$\bullet$] Each connected component of $\gamma$ in $\triangle_v$ does not intersect itself for any $v \in M_{\circ}$;
 \item[$\bullet$] For any $d \in D$, all intersection points of $\gamma$ and $d$ are either positive or negative simultaneously.
\end{enumerate}
\end{definition}
Note that a $D$-laminate is precisely a non-self-intersecting generalized $D$-laminate.

A non-closed generalized (NCG, for short) $D$-laminate $\gamma$ is decomposed into $D$-segments $\gamma_0,\ldots,\gamma_{m}$ in polygons such that $\gamma_{i-1}$ and $\gamma_{i}$ have a common endpoint $p_i$ on $d_i\in D$ for every $i\in \{1,\ldots,m\}$. 
In particular, an end of $\gamma_0$ (resp., $\gamma_m$) does not lie on $D$ and both endpoints of $\gamma_i$ lie on $D$ for $i \in \{1,\ldots,m-1\}$. 
By Proposition \ref{bij seg short}, each $\gamma_i$ provides a short map $\sigma(\gamma_i)$ between $T^{(i)}_{\gamma}$ and $T^{(i+1)}_{\gamma}$ for $i \in \{1,\ldots,m-1\}$, where $T^{(i)}_{\gamma}:= P_{d_i}$.
It yields a complex $T_{\gamma}$ in $\Kb(\proj A(D))$. 
More precisely, $T_{\gamma}$ is a two-term complex $[T_{\gamma}^{-1}\overset{f}{\to} T_{\gamma}^{0}]$ given by
$$
T_{\gamma}^{-1} := \bigoplus_{\text{$p_j$:negative}} T_{\gamma}^{(j)}, \quad T_{\gamma}^0 := \bigoplus_{\text{$p_i$:positive}} T_{\gamma}^{(i)},
$$
$$
f=(f_{ij})_{i,j \in \{1,\ldots,m-1\}}, \quad \text{where} \quad f_{ij}:=
\begin{cases} 
\sigma(\gamma_{j-1}) &\text{if $i=j-1$,}\\
\sigma(\gamma_{j}) &\text{if $i=j+1,$}\\
0 &\text{otherwise.}
\end{cases}
$$
From our construction, we have the equality $g(T_{\gamma})=g(\gamma)$ under the identification of $\mathbb{Z}^{\vert D\vert }$ and $\mathbb{Z}^{\vert A\vert }$ via the map $d\mapsto P_d$, where the vector $g({\gamma})\in \mathbb{Z}^{\vert D\vert }$ is defined by the equality \eqref{def g-vect}.

\begin{lemma}\label{lem NCG}
Suppose that two NCG $D$-laminates 
$\gamma$ and $\gamma'$ are decomposed into $D$-segments $\gamma_0,\ldots,\gamma_m$ and $\gamma'_0,\ldots,\gamma'_{m'}$ as above, respectively. If $m=m'>1$ and $\gamma_i=\gamma'_i$ for $i \in \{1,\ldots,m-1\}$, then $\gamma'=\gamma$.
\end{lemma}

\begin{proof}
The $D$-segment $\gamma_0$ (resp., $\gamma_m$) only depends on the sign of $p_1$ (resp., $p_m$), which is uniquely determined by $\gamma_1$ (resp., $\gamma_{m-1}$). Thus we have $\gamma_0=\gamma'_0$ and $\gamma_m=\gamma'_m$.
\end{proof}

\begin{proposition} \label{D-arc}
The map $T_{(-)}\colon \gamma \mapsto T_{\gamma}$ induces a bijection
$$
T_{(-)} \colon \{\text{NCG $D$-laminates}\} \rightarrow \twoscx A(D)
$$
such that $g(\gamma)=g(T_{\gamma})$ for any NCG $D$-laminate $\gamma$.
\end{proposition}

\begin{proof}
It follows from our construction that $T_{\gamma}$ is contained in $\twoscx A(D)$.
By Lemma \ref{lem NCG}, this map is injective.

In order to prove surjectivity of the map, we give the inverse map. 
If $P \in \twoscx A(D)$ is a stalk complex $P_d$ with $d\in D$ concentrated in degree $0$ (resp., $-1$), then we just take $\gamma=d_{+}^{\ast}$ (resp., $\gamma=d_{-}^{\ast}$). 
On the other hand, let $P \in \twoscx A(D)$ be a non-stalk complex which is one of (1)-(3) in Definition \ref{short string cpx}. 
We only consider the form (1) since the others can be proved similarly.
By Proposition \ref{bij seg short}, $\gamma_1:=\sigma^{-1}(f_{21}),\gamma_2:=\sigma^{-1}(f_{23}),\ldots,\gamma_{m-1}:=\sigma^{-1}(f_{m-1\,m})$ are $D$-segments, and $\gamma_{i-1}$ and $\gamma_i$ have a common endpoint on $d_i$ for $i \in \{2,\ldots,m-1\}$. 
Then there are two $D$-segments $\gamma_0$ and $\gamma_m$ such that the curve $\gamma$ obtained by gluing $\gamma_0,\ldots,\gamma_m$ one by one is an NCG $D$-laminate.
From our construction, we have $P=T_{\gamma}$.
\end{proof}

Finally, we give a geometric model of two-term presilting/silting complexes in $\Kb(\proj A(D))$.
We recall a description of morphisms between two-term short string complexes due to \cite{ALP}. We give examples of these morphisms, singleton single maps and quasi-graph map representatives, in Section \ref{Sec example rep}.

\begin{definition} \cite[Definition 3.7]{ALP} \label{def singleton single map}
For $T, T' \in\twoscx A(D)$, a morphism $f\in\Hom_{\Kb(\proj A(D))}(T, T'[1])$ is called a {\it singleton single map} if it is induced by a short map $p$ as one of the following forms:
\[
(a)
\begin{tikzpicture}[baseline=1mm,scale=1]
\node(l)at(-1.7,0){$\bullet$}; \node(r)at(1.7,0){$\bullet$};
\node(ld)at(-0.5,-0.5){$\bullet$}; \node(rd)at(0.5,-0.5){$\bullet$};
\node(ldd)at(-1.7,-1){$\bullet$}; \node(rdd)at(1.7,-1){$\bullet$};
\draw[->,blue](l)--node[above]{$p$}(r);
\draw[->](l)--node[fill=white,inner sep=2]{$q$}(ld); \draw[->](rd)--node[fill=white,inner sep=2]{$q'$}(r);
\draw[->](ldd)--(ld); \draw[->](rd)--(rdd);
\node at(-1.1,-1){$\vdots$}; \node at(1.1,-1){$\vdots$};
\node at(-2.2,-0.5){$T$}; \node at(2.2,-0.5){$T'$};
\end{tikzpicture}
\hspace{5mm}
(b)
\begin{tikzpicture}[baseline=7mm,scale=1]
\node(l)at(-1.7,0){$\bullet$}; \node(r)at(1.7,0){$\bullet$};
\node(lu)at(-0.5,0.5){$\bullet$};
\node(ld)at(-0.5,-0.5){$\bullet$}; \node(rd)at(0.5,-0.5){$\bullet$};
\draw[->,blue](l)--node[fill=white,inner sep=2]{$p$}(r);
\draw[->](l)--node[fill=white,inner sep=2]{$q$}(ld); \draw[->](rd)--node[fill=white,inner sep=2]{$q'$}(r);
\draw[->](l)--node[above]{$p'p$}(lu);
\node at(-1.1,-0.7){$\vdots$}; \node at(1.1,-0.7){$\vdots$};
\node at(-1.1,1.2){$\vdots$};
\node at(-2.2,0){$T$}; \node at(2.2,0){$T'$};
\end{tikzpicture}
\]
\[
(c)
\begin{tikzpicture}[baseline=7mm,scale=1]
\node(l)at(-1.7,0){$\bullet$}; \node(r)at(1.7,0){$\bullet$};
\node(ru)at(0.5,0.5){$\bullet$};
\node(ld)at(-0.5,-0.5){$\bullet$}; \node(rd)at(0.5,-0.5){$\bullet$};
\draw[->,blue](l)--node[fill=white,inner sep=2]{$p$}(r);
\draw[->](l)--node[fill=white,inner sep=2]{$q$}(ld); \draw[->](rd)--node[fill=white,inner sep=2]{$q'$}(r);
\draw[->](ru)--node[above]{$pp''$}(r);
\node at(-1.1,-0.7){$\vdots$}; \node at(1.1,-0.7){$\vdots$};
\node at(1.1,1.2){$\vdots$};
\node at(-2.2,0){$T$}; \node at(2.2,0){$T'$};
\end{tikzpicture}
\hspace{5mm}
(d)
\begin{tikzpicture}[baseline=7mm,scale=1]
\node(l)at(-1.7,0){$\bullet$}; \node(r)at(1.7,0){$\bullet$};
\node(lu)at(-0.5,0.5){$\bullet$}; \node(ru)at(0.5,0.5){$\bullet$};
\node(ld)at(-0.5,-0.5){$\bullet$}; \node(rd)at(0.5,-0.5){$\bullet$};
\draw[->,blue](l)--node[fill=white,inner sep=2]{$p$}(r);
\draw[->](l)--node[fill=white,inner sep=2]{$q$}(ld); \draw[->](rd)--node[fill=white,inner sep=2]{$q'$}(r);
\draw[->](l)--node[above]{$p'p$}(lu); \draw[->](ru)--node[above]{$pp''$}(r);
\node at(-1.1,-0.7){$\vdots$}; \node at(1.1,-0.7){$\vdots$};
\node at(-1.1,1.2){$\vdots$}; \node at(1.1,1.2){$\vdots$};
\node at(-2.2,0){$T$}; \node at(2.2,0){$T'$};
\end{tikzpicture}
\]
where $p$ and $q$ (resp., $q'$) have no common arrows as paths, and $p'$ and $p''$ are not constant.
\end{definition}

\begin{definition}\label{quasi-graph}\cite[Definition 3.12]{ALP}
For $T, T' \in\twoscx A(D)$, the following diagram is called a {\it quasi-graph map from $T$ to $T'$}:
\[
\begin{tikzpicture}[baseline=7mm,scale=1]
\fill[pattern=north east lines](-1.2,2.3)--(-1.2,1.5)--(2,1.5)--(2,2.3);
\fill[pattern=north east lines](2,-1.8)--(2,-1)--(-1.2,-1)--(-1.2,-1.8);
\node(l)at(-1.7,0.25){$\bullet$}; \node(r)at(0.5,0.25){$\bullet$};
\node(l')at(-0.5,-0.25){$\bullet$}; \node(r')at(1.7,-0.25){$\bullet$};
\node(l'')at(-0.5,0.75){$\bullet$}; \node(r'')at(1.7,0.75){$\bullet$};
\node(lu)at(-0.5,1.5){$\bullet$}; \node(ru)at(1.7,1.5){$\bullet$};
\node(ld)at(-0.5,-1){$\bullet$}; \node(rd)at(1.7,-1){$\bullet$};
\draw[->](l)--node[fill=white,inner sep=2]{$g$}(l'); \draw[->](r)--node[fill=white,inner sep=2]{$g$}(r'); 
\draw[->](l)--node[fill=white,inner sep=2]{$f$}(l''); \draw[->](r)--node[fill=white,inner sep=2]{$f$}(r'');
\draw[blue,double](l)--(r); \draw[blue,double](l')--(r'); \draw[blue,double](l'')--(r'');
\draw[blue,double](lu)--(ru); \draw[blue,double](ld)--(rd);
\node at(-0.5,1.25){$\vdots$}; \node at(-0.5,-0.5){$\vdots$};
\node at(1.7,1.25){$\vdots$}; \node at(1.7,-0.5){$\vdots$};
\draw(lu)--(-1.2,1.5)--(-1.2,2.3) (ru)--(2,1.5)--(2,2.3);
\draw(rd)--(2,-1)--(2,-1.8) (ld)--(-1.2,-1)--(-1.2,-1.8);
\node at(-2.2,0.25){$T$}; \node at(2.2,0.25){$T'$};
\node[rectangle,draw,fill=white]at(0.4,1.9){U}; \node[rectangle,draw,fill=white]at(0.4,-1.4){L};
\node at(0,0){$\circlearrowright$}; \node at(0,0.5){$\circlearrowright$};
\end{tikzpicture}
\]
where the double lines are isomorphisms and $\fbox{U}$ and $\fbox{L}$ are
\[
\begin{tikzpicture}[baseline=5mm,scale=1]
\node(l)at(-1.7,0){$\bullet$}; \node(r)at(0.5,0){$\bullet$}; \node at(-2,1.2){$(e)$};
\node(lu)at(-0.5,0.5){$P$}; \node(ru)at(1.7,0.5){$P'$};
\node(luu)at(-1.7,1){$\bullet$}; \node(ruu)at(0.5,1){$\bullet$};
\draw[->](l)--node[fill=white,inner sep=2]{$q$}(lu);
\draw[->](r)--node[fill=white,inner sep=2]{$q'$}(ru);
\draw[->](luu)--(lu); \draw[->](ruu)--(ru);
\draw[blue,double](l)--(r);
\node at(-1.1,1.3){$\vdots$}; \node at(1.1,1.3){$\vdots$};
\end{tikzpicture}
\hspace{5mm}\text{or}\hspace{5mm}
\begin{tikzpicture}[baseline=5mm,scale=1]
\node(l)at(-0.5,0){$\bullet$}; \node(r)at(1.7,0){$\bullet$}; \node at(-2.2,1.2){$(f)$};
\node(lu)at(-1.7,0.5){$P$}; \node(ru)at(0.5,0.5){$P'$};
\node(luu)at(-0.5,1){$\bullet$}; \node(ruu)at(1.7,1){$\bullet$};
\draw[->](lu)--node[fill=white,inner sep=2]{$r$}(l);
\draw[->](ru)--node[fill=white,inner sep=2]{$r'$}(r);
\draw[->](ru)--(ruu); \draw[->](lu)--(luu);
\draw[blue,double](l)--(r);
\node at(-1.1,1.3){$\vdots$}; \node at(1.1,1.3){$\vdots$};
\end{tikzpicture}
\]
and there is no $p\in\Hom_{A(D)}(P,P')$ such that $pq=q'$ or $r=r'p$, which implies $q'\neq 0$ and $r\neq 0$. Note that $T$ and $T'$ have the common morphisms $\ldots,f,g,\ldots$ which are not contained in $\fbox{U}$ and $\fbox{L}$.
For a quasi-graph map from $T$ to $T'$, the common morphisms naturally induces a unique morphism in $\Hom_{\Kb(\proj A(D))}(T, T'[1])$, called a {\it quasi-graph map representative}.
\end{definition}

Regarding two-term short string complexes as homotopy strings defined as in \cite{ALP} (see also \cite{BM}), we can obtain the following result from \cite{ALP}.

\begin{proposition}\cite[Propositions 4.1 and 4.8]{ALP}\label{basis maps}
For $T,T'\in\twoscx A(D)$, singleton single maps and quasi-graph map representatives give a basis of $\Hom_{\Kb(\proj A(D))}(T, T'[1])$.
\end{proposition}

Let $\gamma$ and $\delta$ be NCG $D$-laminates.
By Propositions \ref{D-arc} and \ref{basis maps}, $\Hom_{\Kb(\proj A(D))}(T_{\gamma}, T_{\delta}[1])$ has a basis consisting of singleton single maps and quasi-graph map representatives.
It follows from the definition that a singleton single map in $\Hom_{\Kb(\proj A(D))}(T_{\gamma}, T_{\delta}[1])$ is given by one of the following local figures:
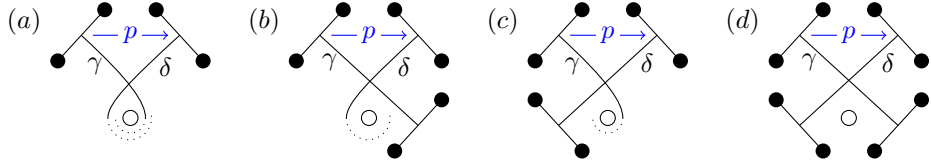
\[
(a)
\begin{tikzpicture}[baseline=7mm,scale=1]
\coordinate(l)at(165:1); \coordinate(lu)at(110:1); \coordinate(ru)at(70:1); \coordinate(r)at(15:1);
\draw(l)--(lu) (ru)--(r);
\fill(l)circle(1mm); \fill(lu)circle(1mm); \fill(ru)circle(1mm); \fill(r)circle(1mm);
\draw(0,-0.5)circle(1mm);
\draw(-0.3,-0.5)..controls(-0.3,-0.3)and(0,0)..node[below,pos=0.9]{$\delta$}($(r)!0.5!(ru)$);
\draw[dotted](-0.3,-0.5)arc(-180:10:2.9mm); 
\draw(0.2,-0.5)..controls(0.2,-0.3)and(0,0)..node[below,pos=0.9]{$\gamma$}($(l)!0.5!(lu)$);
\draw[dotted](0.2,-0.5)arc(0:-180:2mm);
\draw[->,blue](-0.5,0.6)--node[fill=white,inner sep=2]{$p$}(0.5,0.6);
\end{tikzpicture}
\hspace{4mm}
(b)
\begin{tikzpicture}[baseline=7mm,scale=1]
\coordinate(l)at(165:1); \coordinate(lu)at(110:1); \coordinate(ru)at(70:1); \coordinate(r)at(15:1);
\coordinate(rd)at(-15:1); \coordinate(rdd)at(-70:1);
\draw(l)--(lu) (ru)--(r) (rd)--(rdd);
\fill(l)circle(1mm); \fill(lu)circle(1mm); \fill(ru)circle(1mm); \fill(r)circle(1mm);
\fill(rd)circle(1mm); \fill(rdd)circle(1mm);
\draw(0,-0.5)circle(1mm);
\draw(-0.3,-0.5)..controls(-0.3,-0.3)and(0,0)..node[below,pos=0.9]{$\delta$}($(r)!0.5!(ru)$);
\draw[dotted](-0.3,-0.5)arc(-180:10:2.9mm); 
\draw($(rd)!0.5!(rdd)$)--node[below,pos=0.9]{$\gamma$}($(l)!0.5!(lu)$);
\draw[->,blue](-0.5,0.6)--node[fill=white,inner sep=2]{$p$}(0.5,0.6);
\end{tikzpicture}
\hspace{4mm}
(c)
\begin{tikzpicture}[baseline=7mm,scale=1]
\coordinate(l)at(165:1); \coordinate(lu)at(110:1); \coordinate(ru)at(70:1); \coordinate(r)at(15:1);
\coordinate(ld)at(-165:1); \coordinate(ldd)at(-110:1);
\draw(l)--(lu) (ru)--(r) (ld)--(ldd);
\fill(l)circle(1mm); \fill(lu)circle(1mm); \fill(ru)circle(1mm); \fill(r)circle(1mm);
\fill(ld)circle(1mm); \fill(ldd)circle(1mm);
\draw(0,-0.5)circle(1mm);
\draw(0.2,-0.5)..controls(0.2,-0.3)and(0,0)..node[below,pos=0.9]{$\gamma$}($(l)!0.5!(lu)$);
\draw[dotted](0.2,-0.5)arc(0:-180:2mm);
\draw($(ld)!0.5!(ldd)$)--node[below,pos=0.9]{$\delta$}($(r)!0.5!(ru)$);
\draw[->,blue](-0.5,0.6)--node[fill=white,inner sep=2]{$p$}(0.5,0.6);
\end{tikzpicture}
\hspace{4mm}
(d)
\begin{tikzpicture}[baseline=7mm,scale=1]
\coordinate(l)at(165:1); \coordinate(lu)at(110:1); \coordinate(ru)at(70:1); \coordinate(r)at(15:1);
\coordinate(rd)at(-15:1); \coordinate(rdd)at(-70:1); \coordinate(ld)at(-165:1); \coordinate(ldd)at(-110:1);
\draw(l)--(lu) (ru)--(r) (ld)--(ldd) (rd)--(rdd);
\fill(l)circle(1mm); \fill(lu)circle(1mm); \fill(ru)circle(1mm); \fill(r)circle(1mm);
\fill(ld)circle(1mm); \fill(ldd)circle(1mm); \fill(rd)circle(1mm); \fill(rdd)circle(1mm);
\draw(0,-0.5)circle(1mm);
\draw($(ld)!0.5!(ldd)$)--node[below,pos=0.9]{$\delta$}($(r)!0.5!(ru)$);
\draw($(rd)!0.5!(rdd)$)--node[below,pos=0.9]{$\gamma$}($(l)!0.5!(lu)$);
\draw[->,blue](-0.5,0.6)--node[fill=white,inner sep=2]{$p$}(0.5,0.6);
\end{tikzpicture}
\]
where $p$ is the associated short map. 
For a quasi-graph map representative in $\Hom_{\Kb(\proj A(D))}(T_{\gamma}, T_{\delta}[1])$, local figures for $\fbox{U}$ and $\fbox{L}$ in Definition \ref{quasi-graph} are given by one of the following figures respectively:
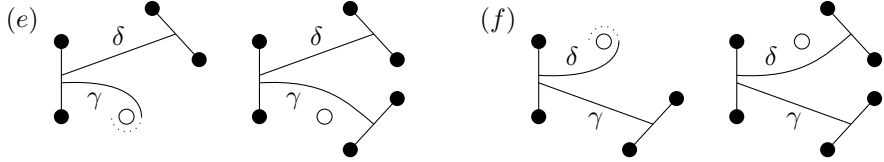
\[
\begin{tikzpicture}[baseline=-1mm,scale=1.4]
\coordinate(l)at(210:1); \coordinate(lu)at(150:1); \coordinate(ru)at(70:1); \coordinate(r)at(15:1);
\draw(l)--(lu) (ru)--(r);
\fill(l)circle(0.7mm); \fill(lu)circle(0.7mm); \fill(ru)circle(0.7mm); \fill(r)circle(0.7mm);
\draw(0,-0.5)circle(0.7mm);
\draw($(l)!0.6!(lu)$)--node[above]{$\delta$}($(r)!0.6!(ru)$);
\draw(0.2,-0.5)..controls(0.2,-0.3)and(0,0)..node[below,pos=0.8]{$\gamma$}($(l)!0.4!(lu)$);
\draw[dotted](0.2,-0.5)arc(0:-180:2mm);
\draw[->,blue](-0.7,0.05)--node[below,pos=0.7]{$q'$}(0.55,0.55); \node at(155:1.4){$(e)$};
\end{tikzpicture}
\hspace{5mm}\text{or}\hspace{5mm}
\begin{tikzpicture}[baseline=-1mm,scale=1.4]
\coordinate(l)at(210:1); \coordinate(lu)at(150:1); \coordinate(ru)at(70:1); \coordinate(r)at(15:1);
\coordinate(rd)at(-15:1); \coordinate(rdd)at(-70:1);
\draw(l)--(lu) (ru)--(r) (rd)--(rdd);
\fill(l)circle(0.7mm); \fill(lu)circle(0.7mm); \fill(ru)circle(0.7mm); \fill(r)circle(0.7mm); \fill(rd)circle(0.7mm); \fill(rdd)circle(0.7mm);
\draw(0,-0.5)circle(0.7mm);
\draw($(l)!0.6!(lu)$)--node[above]{$\delta$}($(r)!0.6!(ru)$);
\draw($(rd)!0.5!(rdd)$)..controls(0.3,-0.4)and(0,-0.2)..node[below,pos=0.8]{$\gamma$}($(l)!0.4!(lu)$);
\draw[->,blue](-0.7,0.05)--node[below,pos=0.8]{$q'$}(0.55,0.55);
\draw[<-,blue](-0.7,-0.05)..controls(0,-0.1)and(0.3,-0.3)..node[above,pos=0.7]{$q$}(0.55,-0.45);
\end{tikzpicture}
\]
\[
\begin{tikzpicture}[baseline=-1mm,scale=1.4]
\coordinate(l)at(210:1); \coordinate(lu)at(150:1); \coordinate(ru)at(70:1); \coordinate(r)at(15:1);
\coordinate(rd)at(-15:1); \coordinate(rdd)at(-70:1);
\draw(l)--(lu) (rd)--(rdd);
\fill(l)circle(0.7mm); \fill(lu)circle(0.7mm); \fill(rd)circle(0.7mm); \fill(rdd)circle(0.7mm);
\draw(0,0.5)circle(0.7mm);
\draw($(l)!0.4!(lu)$)--node[below]{$\gamma$}($(rd)!0.6!(rdd)$);
\draw(0.2,0.5)..controls(0.2,0.3)and(0,0)..node[above,pos=0.8]{$\delta$}($(l)!0.6!(lu)$);
\draw[dotted](0.2,0.5)arc(0:180:2mm);
\draw[<-,blue](-0.7,-0.05)--node[above,pos=0.7]{$r$}(0.55,-0.55); \node at(155:1.4){$(f)$};
\end{tikzpicture}
\hspace{5mm}\text{or}\hspace{5mm}
\begin{tikzpicture}[baseline=-1mm,scale=1.4]
\coordinate(l)at(210:1); \coordinate(lu)at(150:1); \coordinate(ru)at(70:1); \coordinate(r)at(15:1);
\coordinate(rd)at(-15:1); \coordinate(rdd)at(-70:1);
\draw(l)--(lu) (ru)--(r) (rd)--(rdd);
\fill(l)circle(0.7mm); \fill(lu)circle(0.7mm); \fill(ru)circle(0.7mm); \fill(r)circle(0.7mm); \fill(rd)circle(0.7mm); \fill(rdd)circle(0.7mm);
\draw(0,0.5)circle(0.7mm);
\draw($(l)!0.4!(lu)$)--node[below]{$\gamma$}($(rd)!0.6!(rdd)$);
\draw($(r)!0.5!(ru)$)..controls(0.3,0.4)and(0,0.2)..node[above,pos=0.8]{$\delta$}($(l)!0.6!(lu)$);
\draw[<-,blue](-0.7,-0.05)--node[above,pos=0.7]{$r$}(0.55,-0.55);
\draw[->,blue](-0.7,0.05)..controls(0,0.1)and(0.3,0.3)..node[below,pos=0.7]{$r'$}(0.55,0.45);
\end{tikzpicture}
\]
where the left figure of $(e)$ (resp., $(f)$) is in the case of $q=0$ (resp., $r'=0$). 
For example, if both of $\fbox{U}$ and $\fbox{L}$ are given by the right figure of $(e)$, then there is the following intersection point: 
\[
\begin{tikzpicture}[baseline=7mm,scale=1]
\coordinate(ru)at(2,0.7); \coordinate(rd)at(2,-0.7); \coordinate(rruu)at(3,1); \coordinate(rru)at(3.7,0.3);
\coordinate(rrd)at(3.7,-0.3); \coordinate(rrdd)at(3,-1);
\coordinate(lu)at(-2,0.7); \coordinate(ld)at(-2,-0.7); \coordinate(lluu)at(-3,1); \coordinate(llu)at(-3.7,0.3);
\coordinate(lld)at(-3.7,-0.3); \coordinate(lldd)at(-3,-1);
\draw(ru)--(rd) (rruu)--(rru) (rrd)--(rrdd);
\draw(lu)--(ld) (lluu)--(llu) (lld)--(lldd);
\fill(ru)circle(1mm); \fill(rd)circle(1mm); \fill(rruu)circle(1mm); \fill(rru)circle(1mm); 
\fill(rrd)circle(1mm); \fill(rrdd)circle(1mm);
\fill(lu)circle(1mm); \fill(ld)circle(1mm); \fill(lluu)circle(1mm); \fill(llu)circle(1mm); 
\fill(lld)circle(1mm); \fill(lldd)circle(1mm);
\draw(2.7,-0.5)circle(1mm); \draw(-2.7,0.5)circle(1mm);
\draw[blue]($(lldd)!0.5!(lld)$)--node[above,pos=0.65]{$\delta$}($(rruu)!0.5!(rru)$);
\draw[blue]($(lluu)!0.5!(llu)$)..controls(-3,-0.7)and(3,0.7)..node[below,pos=0.6]{$\gamma$}($(rrdd)!0.5!(rrd)$);
\end{tikzpicture}
\]

From the above description, we find that $\gamma$ is not in positive position for $\delta$ whenever there exists a singleton single map or a quasi-graph map representatives in $\Hom_{\Kb(\proj A(D))}(T_{\gamma},T_{\delta}[1])$. Conversely, if $\gamma$ is not in positive position for $\delta$, then they have a common intersection point which gives rise to either a singleton single map or a quasi-graph map representatives as in the above figure respectively. Thus, we have the following.

\begin{proposition} \label{pos.pos} 
The following conditions are equivalent for two NCG $D$-laminates $\gamma$ and $\delta$:
\begin{enumerate}
\item $\Hom_{\Kb(\proj A(D))}(T_{\gamma}, T_{\delta}[1])=0$;
\item $\gamma$ is in positive position for $\delta$.
\end{enumerate}
\end{proposition}

\begin{proof}
The assertion follows from Proposition \ref{basis maps} and the above observations.
\end{proof}

We are ready to state our results. 

\begin{proposition} \label{D-ips}
   The following hold.
\begin{enumerate}
   \item The map $T_{(-)}$ in Proposition \ref{D-arc} restricts to a bijection
   $$
   \{\text{non-closed $D$-laminates}\} \rightarrow \twoips A.
   $$
   \item Two non-closed $D$-laminates $\gamma$ and $\delta$ are compatible if and only if $T_{\gamma}\oplus T_{\delta}$ is presilting.
\end{enumerate}
\end{proposition}

\begin{proof}
 In general, two NCG $D$-laminates $\gamma$ and $\delta$ do not intersect if and only if they are in positive position for each other. 
 By Proposition \ref{pos.pos}, this is equivalent to the condition that $\Hom(T_{\gamma}, T_{\delta}[1])=\Hom(T_{\delta}, T_{\gamma}[1])=0$. 
 Since a non-closed $D$-laminate is precisely a non-self-intersecting NCG $D$-laminate, we get (1) and (2).
\end{proof}

\begin{theorem} \label{D-twosilt}
The map $\mathcal{X}\mapsto T_{\mathcal{X}}:=\bigoplus_{\gamma\in \mathcal{X}}T_{\gamma}$ gives bijections
\begin{eqnarray*}
\{\text{reduced $D$-laminations}\} \to \twopresilt A(D),\ \text{and}
\end{eqnarray*}
\begin{eqnarray*}
\{\text{complete $D$-laminations}\} \to \twosilt A(D)
\end{eqnarray*}
such that $C(\mathcal{X})=C(T_{\mathcal{X}})$ for all reduced $D$-laminations $\mathcal{X}$. 
In particular, we have 
$$\mathcal{F}(D)=\mathcal{F}(A(D)) \quad \text{and}\quad \vert \mathcal{F}(D)\vert =\vert \mathcal{F}(A(D))\vert .$$
\end{theorem}

\begin{proof}
It follows from Proposition \ref{D-ips}. 
\end{proof}

\begin{corollary} \label{FD is non-singular}
The set $\mathcal{F}(D)$ is a non-singular fan whose maximal faces correspond to complete $D$-laminations. 
\end{corollary}

\begin{proof}
It is immediate from Corollary \ref{prop fan} and Theorem \ref{D-twosilt}. 
\end{proof}

Now, we prove Theorem \ref{comp SB gtame}.

\begin{proof}[Proof of Theorem \ref{comp SB gtame}]
By Proposition \ref{loc. gentle}, any complete gentle algebra is given as $A(D)$ for some $\bullet$-dissection $D$ of a $\circ\bullet$-marked surface $(S,M)$, and it follows from Theorems \ref{dense} and \ref{D-twosilt} that complete gentle algebras are $g$-tame. 
Thus, the assertion follows from Corollary \ref{factor-g-tame} since every complete special biserial algebras is a factor algebra of some complete gentle algebra. 
\end{proof}

\subsection{Application to finite dimensional $k$-algebras} \label{appfindim}

Using Proposition \ref{central nilpotent}, we can slightly generalize Theorem \ref{D-twosilt} as follows. 

\begin{corollary}\label{thm:SB-BGA}
Let $(S,M)$ be a $\circ\bullet$-marked surface and $D$ a $\bullet$-dissection of $(S,M)$. 
Let $A(D)$ be a complete gentle algebra associated with $D$, which is module finite over $R=k[[t]]$. 
Let $I_{\rm sc}$ be a two-sided ideal in $A(D)$ generated by all special cycles in $A(D)$. Then, we have $\mathfrak{m}A(D)\subset I_{\rm sc}$. For an arbitrary two-sided ideal $0\subseteq J \subseteq I_{\rm sc}$, canonical surjections
\begin{equation*} 
   A(D) \twoheadrightarrow A(D)/J \twoheadrightarrow A(D)/I_{\rm sc}
\end{equation*}
induce 
\begin{equation} \label{moduloJ} 
   \mathcal{F}(A(D))= \mathcal{F}(A(D)/J) = \mathcal{F}(A(D)/I_{\rm sc}). 
\end{equation}
Furthermore, they coincide with $\mathcal{F}(D)$ of $D$ as non-singular fans.
\end{corollary} 

\begin{proof}
By Proposition \ref{modfin}, $A(D)$ is a module-finite algebra over $R:=k[[t]]$, where $t$ is a sum of all special cycles. 
By Proposition \ref{central nilpotent}, the canonical surjection $A(D)\twoheadrightarrow B:= A(D)/\mathfrak{m}A(D)$ gives $\mathcal{F}(A(D))=\mathcal{F}(B)$. 
On the other hand, $B$ is a finite dimensional $k$-algebra and the image of $I_{\rm sc}$ in $B$ is in the Jacobson radical and generated by central elements of $B$. Applying \cite[Theorem 1]{EJR} to $B$, we have $\mathcal{F}(B)= \mathcal{F}(A(D)/I_{\rm sc})$. 
Thus, we have $\mathcal{F}(A(D))=\mathcal{F}(A(D)/I_{\rm sc})$. 
Furthermore, by using Proposition \ref{factor-twosilt}, we have the equation \eqref{moduloJ} for any two-sided ideal $J$ in $A(D)$ contained in $I_{\rm sc}$. The latter assertion follows from Theorem \ref{D-twosilt}.
\end{proof}

We end this section with an application to finite dimensional $k$-algebras. 
We may apply Corollary \ref{thm:SB-BGA} to study two-term silting theory for two important classes of complete special biserial algebras, finite dimensional gentle algebras and Brauer graph algebras as follows.

\begin{example}\label{example BGA}
   Let $D$ be a $\bullet$-dissection of a $\circ\bullet$-marked surface $(S,M)$. Let $A(D)$ be a complete gentle algebra associated with $D$ and $I_{\rm sc}$ a two-sided ideal in $A(D)$ generated by all special cycles in $A(D)$, see Corollary \ref{thm:SB-BGA}. 
   \begin{enumerate}
      \item[(a)] If all $\circ$-marked points lie on the boundary $\partial S$, then $I_{\rm sc}=0$ and $A(D)$ is precisely a finite dimensional gentle algebra.  
      \item[(b)] Assume that all marked points are punctures (i.e., $\partial S = \emptyset$).  
      Consider a Brauer graph algebra defined from the following Brauer graph (We refer to \cite{Schroll18} for the definitions of Brauer graphs and Brauer graph algebras): 
      \begin{itemize}
      \item The set of vertices corresponds bijectively with $M_{\circ}$;
      \item The set of edges corresponds bijectively with the dual dissection $D^{\ast}$ of $D$;
      \item The cyclic ordering around each vertex is induced from the orientation of $S$;
      \item A multiplicity of a vertex $v$ is an arbitrary positive integer $m_v>0$. 
   \end{itemize}
   By definition, this is a factor algebra of the complete gentle algebra $A(D)$ modulo the two-sided ideal contained in $I_{\rm sc}$. Conversely, it is shown in \cite{Labourie13} that every Brauer graph algebra arises in this way. 
   \end{enumerate}
   Applying Corollary \ref{thm:SB-BGA}, we find that two-term silting complexes over finite dimensional gentle algebras and Brauer graph algebras are given by the surface model explained in (a) and (b) respectively.
\end{example}

\section{Examples for representation theory}\label{Sec example rep}

(1) Let $(S,M)$ be a disk with $\vert M\vert =10$ such that one marked point in $M_{\circ}$ (resp., $M_{\bullet}$) is a puncture and the others lie on $\partial S$.
For a $\bullet$-dissection of $(S,M)$
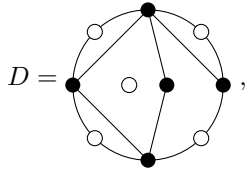
\[
D=
\begin{tikzpicture}[baseline=0mm]
\coordinate(r)at(0:1); \coordinate(ru)at(45:1); \coordinate(u)at(90:1); \coordinate(lu)at(135:1);
\coordinate(l)at(180:1); \coordinate(rd)at(-45:1); \coordinate(d)at(-90:1); \coordinate(ld)at(-135:1);
\coordinate(rc)at(0:0.25); \coordinate(lc)at(180:0.25);
\draw(0,0)circle(1);
\draw(u)--(l)--(d)--(rc)--(u)--(r);
\fill(l)circle(1mm); \fill(r)circle(1mm); \fill(u)circle(1mm); \fill(d)circle(1mm); \fill(rc)circle(1mm);
\draw[fill=white](lu)circle(1mm); \draw[fill=white](ru)circle(1mm); 
\draw[fill=white](ld)circle(1mm); \draw[fill=white](rd)circle(1mm); \draw[fill=white](lc)circle(1mm);
\end{tikzpicture}\ ,
\]
the quiver $Q(D)$ and the ideal $I(D)$ are given by
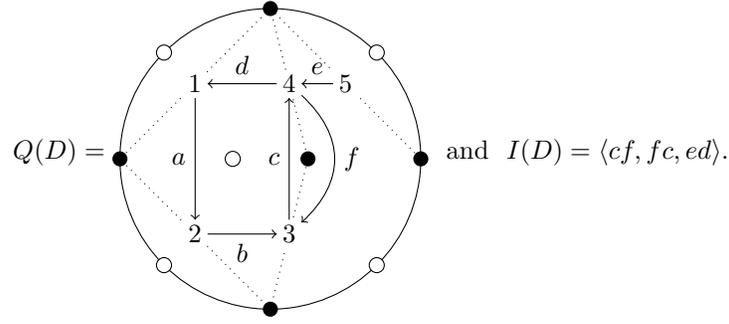
\[
Q(D)=
\begin{tikzpicture}[baseline=0mm]
\coordinate(r)at(0:2); \coordinate(ru)at(45:2); \coordinate(u)at(90:2); \coordinate(lu)at(135:2);
\coordinate(l)at(180:2); \coordinate(rd)at(-45:2); \coordinate(d)at(-90:2); \coordinate(ld)at(-135:2);
\coordinate(rc)at(0:0.5); \coordinate(lc)at(180:0.5);
\draw(0,0)circle(2);
\draw[dotted](u)--node(1)[fill=white,inner sep=2]{$1$}(l)--node(2)[fill=white,inner sep=2]{$2$}(d)
--node(3)[fill=white,inner sep=2]{$3$}(rc)--node(4)[fill=white,inner sep=2]{$4$}(u)
--node(5)[fill=white,inner sep=2]{$5$}(r);
\draw[->](1)--node[left]{$a$}(2); \draw[->](2)--node[below]{$b$}(3); \draw[->](3)--node[left]{$c$}(4);
\draw[->](4)--node[above]{$d$}(1); \draw[->](5)--node[above]{$e$}(4);
\draw[->](4)..controls(1,0.3)and(1,-0.3)..node[right]{$f$}(3);
\fill(l)circle(1mm); \fill(r)circle(1mm); \fill(u)circle(1mm); \fill(d)circle(1mm); \fill(rc)circle(1mm);
\draw[fill=white](lu)circle(1mm); \draw[fill=white](ru)circle(1mm); 
\draw[fill=white](ld)circle(1mm); \draw[fill=white](rd)circle(1mm); \draw[fill=white](lc)circle(1mm);
\end{tikzpicture}\ 
\text{ and }\ I(D)=\langle cf,fc,ed\rangle.
\]

(i) We consider an NCG $D$-laminate $\gamma$, but not a $D$-laminate, that is decomposed into $D$-segments $\gamma_0,\ldots,\gamma_5$ as follows:
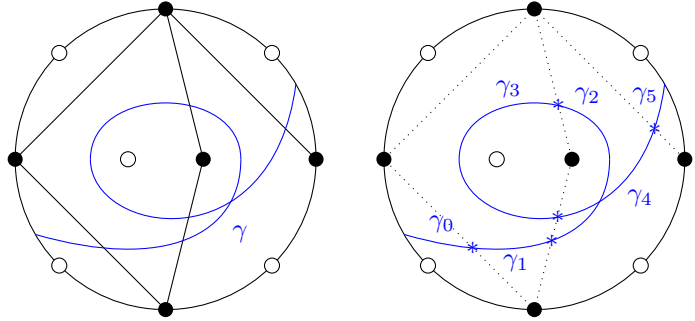
\[
\begin{tikzpicture}[baseline=0mm]
\coordinate(r)at(0:2); \coordinate(ru)at(45:2); \coordinate(u)at(90:2); \coordinate(lu)at(135:2);
\coordinate(l)at(180:2); \coordinate(rd)at(-45:2); \coordinate(d)at(-90:2); \coordinate(ld)at(-135:2);
\coordinate(rc)at(0:0.5); \coordinate(lc)at(180:0.5);
\node[blue]at(-45:1.4){$\gamma$};
\draw(0,0)circle(2);
\draw(u)--(l)--(d)--(rc)--(u)--(r);
\draw[blue](-150:2)..controls(-90:1.5)and(1,-1)..(0:1); \draw[blue](180:1)..controls(-1,1)and(1,1)..(0:1);
\draw[blue](180:1)..controls(-1,-1)and(-45:2)..(30:2);
\fill(l)circle(1mm); \fill(r)circle(1mm); \fill(u)circle(1mm); \fill(d)circle(1mm); \fill(rc)circle(1mm);
\draw[fill=white](lu)circle(1mm); \draw[fill=white](ru)circle(1mm); 
\draw[fill=white](ld)circle(1mm); \draw[fill=white](rd)circle(1mm); \draw[fill=white](lc)circle(1mm);
\end{tikzpicture}
   \hspace{7mm}
\begin{tikzpicture}[baseline=0mm]
\coordinate(r)at(0:2); \coordinate(ru)at(45:2); \coordinate(u)at(90:2); \coordinate(lu)at(135:2);
\coordinate(l)at(180:2); \coordinate(rd)at(-45:2); \coordinate(d)at(-90:2); \coordinate(ld)at(-135:2);
\coordinate(rc)at(0:0.5); \coordinate(lc)at(180:0.5);
\draw(0,0)circle(2);
\draw[dotted](u)--(l)--(d)--(rc)--(u)--(r);
\draw[blue](-150:2)..controls(-90:1.5)and(1,-1)..(0:1); \draw[blue](180:1)..controls(-1,1)and(1,1)..(0:1);
\draw[blue](180:1)..controls(-1,-1)and(-45:2)..(30:2);
\fill(l)circle(1mm); \fill(r)circle(1mm); \fill(u)circle(1mm); \fill(d)circle(1mm); \fill(rc)circle(1mm);
\draw[fill=white](lu)circle(1mm); \draw[fill=white](ru)circle(1mm); 
\draw[fill=white](ld)circle(1mm); \draw[fill=white](rd)circle(1mm); \draw[fill=white](lc)circle(1mm);
\node[blue]at($(l)!0.59!(d)$){$\ast$}; \node[blue]at($(rc)!0.54!(d)$){$\ast$}; \node[blue]at($(rc)!0.36!(u)$){$\ast$};
 \node[blue]at($(rc)!0.38!(d)$){$\ast$}; \node[blue]at($(u)!0.8!(r)$){$\ast$};
\node[blue]at(-145:1.5){$\gamma_0$};\node[blue]at(-100:1.4){$\gamma_1$};\node[blue]at(50:1.1){$\gamma_2$};
\node[blue]at(110:1){$\gamma_3$};\node[blue]at(-20:1.5){$\gamma_4$};\node[blue]at(30:1.7){$\gamma_5$};
\end{tikzpicture}
\]

\noindent Then the corresponding two-term string complex $T_{\gamma}$ is not presilting. In fact, there is a nonzero quasi-graph map representative in $\Hom_{\Kb(\proj A(D))}(T_{\gamma},T_{\gamma}[1])$ induced by the form
\[
\begin{tikzpicture}[baseline=0mm,scale=1]
\node(l1)at(-1,0){$P_2$}; \node(l2)at(-3,-0.7){$P_3$}; \node(l3)at(-1,-1.4){$P_4$};
\node(l4)at(-3,-2.1){$P_3$}; \node(l5)at(-1,-2.8){$P_5$};
\draw[->](l2)--node[fill=white,inner sep=2]{$\sigma(\gamma_1)$}(l1); 
\draw[->](l2)--node[fill=white,inner sep=2]{$\sigma(\gamma_2)$}(l3);
\draw[->](l4)--node[fill=white,inner sep=2]{$\sigma(\gamma_3)$}(l3); 
\draw[->](l4)--node[fill=white,inner sep=2]{$\sigma(\gamma_4)$}(l5);
\node(r1)at(3,-1.4){$P_2$}; \node(r2)at(1,-2.1){$P_3$}; \node(r3)at(3,-2.8){$P_4$};
\node(r4)at(1,-3.5){$P_3$}; \node(r5)at(3,-4.2){$P_5$};
\draw[->](r2)--node[fill=white,inner sep=2]{$\sigma(\gamma_1)$}(r1); 
\draw[->](r2)--node[fill=white,inner sep=2]{$\sigma(\gamma_2)$}(r3);
\draw[->](r4)--node[fill=white,inner sep=2]{$\sigma(\gamma_3)$}(r3); 
\draw[->](r4)--node[fill=white,inner sep=2]{$\sigma(\gamma_4)$}(r5);
\draw[blue,double](l4)--(r2);
\node at(-4,-1.4){$T_{\gamma}$}; \node at(4,-2.8){$T_{\gamma}$,};
\end{tikzpicture}
\]
where $\sigma(\gamma_1)$ (resp., $\sigma(\gamma_2)$, $\sigma(\gamma_3)$, $\sigma(\gamma_4)$) is the short map in $\proj A(D)$ induced by a short path $b$ (resp., $f$, $dab$, $ef$) in $Q(D)$. 
There is no short map $p$ from $P_4$ to $P_2$ (resp., from $P_5$ to $P_4$) such that $\sigma(\gamma_1)=p\sigma(\gamma_3)$ (resp., $\sigma(\gamma_2)=p\sigma(\gamma_4)$). 
Therefore, $T_{\gamma}$ is not presilting.

(ii) We consider two $D$-laminates $\delta$ and $\delta'$ as follows:
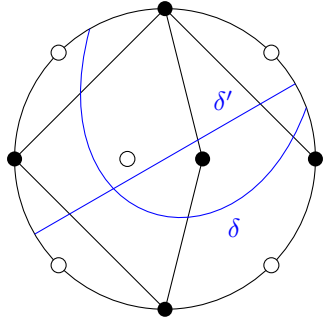
\[
\begin{tikzpicture}[baseline=0mm]
\coordinate(r)at(0:2); \coordinate(ru)at(45:2); \coordinate(u)at(90:2); \coordinate(lu)at(135:2);
\coordinate(l)at(180:2); \coordinate(rd)at(-45:2); \coordinate(d)at(-90:2); \coordinate(ld)at(-135:2);
\coordinate(rc)at(0:0.5); \coordinate(lc)at(180:0.5);
\node[blue]at(-45:1.3){$\delta$}; \node[blue]at(45:1.1){$\delta'$};
\draw(0,0)circle(2);
\draw(u)--(l)--(d)--(rc)--(u)--(r);
\draw[blue](120:2)..controls(-150:2)and(-60:2)..(20:2);
\draw[blue](-150:2)--(30:2);
\fill(l)circle(1mm); \fill(r)circle(1mm); \fill(u)circle(1mm); \fill(d)circle(1mm); \fill(rc)circle(1mm);
\draw[fill=white](lu)circle(1mm); \draw[fill=white](ru)circle(1mm); 
\draw[fill=white](ld)circle(1mm); \draw[fill=white](rd)circle(1mm); \draw[fill=white](lc)circle(1mm);
\end{tikzpicture}
\]
Then $\delta'$ is a positive position for $\delta$, but $\delta$ is not a positive position for $\delta'$. We observe whether $T_{\delta}\oplus T_{\delta'}$ is presilting. It is easy to see that $T_{\delta}$ and $T_{\delta'}$ are presilting respectively. In addition, we have $\Hom_{\Kb(\proj A)}(T_{\delta'}, T_{\delta}[1])=0$. However, there is a nonzero singleton single map from $T_{\delta}$ to $T_{\delta'}[1]$ induced by a short path $b$, as (d) in Definition \ref{def singleton single map}. Thus $T_{\delta} \oplus T_{\delta'}$ is not presilting.

(2)  We consider the $\circ\bullet$-marked surface $(S,M)$ and the $\bullet$-dissection $D=\{1,2\}$ in Section \ref{Secexample}(3). Then the associated quiver $Q(D)$ and the ideal $I(D)$ are given by
\[
Q(D)=
\begin{tikzpicture}[baseline=0mm,scale=1]
\coordinate(b)at(-0.5,0.5); \coordinate(w)at(0.5,-0.5);
\draw(-1.5,-1.5) rectangle (1.5,1.5);
\draw[fill=black](b)circle(1.0mm);
\draw[fill=white](w)circle(1.0mm);
\draw[dotted](-1.5,0.5)--node(l1)[fill=white,inner sep=1]{$1$}(b)--node(r1)[fill=white,inner sep=1]{$1$}(1.5,0.5); 
\draw[dotted](-0.5,-1.5)--node(d2)[fill=white,inner sep=1]{$2$}(b)--node(u2)[fill=white,inner sep=1]{$2$}(-0.5,1.5); 
\node(ll) at(-1.5,0) {\rotatebox{90}{$>>$}};
\node(rr) at(1.5,0) {\rotatebox{90}{$>>$}};
\node(uu) at(0,1.5) {\rotatebox{0}{$>$}};
\node(dd) at(0,-1.5) {\rotatebox{0}{$>$}};
\draw[->](r1)--node[below right=-1mm]{$a_1$}(d2); \draw[->](d2)--node[below left=-1mm]{$b_1$}(l1); 
\draw[->](l1)--node[above left=0mm]{$a_2$}(u2); \draw[->](u2)--node[above right=-1mm]{$b_2$}(r1);
\end{tikzpicture}\ 
\text{ and }\ I(D)=\langle a_1b_1,b_1a_2,a_2b_2,b_2a_1\rangle.
\]

In Section \ref{Secexample}(3), we gave the complete lists of $D$-laminates and complete $D$-laminations. For $i\in\mathbb{Z}_{>0}$ and the non-closed $D$-laminate $\gamma_i$ as in Section \ref{Secexample}(3), the corresponding two-term string complex $T_{\gamma_i}$ is given by
\[
\begin{tikzpicture}[baseline=0mm,scale=1]
\node(l1)at(-1,0){$P_1$}; \node(l2)at(-3,-0.7){$P_2$}; \node(l3)at(-1,-1.4){$P_1$};
\node(l4)at(-3,-2.1){$P_2$}; \node(l5)at(-1,-2.8){$P_1$,};
\draw[->](l2)--node[fill=white,inner sep=2]{$\sigma(a_1)$}(l1); 
\draw[->](l2)--node[fill=white,inner sep=2]{$\sigma(a_2)$}(l3);
\draw[->](l4)--node[fill=white,inner sep=2]{$\sigma(a_2)$}(l5);
\node at(-2,-1.7){$\vdots$};
\end{tikzpicture}
\]
where $\sigma(a_k)$ is a short map induced by $a_k$ and $P_1$ only appears $i$ times in degree $0$ (resp., $P_2$ only appears $i-1$ times in degree $-1$). For $i, j \in \mathbb{Z}_{>0}$, it is easy to show that all nonzero maps between $T_{\gamma_i}$ and $T_{\gamma_j}[1]$ are quasi-graph map representatives induced by the form
\[
\begin{tikzpicture}[baseline=0mm,scale=1]
\node(l0)at(-3,1.9){$P_2$}; \node(l1)at(-1,1.2){$P_1$}; \node(l2)at(-3,0.5){$P_2$};
\node(l4)at(-3,-0.5){$P_2$}; \node(l5)at(-1,-1.2){$P_1$}; \node(l6)at(-3,-1.9){$P_2$};
\draw[->](l0)--node[fill=white,inner sep=2]{$\sigma(a_2)$}(l1); 
\draw[->](l2)--node[fill=white,inner sep=2]{$\sigma(a_1)$}(l1); 
\draw[->](l4)--node[fill=white,inner sep=2]{$\sigma(a_2)$}(l5);
\draw[->](l6)--node[fill=white,inner sep=2]{$\sigma(a_1)$}(l5);
\node(r1)at(3,1.2){$P_1$}; \node(r2)at(1,0.5){$P_2$};
\node(r4)at(1,-0.5){$P_2$}; \node(r5)at(3,-1.2){$P_1$};
\draw[->](r2)--node[fill=white,inner sep=2]{$\sigma(a_1)$}(r1); 
\draw[->](r4)--node[fill=white,inner sep=2]{$\sigma(a_2)$}(r5);
\draw[blue,double](l1)--(r1)(l2)--(r2)(l4)--(r4)(l5)--(r5);
\node at(-4,0){$T_{\gamma_i}$}; \node at(4,0){$T_{\gamma_j}$,};
\node at(-2,0.1){$\vdots$}; \node at(2,0.1){$\vdots$};
\node at(-2,2.3){$\vdots$}; \node at(-2,-2.1){$\vdots$};
\end{tikzpicture}
\]
and such a map exists if and only if $i>j+1$. Therefore, $T_{\gamma_i}\oplus T_{\gamma_j}$ is two-term silting for $j=i, i\pm 1$; otherwise it is not.


\medskip\noindent{\bf Acknowledgements}.
The authors would like to thank their supervisor Osamu Iyama for his guidance and helpful comments. The authors would also like to thank Aaron Chan and Pierre-Guy Plamondon for valuable discussion. They are Research Fellows of Society for the Promotion of Science (JSPS).
T. Aoki is supported by JSPS KAKENHI Grant Number JP19J11408.
T. Yurikusa is supported by JSPS KAKENHI Grant Number JP17J04270.

\medskip\noindent{\bf Data availability}.
Data sharing not applicable to this article as no datasets were generated or analysed during the current study.

\bibliographystyle{alpha} 
\bibliography{CompleteSBalgebrasaregtame}

\end{document}